\documentclass{amsart}
\usepackage[utf8]{inputenc}
\usepackage{graphicx,amssymb,amsmath,amsthm,wrapfig}
\usepackage{caption}
\usepackage{subcaption}
\usepackage{float}
\usepackage{indentfirst}
\usepackage{cancel}
\usepackage{lipsum}
\usepackage{epstopdf}
\usepackage{epsfig}
\usepackage{enumerate,color}
\usepackage{xcolor}   
\usepackage[left=1in,right=1in,top=1in,bottom=.8in]{geometry}
\usepackage[overload]{empheq}

\newcommand{\veps}{\varepsilon}
\newcommand{\ve}{\varepsilon}

\newcommand{\hu}{w}
\newcommand{\hU}{U}
\newcommand{\hP}{P}
\newcommand{\hp}{\phi}

\newcommand{\RR}{\mathbb{R}}

\newcommand{\UUb}{\overline{U}_\beta}
\newcommand{\Vb}{V}
\newcommand{\Wb}{W_\beta} 
\newcommand{\UU}{\overline{U}}
\newcommand{\betav}{\beta_v}
\newcommand{\deltaz}{\eta_0}	
\newcommand{\Hv}{H_v}
\newcommand{\mo}{d_0}
\newcommand{\wt}{\widetilde}
\newcommand{\bb}{\bold{b}}
\newcommand{\aaa}{\bold{a}}

\numberwithin{equation}{section}
\usepackage{soul}
\everymath{\displaystyle}
\newsavebox{\bigimage}

\makeatletter

\renewcommand\subsubsection{\@secnumfont}{\bfseries}%
\renewcommand\subsubsection{\@startsection{subsubsection}{3}
	\z@{.5\linespacing\@plus.7\linespacing}{-.5em}%
	{\normalfont\bfseries}}

\makeatother

\theoremstyle{plain}
\newtheorem{theorem}{Theorem}[section]
\newtheorem{lemma}[theorem]{Lemma}
\newtheorem{proposition}[theorem]{Proposition}

\theoremstyle{definition}

\newtheorem*{defi*}{Definition}
\theoremstyle{remark}
\newtheorem{remark}{Remark}

\theoremstyle{remark}

\usepackage{etoolbox}
\makeatletter
\let\alignts@preamble\align@preamble
\patchcmd{\alignts@preamble}{\displaystyle}{\textstyle}{}{}
\patchcmd{\alignts@preamble}{\displaystyle}{\textstyle}{}{}

\def\alignts{\let\align@preamble\alignts@preamble\start@align\@ne\st@rredfalse\m@ne}

\makeatother

\allowdisplaybreaks

\title[Formation of singularities for the Camassa-Holm equation]{Sharp regularity of gradient blow-up solutions in the Camassa-Holm equation}

 \author[Y. Kim]{Yunjoo Kim}
\address[YK]{Department of Mathematical Sciences, Ulsan National Institute of Science and Technology, Ulsan, 44919, Korea}
\email{gomuli3@unist.ac.kr}

\author[B. Kwon]{Bongsuk Kwon}
\address[BK]{Department of Mathematical Sciences, Ulsan National Institute of Science and Technology, Ulsan, 44919, Korea}
\email{bkwon@unist.ac.kr}

\author[J. Yoon]{Jeongsik Yoon}
\address[JY]{Department of Mathematical Sciences, Ulsan National Institute of Science and Technology, Ulsan, 44919, Korea}
\email{jeongsik@unist.ac.kr}

\date{\today}


\begin{document}
\begin{abstract}
We study the formation of singularities in the Camassa-Holm (CH) equation, providing a detailed description of the blow-up dynamics and identifying the precise H\"older regularity of the gradient blow-up solutions. To this end, we first construct self-similar blow-up profiles and examine their properties, including the asymptotic behavior at infinity, which determines the type of singularity. Using these profiles as a reference and employing modulation theory, we establish global pointwise estimates for the blow-up solutions in self-similar variables, thereby demonstrating the stability of the self-similar profiles we construct.
Our results indicate that the solutions, evolving from smooth initial data within a fairly general open set, form $C^{3/5}$ cusps as the first singularity in finite time.
These singularities are analogous to \emph{pre-shocks} emerging in the Burgers equation, exhibiting unbounded gradients while the solutions themselves remain bounded. 
However, the nature of the singularity differs from that of the Burgers equation, which is a cubic-root singularity, i.e., $C^{1/3}$. 
Our work provides the first proof that generic pre-shocks of the CH equation exhibit $C^{3/5}$ H\"older regularity.
It is our construction of new self-similar profiles, incorporating the precise leading-order correction to the CH equation in the blow-up regime, that enables us to identify the sharp H\"older regularity of the singularities and to capture the detailed spatial and temporal dynamics of the blow-up.
We also show that the generic singularities developed in the Hunter-Saxton (HS) equation are of the same type as those in the CH equation.

	\noindent{\it Keywords}: Camassa-Holm; Hunter-Saxton; Wave breaking; Singularities; Self-similar blow-up 
	\end{abstract}
	\maketitle 
	\tableofcontents
	
	\section{Introduction}
	We consider the Camassa-Holm (CH) equation on the real line $\mathbb{R}$: 
	\begin{equation}\label{CH_main}
		u_t + 2 \gamma u_x -u_{xxt}+3uu_x=2u_xu_{xx}+uu_{xxx}, 
	\end{equation}  
	where $u=u(x,t)$ is the unknown function representing the free surface of water or velocity in non-dimensional variable and $\gamma\ge0$ is a parameter of dispersion, related to the critical shallow water speed. 
%
The CH equation \eqref{CH_main} is an integrable, dimensionless, nonlinear partial differential equation that arises as a shallow water approximation of the Euler equations for inviscid, incompressible fluids. It is derived as a bi-Hamiltonian model describing the unidirectional propagation of shallow water waves over a horizontal bed \cite{CH, CHH}. See also \cite{FF} and \cite{OR} for derivations of \eqref{CH_main} using the recursion operator method and the tri-Hamiltonian duality approach, respectively.

There have been extensive studies on the CH equation, which exhibits rich mathematical structures (see \cite{CL, JHV}). 
Notably, the existence of peakon solutions--solitons with a sharp peak \cite{CH}--in the non-dispersive case $\gamma=0$, and the occurrence of wave-breaking phenomena for all $\gamma\ge0$, are remarkably distinct features of \eqref{CH_main} when compared to the KdV equation. 
In fact, \eqref{CH_main} is the first equation to capture both soliton-type solitary waves and breaking waves. It is known that the solitary waves of the CH equation ($\gamma=0$) are peaked \cite{CH}, and they are stable solitons \cite{AW, RDJ1999}. 
On the other hand, peakon-antipeakon interactions are known to lead to wave breaking in finite time (loss of energy) \cite{13, E}, or to waves passing through each other with energy conserved, with each continuing unscathed as a solitary wave \cite{RDJ2001, P}. We refer to \cite{B-C, 8, 10, 14, 15, F} for various studies on the CH equation and related models.

A central line of inquiry in the analysis of PDEs arising in water waves and fluid dynamics is the formation of singularities in classical solutions. Remarkably, compared to the KdV equation, \eqref{CH_main} models breaking waves, characterized by a bounded solution whose gradient becomes unbounded in finite time. 
Wave-breaking for \eqref{CH_main} has been extensively studied \cite{BC, 8, 11, 12, CS, LO, TYZ, Y, Z}, including the local-in-space blow-up \cite{B}. See also blow-up results for the related models \cite{BC1,  GLOQ}. It is known that the first singularities in solutions can occur only in the form of breaking waves,
see \cite{7, 9}. 
To elucidate this, we note that among several invariant integrals, \eqref{CH_main} possesses an invariant energy integral: 
	\begin{equation}\label{Hamiltonian}
		H(t):=\int_{\mathbb{R}} ( u^2 + u_x^2 ) (x, t)  \,dx. 
	\end{equation}
We refer to Remark~\ref{Hamiltonian-rem} in Section~\ref{Global-est} for \eqref{Hamiltonian}. From \eqref{Hamiltonian}, we see that 
$\| u (\cdot,t ) \|_{L^\infty} \lesssim \| u (\cdot,t) \|_{H^1} \equiv \sqrt{ H(t_0) }  <\infty$ for all $t<T_\ast$, where we let $T_\ast$ be the life span of smooth solution. 
Since the solution is uniformly bounded, the $C^1$ breakdown implies the occurrence of breaking waves.
This blow-up bears some resemblance to the \emph{pre-shocks} occurring in the Burgers equation. 
Until recently, for nonlinear transport-type equations, including the Burgers equation, the compressible Euler equations, the Euler-Poisson system and several models of water waves, most results concerning $C^1$ blow-up were obtained using the method of characteristics combined with energy methods. 
Due to the limitations of this approach, information about blow-up solutions is lacking, and at best, only the temporal blow-up rate along the characteristic curve can be obtained.
It has recently been shown that for these equations, when the solutions experience gradient blow-up, they generically exhibit $C^{1/3}$ regularity at the blow-up time. See \cite{BSV, BSV2} for the multi-D compressible Euler equations and \cite{BKK2, BKK} for the Euler-Poisson system. 

Then a natural question arises: Do the generic blow-up solutions for \eqref{CH_main} also possess $C^{1/3}$ regularity? 
If not, what regularity do these solutions exhibit at the blow-up time? Thanks to the invariant energy integral \eqref{Hamiltonian}, we can easily rule out the possibility of $u$ having $C^{1/3}$ regularity, as in the case of the Burgers equation. This is because, in view of
\footnote{
 Since $ [ u(\cdot, t) ]_{C^{1/2}(\Omega) } \lesssim \| u(\cdot, t) \|_{H^1(\RR)}$ for any bounded set $\Omega\subset \mathbb{R}$ holds due to Morrey's inequality, the sharp local $C^{1/3}$ blow-up implies the blow-up in the $H^1$ norm. }
{Morrey's inequality,}
 $\| u(\cdot, t) \|_{H^1(\RR)}$ becomes unbounded as $t\to T_\ast$, which contradicts the Hamiltonian invariant integral \eqref{Hamiltonian}.
From this, one can infer that the singularities, if occur, are milder than $C^{1/2}$. What exactly is this regularity? The present work is devoted to answering this question. 
Interestingly, some numerical studies \cite{  CGV, DLSS}, utilizing the spectral method for tracing complex singularities developed in \cite{SSF}, suggest that the regularity is $C^\alpha$ with $\alpha \approx 0.58$. 
In fact, our results assert that $\alpha=3/5$.  
To the best of our knowledge, this is the first proof of sharp H\"older regularity of gradient blow-up solutions of the CH equation \eqref{CH_main}, along with the detailed spatial and temporal dynamics of the blow-up. 
To prove this, we first construct the asymptotically self-similar blow-up profiles to 
\eqref{CH_main}.
	\subsection{Self-similar blow-up profiles} 
	We introduce a smooth function $\UU(y):\mathbb{R} \to \mathbb{R}$ satisfying the ODE: 
	\begin{equation}\label{Ueq-int}
		\left(1+\frac{1}{2}\UU'(y) \right)\UU'(y) +\left(\UU(y) +\frac{5}{2}y\right)\UU''(y) = 0, \quad y\in \mathbb{R}.
	\end{equation}
	
		 In Proposition~\ref{Profile-construct} in Section~\ref{BPA}, we prove that \eqref{Ueq-int} admits a one-parameter family of solutions, $\{ \UUb(y) : \beta>0\}$, which are odd functions whose first derivatives attain their minimum at $y=0$. These solutions satisfy
	\begin{equation}\label{U0}
	 \UUb(0) =0, \quad \UUb'(0) = -2, \quad \UUb''(0) = 0, \quad \UUb^{(3)}(0) = 256 \beta, 
	\end{equation} 
	and for $k=0,1,2$, 
		\begin{equation}\label{decay-infty-35}
		\sup_{y \in \mathbb{R}}  \left( (1+ |y|)^{k-3/5} | \UUb^{(k)} (y)| \right) \lesssim_{k, \beta} 1.
		 \end{equation} 
	Specifically, for each $\beta>0$, it holds that
	\begin{equation}\label{UU'-decay}
 |y|^{2/5} | \UUb'(y) | \to  {(50\beta)^{-1/5}} \quad \text{ as } |y|\to \infty.
      \end{equation}
      More detailed properties of $\UUb$ are presented, along with their proofs, in Proposition~\ref{Profile-construct}. 

	\subsection{Assumptions on the initial data}\label{Initial_subs}
We describe the initial conditions that lead to a finite-time $C^1$ blow-up in the solution of \eqref{CH_main}. Let the initial time be $t = -\veps$, where $\veps$ is chosen to be sufficiently small. That is, we consider the initial value problem for \eqref{CH_main} in the domain $ \mathbb{R} \times [-\veps, \infty)$, with the initial data
	\begin{equation}\label{in-H-C}
		u_0(x)=u(x,-\veps)\in H^5(\mathbb{R})\subseteq C^4(\mathbb{R}).
	\end{equation}
	We assume that $\partial_x u_0$ attains its global minimum at $x=0$ such that  
	\begin{equation}
	 \partial_xu_0(0)=-2\veps^{-1}, \quad \partial_x^2u_0(0)=0, \quad \partial_x^3u_0(0) =: k_3 \in [\veps^{-(1+\deltaz)},\veps^{-(11-\deltaz)}] \label{init_w_3-p}
	\end{equation}
for some $0<\deltaz\ll 1$, and that 
\footnote{Here $A \lesssim B $  means $A \le C B$ for some generic constant $C>0$, independent of $\veps$.}
	\begin{equation}\label{init_24-p}
	\|\partial_x u_0\|_{L^{\infty}}\leq 2\veps^{-1}, \quad 
	\|\partial_x^2u_0\|_{L^{\infty}}\lesssim   {k_3}^{1/2} \veps^{-1/2}, \quad
 \|\partial_x^3u_0\|_{L^{\infty}}\lesssim k_3, \quad
  \|\partial_x^4u_0\|_{L^{\infty}}\lesssim  k_3^{3/2}\veps^{1/2}. 
\end{equation}	
%
%
%
%
%
%
%
%
	Letting 
	\begin{equation} \label{beta}
			\beta:=\frac{ \veps^6}{256} k_3
 \in \left[\frac{\veps^{5-\deltaz}}{256},\frac{\veps^{-(5-\deltaz)}}{256}\right],
		\end{equation} 
	  we assume that 
	  	\begin{equation}\label{4.3a-p}
		\left|\varepsilon(\partial_x u_0)(x)-\overline{U}'_\beta \left(\frac{x}{\veps^{5/2}}\right)\right|\leq \min\left\{\frac{\beta(\frac{x}{\veps^{5/2}})^2}{3000(1+\beta(\frac{x}{\veps^{5/2}})^2)}, \frac{ \Theta   }{1+\beta^{1/5}(\frac{x}{\veps^{5/2}})^{2/5}}\right\} 
	\end{equation}
	for all $x\in\mathbb{R}$, where $\overline{U}_\beta$ is the smooth solution of \eqref{Ueq-int} satisfying \eqref{U0}--\eqref{UU'-decay}, and 
	$\Theta>0$ is any number such that 
		\begin{equation}
		\label{theta-val}
		50^{-1/5} ( \approx 0.457 ) <  \Theta <  \frac{6}{13}. 
	\end{equation}
	We also assume that 
	\begin{equation}\label{4.3a-p2}
		\lim_{|x|\rightarrow \infty}|x^{2/5}\partial_xu_0(x)|\leq \frac{\theta}{2\beta^{1/5}},
	\end{equation}
 where 
	\begin{equation}\label{theta} 
	\theta := \frac13 \left( \frac{6}{13} - \Theta\right).
	\end{equation}

We remark that \eqref{4.3a-p}--\eqref{4.3a-p2} represent the conditions of localization, ensuring that the singularity develops along the modulation curve where $u$ maintains its steepest gradient. 
In our analysis, which makes use of the maximum principle for transport-type equations, the spatial localization of solutions is crucial for deriving pointwise estimates that establish global stability.  
Note from \eqref{UU'-decay} and  \eqref{theta-val} that 
\[
 \lim_{|y| \to \infty}  |y|^{2/5} | \UUb'(y) | = (50 \beta)^{-1/5} < \Theta \beta^{-1/5}.
\]
In view of this, we observe that \eqref{4.3a-p} permits $u_0$ to decay sufficiently fast to $0$ as $|x| \to \infty$, while remaining consistent with \eqref{in-H-C}. Specifically, $u_0$ is allowed to have compact support.

	In what follows, $C^\alpha(\Omega)$  denotes  the H\"older space with exponent $\alpha\in(0,1]$, and $[u]_{C^{\alpha}(\Omega) }$ represents the corresponding H\"older semi-norm, defined as 
\begin{equation*}
	[u]_{C^{\alpha}(\Omega) } := \sup_{x,y\in \Omega, x\ne y} \frac{|u(x) - u(y) | }{|x-y|^{\alpha}}.
\end{equation*}

Now, we state our main theorem.
                \begin{theorem}\label{mainthm}
		There exists a constant 
		$\ve_0 = \ve_0(\gamma, \|u_0\|_{H^1}, \|u_0\|_{C^4})>0$ such that for each $\veps\in(0,\veps_0)$, if the initial data $u_0\in H^5(\mathbb{R})$ satisfies \eqref{in-H-C}--\eqref{4.3a-p2},
		then the initial value problem \eqref{CH} admits a unique smooth solution 
		$$u\in C([-\ve, T_\ast); H^5(\mathbb{R}) ) \subset C^4  \left([-\veps,T_\ast) \times \mathbb{R}\right),$$
		which blows up in the $C^1$-norm at $t=T_\ast$ for some $T_\ast<\infty$. 
		Furthermore, $u$ exhibits regularity of $C^{3/5}$ at $t=T_\ast$ and $x=x_\ast$, where $x_\ast$ is the blow-up location.\footnote{$x_\ast$ is defined in the proof of Theorem~\ref{mainthm}.} 
		More specifically, for any bounded open subset $\Omega\subset\mathbb{R}$, it holds that 
		\[ \sup_{t<T_\ast} \left[ u(\cdot, t) \right]_{C^{3/5}(\Omega)}<\infty; \]
		and for any $\alpha> 3/5$, 
		\begin{equation}\label{blow-up-result}
		\lim_{t\nearrow T_\ast} \left[ u (\cdot, t) \right]_{C^{\alpha}(\Omega)}
		\begin{cases} 
		=\infty & \text{if } x_\ast\in \Omega,
		\\
		<\infty & \text{if } x_\ast \notin \overline{\Omega}. 
		\end{cases}
		\end{equation}
		  For any open bounded set $\Omega$ containing $x_\ast$ and for $\alpha>3/5$, the temporal blow-up rate is given by 
				\begin{equation}\label{blow-up-rate}
				\left[ u(\cdot, t) \right]_{C^\alpha(\Omega)}\sim (T_\ast-t)^{-\frac{5\alpha-3}{2}}
				\end{equation} 
				for all $t$ sufficiently close to $T_\ast$.\footnote{Here, $A(t) \sim B(t) $ means that $C^{-1} B(t) \le A(t) \le C B(t)$ for some $C>0$ independent of $t$. 
				We also note that $\ve_0=\ve_0(\gamma, \|u_0\|_{H^1}, \|u_0\|_{C^4})>0$ depends only on $\gamma, \|u_0\|_{H^1}$ and $\|u_0\|_{C^4}$.}
                \end{theorem}
The proof of Theorem~\ref{mainthm} is presented in the end of Section~\ref{Global-est}. 
Theorem~\ref{mainthm} indicates that the gradient blow-up solutions, evolving from smooth initial data close to $\UU$ in $C^4$-topology, form $C^{3/5}$ cusps as the first singularity in finite time, i.e., $u(\cdot, T_\ast)\in C^{3/5}$ at the moment $T_\ast$ of the blow-up. 
%
%
%
%
%
%
%
%
It is easy to see from \eqref{blow-up-result} that $\partial_x u$ becomes unbounded, while $u$ remains bounded, i.e., $\| u (\cdot, t) \|_{L^\infty} \lesssim \| u(\cdot, t) \|_{H^1} = \| u_0 \|_{H^1} <\infty$ for all $t<T_\ast$, due to the invariant integral \eqref{Hamiltonian}. It is analogous to \emph{pre-shocks} emerging in the Burgers equation. 
However, the nature of the singularity differs from that of the Burgers equation, which is a cubic-root singularity, i.e., $C^{1/3}$. 
This is a new type of singularity that has not been observed in any partial differential equations arising from water waves or fluid dynamics. It is a remarkable feature of \eqref{CH_main}, since nonlinear transport-type equations, including the compressible Euler equations \cite{BSV, BSV2}, the Euler-Poisson system \cite{BKK2, BKK} and some perturbations of the Burgers equation \cite{OP}, are proved to exhibit the same types of singularities as the Burgers equation. 
In fact, thanks to
the invariant integral \eqref{Hamiltonian},  
one can infer that the H\"older regularity of the pre-shocks, if occurs, must be better than $C^{1/2}$, which simply excludes the possibility of $C^{1/3}$.

Note that \eqref{CH_main} can be written in the form of a hyperbolic-elliptic system with its transport part being exactly the Burgers equation: 
	\begin{subequations}\label{CH}
		\begin{align}
			& u_t + u u_x = -p_x , \label{CH_1} 
			\\ 
			& p - p_{xx} =  u^2 + \frac{1}{2} (u_x)^2 + 2 \gamma u, \label{CH_2} 
		\end{align}
	\end{subequations}  
	where $p$ is the (dimensionless) pressure or surface elevation. 
	A key observation is that the precise leading-order correction to \eqref{CH} in the blow-up regime is more akin to 
	\begin{equation}\label{H-S}
		v_{xt} + v v_{xx} + \frac12 v_x^2 =0 
	\end{equation}
 than the Burgers equation. 
It is through our construction of new self-similar blow-up profiles $\{\UUb(y) : \beta>0\}$ to \eqref{H-S}, that we are able to identify the sharp Hölder regularity of the singularities and capture the detailed blow-up dynamics. 

\subsection{Blow-up for the Hunter-Saxton equation} 
We also consider the so-called Hunter-Saxton (HS) equation \eqref{H-S}. 
It is an integrable partial differential equation arising in the theoretical study of nematic liquid crystals, see \cite{HS}.  
As discussed earlier, \eqref{H-S} turns out to be the \emph{leading}-order correction to \eqref{CH} in the blow-up regime.
On the other hand, \eqref{H-S} can be derived from the CH equation \eqref{CH_main} in the \emph{high frequency limit}, \cite{HZ}. We also refer to \cite{BC-new, BHR-new, GLW-new, YY-new, Y-new} for various studies on the HS equation, including global existence, blow-up phenomena and the regularity structure of solutions.
%
%
%
%
%
%
%
%
%

We consider the initial value problem \eqref{H-S} with the initial data $v(\cdot, -1) =: v_0(\cdot) \in H^5(\mathbb{R})$ on $\mathbb{R}\times [-1, \infty)$. 
	We assume that $\partial_x v_0$ attains its global non-degenerate minimum at $x=0$ such that  
	\begin{equation}
		\partial_x v_0(0) <0, \quad \partial_x^2 v_0(0)=0, \quad \partial_x^3 v_0(0) >0. \label{init_v}
	\end{equation}
\newcommand{\alphav}{k_v}
Letting 
\begin{equation*}
	\betav :=    \frac{ \partial_x^3 v_0(0) }{ 128 (-\partial_x v_0(0) ) }>0, \quad \alphav := - \frac{\partial_x v_0(0) }{2}>0,
\end{equation*} 
	  we assume that  
	  	\begin{equation*}
	  	\left|  \partial_x v_0 (x)- \alphav \overline{U}'_{\betav } \left( x \right)\right|\leq \alphav \min\left\{\frac{\betav {x}^2}{3000(1+\betav x^2 ) }, \frac{\Theta}{1+\betav^{1/5} x^{2/5}}\right\}
	  \end{equation*}
%
%
	for all $x\in\mathbb{R}$, where $\overline{U}_{\betav}$ is the smooth solution of \eqref{Ueq-int} satisfying \eqref{U0}--\eqref{UU'-decay}, and $\Theta>0$  
	is any number such that $\textstyle 50^{-1/5} ( \approx 0.457 ) <  \Theta <  \frac{6}{13}$ as in \eqref{theta-val}.  
	We also assume that 
		\begin{equation}\label{4.3a-v2}
		\lim_{|x|\rightarrow \infty}|x^{2/5}\partial_x v_0(x)|\leq \frac{\alphav \theta }{2\betav^{1/5}}, 
\end{equation}
where $\textstyle \theta := \frac13 \left( \frac{6}{13} - \Theta\right)>0$.

We then obtain a blow-up result similar to Theorem~\ref{mainthm} for \eqref{H-S}, without any largeness condition on the initial gradient. 
\begin{theorem}\label{thm-HS}
Let  $v_0\in H^5(\mathbb{R})$. Suppose that $\partial_x v_0(x)$ attains its global non-degenerate negative minimum at $x=0$, and that $v_0$ satisfies \eqref{init_v}--\eqref{4.3a-v2}. 
Then the initial value problem \eqref{H-S} on $\mathbb{R}\times [-1, \infty)$ with the initial data $v_0(x), x\in \mathbb{R}$, admits a unique solution $v  \in C([-1, T_\ast) ; H^5( \mathbb{R})) \subset C^4  \left([-1, T_\ast) \times \mathbb{R}\right)$, which blows up at $T_\ast := -1+ \alphav^{-1}$.  
Furthermore, there exists a blow-up location $x_\ast\in\mathbb{R}$ such that for any open bounded subsets $\Omega\ni x_\ast$ and $\overline{\Lambda} \not\owns x_\ast$ of $\mathbb{R}$, 
it holds that 
	\begin{enumerate}[(i)]
		\item $\sup_{t< T_\ast } \left[ v(\cdot, t) \right]_{C^{\alpha}(\Omega)}<\infty
		 \quad \text{for}\quad \alpha\leq 3/5$;
		\item $\lim_{t\nearrow  T_\ast } \left[ v (\cdot, t) \right]_{C^{\alpha}(\Omega)}=\infty $
		$\quad \text{for}\quad \alpha> 3/5$;
		\item for $\alpha>3/5$, the temporal blow-up rate is given by 
		\ \begin{equation*} \left[ v(\cdot, t) \right]_{C^\alpha(\Omega)}\sim ( T_\ast -t)^{-\frac{5\alpha-3}{2}}
		\end{equation*} for all $t$ sufficiently close to $T_\ast$; 
		\item $\sup_{t< T_\ast } \left[ v(\cdot, t) \right]_{C^{\alpha}(\Lambda)}<\infty$ for any $\alpha\in(0,1]$. 
	\end{enumerate} 
\end{theorem} 
 We note that, unlike in Theorem~\ref{mainthm}, which requires a large initial gradient $u_0(0)=-2\ve^{-1}$, i.e., $\ve\ll1$, Theorem~\ref{thm-HS} imposes no such requirement. 
 This difference arises from the fact that $\UUb$ are the exact self-similar blow-up profiles for \eqref{H-S}. The proof of Theorem~\ref{thm-HS} is similar to, yet simpler than, that of Theorem~\ref{mainthm}, which we outline only in Section~\ref{HS-proof}. 
 We observe that the HS equation \eqref{H-S} exhibits the scaling invariance, $v(x,t) \to  \mu \lambda^{-1} \tilde v({x /\mu}, {t/ \lambda})$ for $\lambda, \mu>0$. Using this symmetry, we can, without loss of generality, suppress $\alphav=1$, i.e., $\partial_x v_0(0) = -2$, for the initial value problem \eqref{H-S}--\eqref{4.3a-v2}. For simplicity, we outline the proof for the case $\alphav=1$ in Section~\ref{HS-proof}.

\emph{Outline of the paper:} In Section~\ref{BPA}, we construct stable self-similar blow-up profiles to HS equation \eqref{H-S} and examine their properties, including their asymptotic behavior at infinity.  
In Section~\ref{Global-est}, using the bootstrap argument, we establish global stability estimates for the blow-up solutions in the self-similar variables. Utilizing these estimates, we prove Theorem~\ref{mainthm} at the end of this section. Section~\ref{sec3-boot} is devoted to closing the bootstrap assumptions, which are used to derive the global stability estimates. In Section~\ref{HS-proof}, we present a similar blow-up result for the Hunter-Saxton equation, stated in Theorem~\ref{thm-HS}, and outline its proof. In the Appendix, we provide several estimates for the self-similar blow-up profiles we construct and other supplementary materials used throughout our analysis. 

	\section{Construction of self-similar blow-up profiles}\label{BPA}
	%
In this section, we construct the stable self-similar blow-up profiles and examine their properties, such as their asymptotic behavior at infinity.
	In fact, \eqref{Ueq-int} is the blow-up profile equation in the self-similar variables, and it can be obtained by considering a self-similar solution of the form 
	\begin{equation*}
		v(x,t) = (-t)^{3/2} \UU(y), \quad y:= \frac{x}{(-t)^{5/2}}
	\end{equation*} to \eqref{H-S}. 
	To see this, we define the self-similar variable and unknown function $\UU$ as 
	\begin{equation}\label{self-sim} y:= \frac{x}{(-t)^b},  \quad u(x,t) =: (-t)^a \UU\left(\frac{x}{(-t)^b}\right)
	\end{equation}
	for some $a,b\in \mathbb{R}$. 
	Then a necessary condition that \eqref{H-S} admits a self-similar solution of the form    \eqref{self-sim}, is that $a=b-1$, under which $\UU(y)$ solves  
	\begin{equation}\label{Ueq-alpha}
		\left(1+\frac{1}{2}\UU'\right)\UU'+\left(\UU+\ (a+1) y\right)\UU''=0, \quad y\in\mathbb{R}. 
	\end{equation}
	We also find that $a=3/2$ is the only choice in order for  \eqref{Ueq-alpha} to admit a \emph{smooth and odd} global solution for all $y\in\mathbb{R}$. 
	In view of the odd symmetry, we assume that $\UU(0) =0$. 
From now on, we consider the case where $a = 3/2$ and $b = 5/2$, i.e., the self-similar variable and the corresponding self-similar solution take the form: 
	\begin{equation}\label{self-sim-spec}
		y:= \frac{x}{(-t)^{5/2}},  \quad u(x,t) =: (-t)^{3/2} \UU\left(\frac{x}{(-t)^{5/2}}\right).
	\end{equation}
	Now, we discuss the existence of a smooth solution $\UU(y)$ of
	\begin{equation}\label{Ueq}
		\left(1+\frac{1}{2}\UU'\right)\UU'+\left(\UU+\frac{5}{2}y\right)\UU''=0 
	\end{equation}
	satisfying the growth condition 
	\begin{equation}\label{decay-infty}
		| \UU(y)|\le C |y|^\alpha, \quad |y|\ge1 
	\end{equation} 
	for some $\alpha\in[0,1]$. 
	

\begin{proposition}\label{Profile-construct}
		The ODE problem \eqref{Ueq}--\eqref{decay-infty} admits a one-parameter family of the smooth solutions $\{ \UUb(y) : \beta>0 \}$ such that 
		for each $\beta>0$, 
		$\UUb(y)$ is monotonically decreasing on $\mathbb{R}$ and is odd, i.e., $\UUb(-y) = - \UUb(y)$ for all $y\in\mathbb{R}$. 
		Furthermore, it holds that  
\begin{equation}\label{U_taylor_small}
		\overline{U}_\beta (y) = -2y+\frac{128 \beta}{3} y^3+\mathcal{O}(y^5) \quad\text{for }|y|\ll 1 ,
	\end{equation}
	and
	\begin{equation} \label{U_taylor_far}
		\overline{U}_\beta (y) = -\frac{5}{3(50\beta)^{1/5}}y^{3/5}+ o(y^{3/5}) \quad \text{for } |y|\gg 1.
	\end{equation}
In particular, it holds that 
\begin{equation}\label{UU-0}
	\UU_\beta(0) = 0, \quad \UU_\beta'(0) =-2, \quad \UU_\beta''(0) = 0, \quad\UU_\beta^{(3)} (0) = 256 \beta, 
\end{equation} 
and as $|y|\rightarrow \infty$,
\begin{equation}\label{asymp-y-infty}
	|y|^{-3/5} | \UU_\beta(y) | \to \frac{5}{3}(50\beta)^{-1/5}, \quad |y|^{2/5} | \UU_\beta'(y) | \to (50\beta)^{-1/5}, \quad |y|^{7/5} | \UU_\beta''(y) | \to \frac{2}{5}(50\beta)^{-1/5}. 
\end{equation} 
\end{proposition}
\begin{proof}
First we prove the existence of smooth solutions. Since equation \eqref{Ueq} admits odd symmetry, we consider the problem \eqref{Ueq}-\eqref{decay-infty} on the half-line $\{y \in \mathbb{R} : y \geq 0 \}$. 
For simplicity, we set $W (y) := \UU(y) + 5y/2$, then \eqref{Ueq} is reduced to the following \emph{autonomous} ODE:
\begin{equation}\label{Weq}
	\left( W'-\frac12 \right)\left(W'-\frac{5}{2}\right)+ 2 W W''=0.
\end{equation}
We seek a solution $W$ of the form $W'=V(W)$ for some function $V=V(W)$. 
Then using the relation $
\textstyle W''=V\frac{dV}{dW}$, one has 
\begin{equation*}
	\left(V-\frac{1}{2}\right) \left(V-\frac{5}{2}\right)+ 2 WV\frac{dV}{dW}=0.
\end{equation*}  
Then using the method of separation of variables, we obtain 
\begin{equation}\label{CW2}
	\beta W^2=\frac{2V-1 }{(5-2V)^5}, 
\end{equation}	
where $\beta>0$ is a constant of integration. 
We can readily check that for each $\beta>0$, \eqref{CW2} defines a smooth global function $\Vb=\Vb(W)$ for $W\in(-\infty, \infty)$ such that  
\begin{equation}\label{Vbound}
	\frac{1}{2}\leq \Vb(W)\leq \frac{5}{2}, \quad \lim_{|W|\rightarrow \infty}\Vb(W)=\frac{5}{2}, \quad \Vb(0) = \frac12.
\end{equation}
On the other hand, from \eqref{CW2}, we find that
\begin{equation*}
	\frac{d\Vb}{dW}=\frac{\beta W(5-2\Vb)^6}{8\Vb}.
\end{equation*}
By this together with \eqref{Vbound}, we see that $d\Vb/dW$ is smooth and  uniformly bounded  for all $W\in\mathbb{R}$. 
Hence, by a standard ODE theory, for each $\beta>0$, the initial value problem  $W'=V(W), W(0)=0, y\ge0$ admits a unique global smooth solution $W=\Wb(y)$ for all $y \ge 0$. Since \eqref{Weq} admits  odd symmetry,  the odd extension $W=\Wb(y)$ is a global solution to \eqref{Weq} for all $y\in\mathbb{R}$, in turn $\UUb(y):= \Wb(y) -5y/2$ is an odd global solution to \eqref{Ueq}. 	
		That is, 
		\begin{equation} \label{Ub-eq}
			\left(1+\frac{1}{2}\UUb'\right)\UUb'+\left(\UUb+\frac{5}{2}y\right)\UUb''=0. 
		\end{equation}

%

Now, we show \eqref{U_taylor_small} and \eqref{U_taylor_far}, which describe the asymptotic behaviors of $\UUb(y)$ near $y = 0$ and $y = \infty$, respectively. 
First we can check that $ \UUb(y) := \Wb(y) - {5y/2}$   
satisfies
\begin{equation}\label{UU-0-pf}
	\begin{split}
		&\UUb(0)=\Wb (0)=0,
		\\
		&\UUb'(0) = \Wb'(0) -\frac52 = -2,  
		\\
		&\UUb''(0) = \Wb''(0) = 0, 
		\\
		&\UUb^{(3)} (0) = 256 \beta. 
	\end{split}
\end{equation}
Here, we used the fact that $\Wb$ is odd so that $\Wb^{(2i)}(0)=0$, $i\in \mathbb{N}$ for the first and third equalities, and used \eqref{CW2} for the second equality. By the Taylor expansion and \eqref{CW2}, we get the last equality. 
Then using \eqref{UU-0-pf}, we obtain the desired Taylor expansion for $\UU_\beta(y)$, $|y|\ll1$. 
This proves \eqref{U_taylor_small}. 
%

Now we prove \eqref{U_taylor_far}, i.e., the estimate of $\overline{U}_\beta$ for $|y|\gg 1$. 
From 
\eqref{Vbound}, one has  
\begin{equation}\label{Wbar_neg}
	-2 \le \UUb'(y) \le 0, \quad  y \in \mathbb{R},
\end{equation}
by which together with the fact that $\UUb(0) =0$, we have crude bounds:
\begin{equation}\label{U-rough}  
	y\UUb(y) \le 0, \quad | \UUb(y) | \le 2 |y|, 
 \quad y\in\mathbb{R}.
\end{equation}

Thanks to the odd symmetry of $\UUb(y)$, we can, without loss of generality, restrict our consideration to $y \geq 0$. 
 Using \eqref{CW2}, we have  
\begin{equation}\label{3.4}
	\UUb+\frac{5}{2}y= \left(\frac{2+\overline{U}_\beta'}{16\beta (-\overline{U}_\beta')^5}\right)^{1/2}, 
\end{equation}
where the right-hand side is well-defined thanks to \eqref{Wbar_neg}. 
By substituting \eqref{3.4} into \eqref{Ub-eq}, we have 
\begin{equation}\label{U''_TEMP}
	\overline{U}_\beta''=2\sqrt{\beta}(2+\overline{U}_\beta')^{1/2}(-\overline{U}_\beta')^{7/2}.
\end{equation}
From \eqref{U''_TEMP}, we apply the method of separation of variable to have
\begin{equation}\label{U0FAR}
	-y^{2/5}\overline{U}_\beta'=(30\sqrt{\beta})^{-2/5}(2+\overline{U}_\beta')^{1/5}(2(\overline{U}_\beta')^2-2\overline{U}_\beta'+3)^{2/5}. 
\end{equation}
Then from \eqref{3.4} and \eqref{U0FAR}, we have 
\begin{equation*}
	\begin{split}
		\overline{U}_\beta&=-\frac{5}{2}y+\frac{(2+\overline{U}_\beta')^{1/2}}{4\sqrt{\beta}(-\overline{U}_\beta')^{5/2}}=\frac{-10\sqrt{\beta}(-y^{2/5}\overline{U}_\beta')^{5/2}+(2+\overline{U}_\beta')^{1/2}}{4\sqrt{\beta}(-\overline{U}_\beta')^{5/2}}
		=-\frac{1}{6\sqrt{\beta}}\frac{(2+\overline{U}_\beta')^{1/2}(1-\overline{U}_\beta')}{(-\overline{U}_\beta')^{3/2}}.
	\end{split}
\end{equation*}
By this and \eqref{U0FAR}, we obtain
\begin{equation}\label{Ubar35}
	\begin{split}
		\frac{\overline{U}_\beta}{y^{3/5}}&=-\frac{1}{6\sqrt{\beta}}\frac{(2+\overline{U}_\beta')^{1/2}(1-\overline{U}_\beta')}{(-y^{2/5}\overline{U}_\beta')^{3/2}}
		=-\frac{(30\sqrt{\beta})^{3/5}}{6\sqrt{\beta}}\frac{(2+\overline{U}_\beta')^{1/5}(1-\overline{U}_\beta')}{(2(\overline{U}_\beta')^2-2\overline{U}_\beta' + 3)^{3/5}}.
	\end{split}
\end{equation}
Similarly, by \eqref{U''_TEMP} and \eqref{U0FAR}, we get
\begin{equation}\label{Ubar75}
	y^{7/5}\overline{U}_\beta''=\frac{1}{15(30\sqrt{\beta})^{2/5}}(2+\overline{U}_\beta')^{6/5}(2(\overline{U}_\beta')^2-2\overline{U}_\beta'+3)^{7/5}.
\end{equation}
From \eqref{Vbound} and the fact that $\textstyle \UUb'(y) = W'(y) -\frac{5}{2} = V(W(y)) - \frac{5}{2}$, we note that $\textstyle\lim_{|y|\rightarrow \infty}|\overline{U}_\beta'(y)|=0$. 
Using this for  \eqref{U0FAR}, \eqref{Ubar35} and \eqref{Ubar75}, we obtain \eqref{asymp-y-infty}, consequently \eqref{U_taylor_far}.
  This completes the proof. 
\end{proof}



Next, we present some inequalities involving $\UUb$ with $\beta=1$ that are used in the proofs in Section~\ref{sec3-boot}.
\begin{lemma}\label{lem-ubar}
Let $\UU(y) := \overline{U}_\beta(y)$, $\beta=1$ be the solution of the ODE problem \eqref{Ueq}--\eqref{decay-infty}. Then it holds that  
\begin{subequations}
	\begin{align}
		&2+\overline{U}'(y) - \frac{6 y^2}{5(1+y^2)}\geq 0, \quad y\in\mathbb{R},  \label{num_4}
		\\
		&1+\overline{U}'(y) + \frac{2}{1+y^2} \left(\frac{5}{2}+\frac{\overline U(y) }{y}\right)\geq \frac{y^2}{5(1+y^2)}, \quad y\in\mathbb{R}, \label{num_2}
		\\
		&\frac{7}{2}+2\overline{U}'(y) + \frac{1}{1+y^2}\left(\frac{5}{2}+\frac{\overline{U}(y) }{y}\right) \geq \frac{19y^2}{10(1+y^2)}, \quad y\in\mathbb{R}. \label{num_3}
		\end{align}
\end{subequations}
Moreover, there exist $\lambda>1$ and $\delta\in(0,1)$ such that 
\begin{equation}
		  \delta \left(1+\overline{U}' (y) +\frac{2}{1+y^2}\left(\frac{5}{2}+\frac{\overline{U}(y) }{y}\right)-\frac{y^2}{500(1+y^2)}\right) \geq \lambda|\overline{U}''(y)|\frac{y^2+1}{y^2}\int^{|y|}_0{\frac{y'^2}{1+y'^2 }\,dy'}, \quad y \in \mathbb{R}, \label{num_6}
\end{equation}
and 
\begin{multline}
		\frac{10}{13(1+y^{2/5})}+\overline{U}'(y) - \frac{2y^{2/5}}{5(1+y^{2/5})}\left(\frac{\overline{U}(y) }{y}+\frac{6}{13y}\int^y_0\frac{dy'}{1+(y')^{2/5}}\right) \label{num_1}
		\\
		  \geq \lambda(y^{2/5}+1)|\overline{U}''(y)|\int^{|y|}_0\frac{dy'}{1+(y')^{2/5}}, 
\quad |y|\geq m_0 
	\end{multline}
for some constant $m_0>0$.
 In fact, for $\lambda =1.0001$, we manually verify that \eqref{num_1} holds with $m_0 = 2 \cdot 10^6$, and we also confirm numerically that $m_0=93$ is sufficient.   
\end{lemma}
\begin{proof}
For the proofs, we refer to Lemma~\ref{5.1} through Lemma~\ref{5.4} in the Appendix. 
In particular, the proof of \eqref{num_1} is given in Lemma~\ref{5.4}. We also provide numerical verification of \eqref{num_1} for $m_0 =93$ and $\lambda=1.0001$ in the Appendix. 
\end{proof}
	
%
%
%
	
	\section{Global stability estimates}\label{Global-est}
	In this section, we establish the global stability estimates in the self-similar time, of which we make use to prove Theorem~\ref{mainthm}.
\subsection{Preliminaries}
First 
%
%
%
we present a local existence theorem for the initial value problem \eqref{CH_main} with the initial data \eqref{in-H-C}. 
\begin{lemma}\label{criterion} 
	For the initial data \eqref{in-H-C}, there is $T\in (-\veps,\infty)$ such that the initial value problem \eqref{CH_main} admits a unique solution 
	$ u\in C([-\veps,T);  H^5(\mathbb{R}) )$.
	Suppose further that
	\begin{equation*}
		\lim_{t\nearrow T}\|u_x(\cdot, t)\|_{L^{\infty}}<\infty.
	\end{equation*}
	Then, the solution can be continuously extended beyond $t=T$, i.e., there exists a unique solution
	\begin{equation*}
		 u \in C([-\veps,T+\delta); H^5(\mathbb{R}) )
	\end{equation*}
	for some $\delta>0$.
\end{lemma}
\begin{proof}
The local existence can be proved by a standard energy method for the hyperbolic equations.  
	We refer to \cite{Y}, for instance.
	By revisiting the energy estimates for \eqref{CH_main}, for any $k\ge2$, 
	we can readily check  
	\begin{equation*}
		\frac{d}{dt} \|u\|^2_{H^{k}} \leq C(  \|u_x\|_{L^{\infty}} )  \|u\|^2_{H^{k}}, 
	\end{equation*}
	where $C=C(\|u_x\|_{L^{\infty}})>0$ is a constant depending on $ \|u_x\|_{L^{\infty}}$. 
This yields the desired  extension criteria.
\end{proof}
Note that by the Sobolev embedding and the structure of the CH equation \eqref{CH_main}, if $u \in C([-\veps,T); H^5(\mathbb{R}) )$, then $u \in C^4([-\veps,T)\times \mathbb{R})$.

\begin{remark}\label{Hamiltonian-rem}
	It is straightforward to check that \eqref{CH_main} possesses an invariant Hamiltonian integral. Multiplying \eqref{CH_main} by $u$, we obtain
	\begin{equation*}
		\frac12 \frac{d}{dt} \int_{\mathbb{R}}  u^2\,dx-\int_{\mathbb{R}}uu_{xxt}\,dx =-  \int_{\mathbb{R}} 3 u^2u_x\,dx + \int_{\mathbb{R}}2uu_xu_{xx}\,dx+\int_{\mathbb{R}} u^2u_{xxx}\,dx, 
	\end{equation*}
	where by using $u(\cdot,t)\in H^5(\mathbb{R})$,  each term is estimated as 
	\begin{subequations}
		\begin{align*}
			&- \int_{\mathbb{R}} uu_{xxt}\,dx= \int_{\mathbb{R}} u_xu_{xt}\,dx =   \frac{1}{2} \frac{d}{dt} \int_{\mathbb{R}} u_x^2 \,dx,
			\\
			&- \int_{\mathbb{R}} 3u^2u_x\,dx= - \int_{\mathbb{R}}(u^3)_x\,dx=0,
			\\
			&\int_{\mathbb{R}}  2uu_xu_{xx}+u^2u_{xxx} \,dx
			=\int_{\mathbb{R}} (u^2u_{xx})_x\,dx=0. 
		\end{align*}
	\end{subequations} 
	This implies that   
	\begin{equation}\label{H}
		H(t):=\int_{\mathbb{R}} (u^2+u_x^2)(x,t) \,dx=H(-\veps).
	\end{equation}
\end{remark} 
	This together with a standard Sobolev inequality yields the uniform bound for $u$:  
	\begin{equation}\label{u_bound}
		 |u(x,t)|\leq C \| u (\cdot, t) \|_{H^1} = C \sqrt{ H(-\ve ) }  =: C_u.
	\end{equation}

We note that $\textstyle K(x):=\frac12 e^{-|x|}$ is the resolvent kernel for $(1-\partial_x^2)$ on $\mathbb{R}$. I.e., 
\[ ( 1- \partial_x^2 ) K(x) = \bold{\delta}_0(x),\] 
where $\delta_0$ is the Dirac-delta measure concentrated at $x=0$. 
Then, in view of \eqref{CH_2}, one can rewrite $p$ as 
		\begin{equation}\label{p_form}
			p(x,t) = \left(  K* (2\gamma u + u^2+\frac{1}{2}u_x^2) \right) (x,t), 
	\end{equation}
	where $*$ denotes convolution. 
	Hence, there is $C_p>0$ such that 
		\begin{equation}\label{up_bound}
		\begin{split}
			|  p(x,t)  | &= \left|  K* (2\gamma u + u^2+\frac{1}{2}u_x^2) (x,t)  \right| 
			\\
			&\leq \gamma \int_{\mathbb{R}}e^{-|x-z|}|u(z,t)|\,dz + \frac{1}{2}\int_{\mathbb{R}} (u^2+\frac{1}{2}u_x^2)(z, t)\,dz 
			\\
			&\leq \gamma C_u + \frac12 H(-\veps) =: C_p, 
		\end{split}
	\end{equation}
%
	where in the last inequality, we used \eqref{u_bound} and \eqref{H}. 
	%
	Differentiating \eqref{p_form} in $x$, we have 
	\begin{equation*}
		p_x(x,t) =\int^{\infty}_{-\infty}\frac{1}{2}e^{-|x-z|}\frac{x-z}{|x-z|}(2\gamma u + u^2+\frac{1}{2}u_x^2)(z, t)\,dz, 
	\end{equation*} 
	for which  similarly to \eqref{up_bound}, we have  
	\begin{equation}\label{p_x}
		|p_x(x,t)|\le  \gamma C_u + \frac12 H(-\veps) = C_p.
	\end{equation}
	%
	In what follows, we will normalize $u$ and $p$ for computational convenience by defining   
\begin{equation} \label{w-u-x-xp} 
x' := \sqrt{\beta}x, \quad  \hu(x',t) := \sqrt{\beta}u(x,t), \quad \phi(x',t) := \beta p(x,t),
\end{equation}
where $\beta$ is defined in \eqref{beta}.
Then \eqref{CH} becomes 
		\begin{equation}\label{CH-w}
			\begin{split}
				& \hu_t + \hu \hu_{x'} = -\hp_{x'}, 
				\\ 
				& \left(\frac{1}{\beta}-\partial_{x'}^2\right)\hp =  \frac{1}{\beta}\hu^2 + \frac{1}{2} (\hu_{x'})^2 +\frac{2\gamma}{\sqrt{\beta}}\hu. 
			\end{split}
		\end{equation}
		 Then, the initial conditions \eqref{in-H-C}--\eqref{4.3a-p2} can be written as 
		\begin{subequations}\label{IC-w}
			\begin{align}
				&\hu_0(x)=\hu(x,-\veps)\in H^5(\mathbb{R})\subseteq C^4(\mathbb{R}), \label{w0-1}
				\\
				&\hu_0(0)
				=:\kappa_0, \quad \partial_{x'}\hu_0(0)=-2\veps^{-1}, \quad \partial_{x'}^2\hu_0(0)=0, \quad \partial_{x'}^3\hu_0(0)=256 \veps^{-6}, \label{init_w_3-p-w}
				\\
				& \|\partial_{x'} \hu_0\|_{L^{\infty}}\leq 2\veps^{-1}, \quad
				\|\partial_{x'}^2\hu_0\|_{L^{\infty}}\lesssim \veps^{-7/2}, \quad \|\partial_{x'}^3\hu_0\|_{L^{\infty}}\lesssim \veps^{-6},  \quad \|\partial_{x'}^4\hu_0\|_{L^{\infty}}\lesssim \veps^{-17/2},   \label{init_24-p-w}
				\\
				&\left|\varepsilon(\partial_{x'} \hu_0)(x')-\overline{U}'\left(\frac{x'}{\veps^{5/2}}\right)\right|\leq \min\left\{\frac{(\frac{x'}{\veps^{5/2}})^2}{3000(1+(\frac{x'}{\veps^{5/2}})^2)}, \frac{\Theta }{1+(\frac{x'}{\veps^{5/2}})^{2/5}}\right\}, \label{4.3a-p-w}
				\\
				&\lim_{|x'|\rightarrow \infty}|x'^{2/5}\partial_{x'}\hu_0(x')|\leq \frac{\theta}{2}.\label{4.3a-p2-w}
			\end{align}
		\end{subequations} 
		Here, we used the relation that 
		$$  \UU(\frac{x' }{\veps^{5/2}}) = \beta^{1/2} \UUb(\frac{ x}{  \veps^{5/2}}), \quad x' = \beta^{1/2} x, $$
		 which comes from the fact that \eqref{Ueq} admits the symmetry $\UUb(y) = \beta^{-1/2} \UU( \beta^{1/2} y)$. 
		Thanks to this normalization, we can compare $\hu$ with $\overline{U}$, i.e., we suppress $\beta$-dependence for the reference profiles.  
		Note from \eqref{u_bound}, \eqref{up_bound} and \eqref{p_x}, that 
		\begin{equation}\label{hat_up}
			|\hu(x',t)|\leq \sqrt{\beta}C_u, \quad |\hp(x',t)|\leq \beta C_p, \quad |\hp_{x'}(x',t)|\leq \sqrt{\beta}C_p.
		\end{equation}

	\subsection{Self-similar and modulation variables}\label{modul_sec}
	We define three dynamic modulation functions $(\tau,\kappa,\xi)(t): [-\veps,\infty) \rightarrow \mathbb{R}$ satisfying a system of ODEs:
		\begin{subequations}\label{modulation-new}
			\begin{align}
				&\dot{\tau}=-\frac{(\tau(t)-t)^2}{2\beta}\left(\hp(\xi(t),t)-2\gamma\beta^{1/2}\kappa(t)- \kappa(t)^2\right),
				\\
				&\dot{\kappa}= - \frac{ 2 \left( (\tau(t)-t)^{-1}\hp_{x'}(\xi(t),t)+4(\tau(t)-t)^{-2}(\kappa(t)+\gamma\beta^{1/2}) \right) }{ \beta \partial_{x'}^3\hu(\xi(t),t)}-\hp_{x'}(\xi(t),t),
				\\
				&\dot{\xi}=\frac{ \hp_{x'}(\xi(t),t)+4(\tau(t)-t)^{-1}( \kappa(t) + \gamma\beta^{1/2}) }{ \beta \partial_{x'}^3\hu(\xi(t),t)}+\kappa(t),
			\end{align}
		\end{subequations}
	with the initial values 
	\begin{equation}\label{Modul_init-p}
		\tau(-\veps)=0, \quad \kappa(-\veps)=\kappa_0, \quad \xi(-\veps)=0, 
	\end{equation}
	where $\kappa_0 := w_0(0) = \sqrt{\beta} u_0(0)$. 
	We note by Lemma~\ref{criterion} that all the functions appearing on the right side of \eqref{modulation-new} are of $C^1$. 
	By a standard ODE theory, the initial value problem \eqref{modulation-new}-\eqref{Modul_init-p} 
	 admits a unique local smooth solution $(\tau,\kappa,\xi)$.

	Similarly to \eqref{self-sim-spec}, we introduce the self-similar variables,
	incorporated with $\tau$ and $\xi$, as follows:
	\begin{equation}\label{ys}
		y(x',t) := \frac{x'-\xi(t)}{(\tau(t)-t)^{5/2}}, \quad s(t) :=-\log\left( \tau(t) - t \right). 
	\end{equation}
	Let 
	\begin{equation*}
		s_0 := s(-\ve) = -\log \veps
	\end{equation*}
	be the corresponding initial value for $s$.
	Then, we define $T_*$ as the first time such that $\tau(t)=t$,  
	i.e., 
	\begin{equation}\label{T-star}
		T_* := \inf\{ t \in [-\ve,\infty) : \tau(t) = t\}.
	\end{equation}
	We claim that $T_\ast<\infty$ is well-defined. To see this, we first check
	by \eqref{hat_up} and \eqref{modulation-new},  that 
	\begin{equation}\label{tau}
		\begin{split}
			|\dot{\tau}|&=\frac{e^{-2s}}{2}\left|\beta^{-1}\hp(\xi(t),t)-2\gamma \beta^{-1/2}\hu(\xi(t),t) - \beta^{-1}(\hu(\xi(t),t))^2\right|\leq Ce^{-2s}, 
		\end{split}
	\end{equation}
	where $C>0$ is a constant independent of $\beta$ and $\ve>0$.
	By choosing $\ve>0$ to be sufficiently small, we have from \eqref{tau} that 
		\begin{equation}\label{tau-dot}
			|\dot{\tau} | \le C\ve^2<\frac12. 
	\end{equation}
	Let $h(t) := \tau(t) - t$, which is a continuous function. Then from \eqref{Modul_init-p} and \eqref{tau-dot}, we have $h(-\ve)=\ve>0$ and $h'(t) = \dot\tau(t) -1< -1/2$. This implies that there is $T> -\ve$ such that $h(T) =0$, i.e., $\tau(T) = T$. Hence $T_\ast<\infty$ is well-defined. 

	Later we will show that $T_*$ is the first time, at which $\partial_x u$ blows up, i.e., the blow-up time, in the proof of Theorem~\ref{mainthm}.
%
%
%
%
	%
We note, from \eqref{tau-dot}, that
	\begin{equation}\label{dottau}
		\frac{1}{1-\dot{\tau}}\leq 1+\veps.
		\end{equation} 
	Also we see that the map $s: [-\ve, T_\ast) \to [s_0, \infty)$ defined in \eqref{ys}
	 is bijective and monotonically increasing since $\textstyle \dot{s}(t) = \frac{1-\dot\tau(t)}{\tau(t) -t }>0$ for all $t\in[-\ve, T_\ast)$. 
	We define $(U,P)(y,s)$ as
	\begin{equation}\label{WZPhi-p}
		\hu(x',t)-\kappa(t) =: e^{-3s/2} U(y,s),  \quad  \hp(x',t) = : P(y,s). 
	\end{equation}
	%
	By a straightforward calculation, one can check that 
	\begin{equation}\label{dt-dx-form}
		\begin{split} 
			{\partial_t} 
			%
			%
			%
			%
			%
			%
			= (  -\dot\xi(t) e^{5s/2} + \frac52 y (1-\dot\tau) e^s ) \partial_y + (1-\dot\tau ) e^s \partial_s, 
			\qquad 
			\partial_{x'}
			= e^{5s/2} \partial_y.
		\end{split} 
	\end{equation}
	By inserting the ansatz \eqref{WZPhi-p} into the system \eqref{CH-w}, we obtain the equations for $U$ and $P$:  
	\begin{subequations}
		\begin{align}
			& \partial_s U - \frac{3}{2} U+\mathcal{V}U_y=-\frac{e^{s/2}\dot{\kappa}}{1-\dot{\tau}}-\frac{e^{3s} P_y}{1-\dot{\tau}}, \label{CH_diff_0}
			\\
			& \left( \frac{1}{\beta}e^{-5 s} - \partial_{y}^2 \right) \hP = e^{-5 s} \left(\frac{(e^{- 3s/2}\hU + \kappa)^2}{\beta}+\frac{2\gamma}{\sqrt{\beta}}(e^{-3s/2}\hU+\kappa)\right)  + \frac{e^{-3s}\hU_{y}^2}{2}, \label{CH_P}
		\end{align}
	\end{subequations}
	where
	\begin{equation}\label{V}
		\mathcal{V}:= \frac{U}{1-\dot{\tau}} + \frac{5}{2} y + \frac{e^{3s/2}(\kappa-\dot{\xi})}{1-\dot{\tau}} .
	\end{equation}
	For future reference, we present the differentiated equations here. Applying $\partial_y^n$ for $n=1,2,3,4$, to \eqref{CH_diff_0}, we have
		\begin{subequations}\label{CH_diff}
		\begin{align}
			&\left( \partial_s + 1 + \frac{\hU_{y}}{2(1-\dot{\tau})} \right)\hU_{y} + \mathcal{V} \hU_{yy} = -\frac{e^{-2s}}{1-\dot{\tau}}\left(\frac{1}{\beta}\hP- \frac{1}{\beta}(e^{-3s/2}\hU+\kappa)^2-\frac{2\gamma}{\sqrt{\beta}}(e^{-3s/2}\hU+\kappa)\right), \label{CH_diff_1} 
			\\ 
			& \left( \partial_s + \frac{7}{2} + \frac{2\hU_{y}}{1-\dot{\tau}} \right)U_{yy} + \mathcal{V} \partial_{y}^3 \hU   = -\frac{ e^{-2s}}{1-\dot{\tau}}\left(\frac{1}{\beta}\hP_{y} - \frac{2}{\beta}e^{-3s/2}(e^{-3s/2}\hU+\kappa)\hU_{y}-\frac{2\gamma}{\sqrt{\beta}}e^{-3s/2}\hU_{y}\right),  \label{CH_diff_2} 
			\\
			& \left( \partial_s + 6 + \frac{3\hU_{y}}{1-\dot{\tau}} \right)\partial_{y}^3\hU + \mathcal{V} \partial_{y}^4 \hU  \label{CH_diff_3}
			\\
			&\quad  =  -\frac{ e^{-2s} }{1-\dot{\tau}}\left(\frac{1}{\beta}\hP_{yy}- \frac{2}{\beta}e^{-3s}\hU_{y}^2-\frac{2}{\beta}e^{-3s/2}(e^{-3s/2}\hU+\kappa)\hU_{yy}-\frac{2\gamma}{\sqrt{\beta}}e^{-3s/2}\hU_{yy}\right) - \frac{2 \hU_{yy}^2}{1-\dot{\tau}},  \nonumber
			\\
			& \left( \partial_s + \frac{17}{2} + \frac{4\hU_{y}}{1-\dot{\tau}} \right) \partial_{y}^4\hU + \mathcal{V} \partial_{y}^5 \hU  \label{CH_diff_4}
			\\
			&\quad  =  -\frac{ e^{-2s} }{1-\dot{\tau}}\left(\frac{1}{\beta}\partial_{y}^3\hP- \frac{6}{\beta}e^{-3s}\hU_{y}\hU_{yy}-\frac{2}{\beta}e^{-3s/2}(e^{-3s/2}\hU+\kappa)\partial_{y}^3\hU-\frac{2\gamma}{\sqrt{\beta}}e^{-3s/2}\partial_{y}^3\hU\right)-\frac{7\hU_{yy}\partial_{y}^3\hU}{1-\dot{\tau}}.  \nonumber
		\end{align}
	\end{subequations}
	%
	%
%
	Regarding \eqref{U0},
	 we impose the constraints for $U$: 
	\begin{equation}\label{constraint-p}
		U(0,s) = 0 , \quad  U_y(0,s) = -2, \quad U_{yy}(0,s)=0. 
	\end{equation}
	Using the constraints \eqref{constraint-p} for \eqref{CH_diff_0}, \eqref{CH_diff_1} and \eqref{CH_diff_2}, we obtain 
	\begin{subequations}\label{b-mod}
		\begin{align} 
			&   \dot{\tau} = -\frac{e^{-2s}}{2 \beta }\left( \hP(0,s)- \kappa^2-2\gamma \beta^{1/2}\kappa\right),  \label{Eq_Modul1-p} 
			\\
			&  \dot{\kappa} = -2e^{-5s/2}\frac{ \hP_{y}(0,s)+4e^{-3s/2}( \kappa + \gamma \beta^{1/2})}{ \beta \partial_{y}^3\hU(0,s)}-e^{5s/2}\hP_{y}(0,s), \label{Eq_Modul0-p} 
			\\
			& \dot{\xi} = e^{-7s/2}\frac{ \hP_{y}(0,s)+4e^{-3s/2} ( \kappa+\gamma\beta^{1/2})}{ \beta \partial_{y}^3\hU(0,s)}+\kappa. \label{Eq_Modul2-p}
		\end{align}
	\end{subequations}
\begin{remark}\label{constraints-con}
We note that \eqref{b-mod} is equivalent to \eqref{modulation-new}. Hence the initial value problem \eqref{b-mod} with the initial data \eqref{Modul_init-p} admits a unique local-in-time solution $(\tau,\kappa,\xi)$.  
%
%
It is straightforward to check that  
	\begin{equation*}
		(U(0,s), U_y(0,s), U_{yy}(0,s)) \equiv (0, -2, 0)
	\end{equation*} is a solution to a set of equations \eqref{CH_diff_0}, \eqref{CH_diff_1} and \eqref{CH_diff_2} 
satisfying the condition \eqref{constraint-p} at $s=s_0$. By uniqueness, as long as the solution $(\tau, \kappa, \xi)$ to the initial problem \eqref{b-mod} with \eqref{Modul_init-p} exists, we see that  the solution $U(y,s)$ continues to satisfy the constraints \eqref{constraint-p}.
\end{remark}

We present uniform bounds for $P$ and $P_y$ that are used throughout our analysis. 
\begin{lemma}\label{P-Py-lem}
It holds that 
\begin{equation}\label{P-Py}
|P (y, s) |\leq C\beta, \quad  e^{5s/2} | P_y (y, s) | \le C\sqrt{\beta}, \quad y\in \mathbb{R}, \quad s\ge s_0,
\end{equation}
where $C=C( \| u_0 \|_{H^1} )>0$ is a uniform constant depending only on the initial data $\| u_0 \|_{H^1}$. 
\end{lemma}
\begin{proof}
Note from \eqref{WZPhi-p} and \eqref{dt-dx-form} that $\hp(x',t) = : P(y,s)$ and  
$\partial_{x'} \hp = e^{5s/2} \partial_y P$, respectively.  
Then \eqref{P-Py} immediately follows from \eqref{hat_up}. 
\end{proof}
Now we establish the global pointwise estimates by a bootstrap argument.

\subsection{Bootstrap argument}\label{boots_sec} 
We use a bootstrap argument to establish \emph{global} pointwise estimates that demonstrate the stability of the blow-up profiles. We begin with the bootstrap assumptions.  By Lemma~\ref{criterion}, there exists a unique local-in-time solution $u\in C([-\ve, T); H^5( \mathbb{R} ))$ to the initial value problem \eqref{CH_main} satisfying \eqref{in-H-C}--\eqref{4.3a-p2} for some $T>-\ve$. Correspondingly, $w\in C([-\ve, T); H^5( \mathbb{R} ))$, defined by \eqref{w-u-x-xp}, is a solution to \eqref{CH-w} with \eqref{IC-w}. 
We consider the corresponding solution $U\in C([s_0, \sigma_1); H^5( \mathbb{R} ))$, defined by \eqref{ys} and \eqref{WZPhi-p}, for some $\sigma_1>s_0$.  
We assume that  for  $s\in[s_0,\sigma_1]$ and $y\in\mathbb{R}$,  $U$  satisfies 
		\begin{subequations}\label{Boot_2-p}
			\begin{align}
				&|U_y(y,s)-\overline{U}'(y)| \leq  \frac{ y^2}{1000(1+y^2)} ,\label{EP2_1D1-p}
				\\
				&|U_y(y,s)-\overline{U}'(y)| \le  \frac{6}{13(1+y^{2/5})},   \label{Utildey_M-p}
				\\
				& |U_{yy}(y,s)| \leq  \frac{M^{1/8} |y|}{(1+y^2)^{1/2}}  , \label{EP2_1D3-p}
				\\
				&  |\partial_y ^3 U(0,s)-256| \leq 1,   \label{EP2_1D2-p} 
				\\
				& \|\partial_y ^3 U(\cdot,s)\|_{L^{\infty}} \leq M^{3/4},\label{EP2_1D5-p}
				\\
				& \|\partial_y ^4 U(\cdot,s)\|_{L^{\infty}} \leq M, \label{EP2_1D4-p}
			\end{align}
		\end{subequations}
		where $M>0$ is a sufficiently large constant to be chosen later.  
		
			\begin{remark}\label{init_rmk}{(Initial conditions for $U$)}
				The set of initial conditions \eqref{init_w_3-p-w}--\eqref{4.3a-p-w} implies that the new unknown $U$  in  the self-similar variables $(y,s)$  satisfies  
				\begin{subequations}\label{EP-W-IC}
					\begin{align}
						& |U_y(y,s_0)-\overline U'(y)| \leq \frac{y^2}{3000(1+y^2)},   \label{W-y2}
						\\
						&|U_y(y,s_0)-\overline{U}'(y)| \le \frac{\Theta}{1+y^{2/5}},\label{Utildey_init}
						\\
						&\|U_{yy}(\cdot, s_0)\|_{L^{\infty}} \leq C_2, \label{1D3}
						\\
						&|\partial_y ^3 U(0, s_0)-256| = 0,  \label{1D2} 
						\\
						&\|\partial_y ^3 U(\cdot, s_0) \|_{L^{\infty}} \leq C_3, \label{1D5}
						\\
						&\|\partial_y ^4 U( \cdot, s_0) \|_{L^{\infty}} \leq C_4 \label{1D4}
					\end{align}
				\end{subequations}
				for some positive constants $C_2, C_3$ and  $C_4$, which are independent of $\veps$.

				By the Taylor expansion with \eqref{1D2} and \eqref{1D4}, 
				\begin{equation*}
					|U_{yy}(y,s_0)|\leq|y||\partial_y^3U(0,s_0)|+\frac{y^2}{2}\|\partial_y^4U(\cdot,s_0)\|_{L^{\infty}}\leq 256|y|+\frac{C_4}{2}y^2. 
				\end{equation*}
				By this and \eqref{1D3}, we have 
				\begin{equation}\label{1D3'}
					|U_{yy}(y,s_0)|\leq \min\left\{ 256|y|+\frac{C_4}{2}y^2 ,C_2\right\}\leq \frac{M^{1/8}|y|}{4(1+y^2)^{1/2}},\quad   y \in \mathbb{R}
				\end{equation}
				for sufficiently large $M>0$. 
				Hence, from \eqref{EP-W-IC} and \eqref{1D3'}, we see that the initial data $U(y, s_0)$ satisfies the bootstrap assumptions \eqref{EP2_1D1-p}--\eqref{EP2_1D4-p}, provided $M>0$ is sufficiently large.  
			\end{remark}
		
\begin{remark}
It is straightforward to check that there exists a time-interval $[s_0,\sigma_1]$ for some $\sigma_1>s_0$, in which the solution $U(y,s)$ satisfies the bootstrap assumptions \eqref{EP2_1D1-p}--\eqref{EP2_1D4-p}. 
More specifically, as discussed earlier, thanks to Lemma~\ref{criterion}, we have $U\in C([s_0, \sigma_1); H^5( \mathbb{R} ))$, which specifically implies that $\partial_y^i U \in C([s_0, \sigma_1]; L^\infty(\mathbb{R}))$ for $0 \leq i \leq 4$. 
Hence, one can deduce that \eqref{EP2_1D3-p}--\eqref{EP2_1D4-p} hold for some $\sigma_1>s_0$. 
On the other hand, it is more subtle to show \eqref{Utildey_M-p}. We consider the equation for $W_1(y,s) := (1+y^{2/5}) U_y(y,s)$: 
			\[ \partial_s W_1 + D_1^1 W_1 + \mathcal{V} \partial_y W_1 = \tilde F, \]
			where 
			$\textstyle D_1^1 := 1 + \frac{U_y}{2(1-\dot\tau)} - \frac25 \frac{\mathcal{V}}{y^{3/5} + y}$ and $\tilde F$ is continuous and bounded. 
			We note that the term $\textstyle \frac25 \frac{\mathcal{V}}{y^{3/5} + y}$ in $D_1^1$ is singular near $y=0$. To overcome this difficulty, we decompose the region into two parts.
  Near $y=0$,  using  the fact that $\partial_y^i U \in C([s_0, \sigma_1];  L^\infty(\mathbb{R}))$ for $1\leq i\leq 4$ and  \eqref{constraint-p}, we apply the Taylor expansion near $y=0$ to verify that $W_1 \in C([s_0, \sigma_1]; L^\infty(|y|\le l))$ for some $l>0$. 
%
 For the other part, $|y| \ge l$, we see that $|D_1^1(y,s)|\lesssim 1$ and $ | \tilde F|\lesssim 1$ for all $s\in[s_0, \sigma_1]$. By integration along the characteristic curve associated with $\mathcal{V}$, 
we verify that 
$W_1\in C([s_0, \sigma_1];  L^\infty(\mathbb{R} ))$. From this together with the initial condition \eqref{4.3a-p-w}, we deduce that the desired assumption \eqref{Utildey_M-p} holds for some $\sigma_1>s_0$. 
Similarly, we can show \eqref{W-y2} and \eqref{1D3'} as well. 
\end{remark}
%
%
%
%
%
%
%
%

Now, we give several bounds for $U$ and its derivatives, obtained from the bootstrap assumptions \eqref{Boot_2-p}. 	
First, by \eqref{EP2_1D1-p} and \eqref{num_4}, we have an uniform bound of $U_y$ as  
\begin{equation} \label{Uy1-p}
| U_y (y,s)| \le | \overline{U}'(y) | + | U_y(y,s) -\overline{U}'(y) | \le 2-\frac{6y^2}{5(1+y^2)} +\frac{ y^2}{1000(1+y^2)} \le 2.
\end{equation}
Also \eqref{U_taylor_far} and \eqref{Utildey_M-p} give   
\begin{equation}\label{ax1_0}
	|U_y(y,s)| \leq |\overline{U}'(y)|+|U_y(y,s)-\overline{U}'(y)| \leq |\overline{U}'(y)|+\frac{C}{1+y^{2/5}}\leq \frac{C}{y^{2/5}}.
\end{equation}
By integrating this, we get
\begin{equation*}
	|U(y,s)|\leq C|y|^{3/5}.
\end{equation*}

\begin{remark}
	Here we present an inequality that is frequently used throughout our analysis. 
	For any $a, b > 0$, it holds that 
	\begin{equation}\label{beta-exp}
		\beta^{-a} e^{-b s } \lesssim \ve^{-a (5 - \deltaz ) } e^{-b s} =  \ve^{a \deltaz} \underbrace{ \ve^{-5 a } e^{ -5 a s} }_{\le 1}  e^{ ( 5a - b ) s } \le \ve^{ a \deltaz } e^{ ( 5a - b ) s }, 
	\end{equation}
	where $\deltaz>0$ is given as in \eqref{init_w_3-p}.
\end{remark}

Let $T_{\sigma_1}:= s^{-1}(\sigma_1)$, where $s:[-\ve,T_\ast) \to [s_0, \infty)$, defined in \eqref{ys}, is bijective and monotonically increasing.
\begin{lemma}\label{xilem} 
If $U$   satisfies  \eqref{Boot_2-p} for all $s\in[s_0,\sigma_1]$ and $y\in\mathbb{R}$, then it holds 
	that for all $s\in[s_0,\sigma_1]$, correspondingly  for all $t\in[-\ve, T_{\sigma_1}]$, 
	\begin{equation}\label{xibound}
		|\kappa(t) |\leq C\sqrt{\beta}, \qquad |\xi(t) |\leq C(\beta),
	\end{equation}
	and 
	\begin{equation} \label{temp_1-p}
		|e^{3s/2}(\kappa-\dot{\xi})| 
		\leq Ce^{-s}
	\end{equation}
	for some constant $C>0$. 
	Furthermore, we have
	\begin{equation}\label{UW_far}
		\inf_{\{|y|>1, s\in[s_0, \sigma_1]\}} \mathcal{V}(y,s) \frac{y}{|y|}  \geq \frac{1}{8},
	\end{equation}
	where $\mathcal{V}(y,s)$ is defined in \eqref{V}. 
\begin{proof}
	By \eqref{WZPhi-p} with $x'= \xi(t)$, and \eqref{hat_up}, \eqref{constraint-p}, we have $|\kappa(t)|=|\hu(\xi(t), t)|\leq C\sqrt{\beta}$. Using this, \eqref{P-Py}, \eqref{EP2_1D2-p} and \eqref{beta-exp} for \eqref{Eq_Modul2-p}, we obtain
	\begin{equation*}
		\begin{split}
			|\dot{\xi}|&=\left|\frac{e^{-7s/2}}{\partial_y^3\hU(0,s)}(\beta^{-1}\hP_{y}(0,s)+4e^{-3s/2} (\beta^{-1}\kappa+\gamma\beta^{-1/2}))+\kappa\right|
			\\
			&\leq C\left(e^{-7s/2}(\beta^{-1/2}e^{-5s/2}+\beta^{-1/2}e^{-3s/2} )+\beta^{1/2}\right)
			\\
			&\leq C\max\{\beta^{1/2},1\}, 
		\end{split}
	\end{equation*} 
	from which we have
	\begin{equation*}
		|\xi|\leq \int^t_{-\veps} |\dot{\xi}|\,dt'
		=\int^s_{s_0}|\dot{\xi}|\frac{e^{-s'}}{1-\dot{\tau}}\,ds' 
		\leq C(\beta)\int^s_{s_0}e^{-s'}ds'
		\leq C(\beta).
	\end{equation*}
	This finishes the proof of \eqref{xibound}. 	
	In addition, by \eqref{Eq_Modul2-p}, \eqref{EP2_1D2-p}, \eqref{xibound} and
	\eqref{P-Py}, we obtain
	\begin{equation*}
		\begin{split}
			|e^{3s/2}(\kappa-\dot{\xi})|&=\left|\frac{e^{-2s}(\beta^{-1}\hP_{y}(0,s)+4e^{-3s/2} (\kappa \beta^{-1}+\gamma\beta^{-1/2}))}{\partial_{y}^3\hU(0,s)}\right|
			\leq C\beta^{-1/2}e^{-7s/2}\leq Ce^{-s}, 
		\end{split}
	\end{equation*}
	where the last inequality holds due to \eqref{beta-exp}.
	%
	Lastly, using \eqref{tau}, \eqref{constraint-p}, \eqref{Uy1-p} and \eqref{temp_1-p} for \eqref{V}, one can check that  
	\begin{equation*}
		\begin{split}
			\mathcal{V}(y,s) &\geq -\frac{2y}{1-\dot{\tau}}+\frac{5y}{2}-\frac{Ce^{-s}}{1-\dot{\tau}}
			\ge  \frac{y}{4}-C\veps \quad \text{ for }y\geq 0,
		\end{split}
	\end{equation*}
	\begin{equation*}
		\begin{split}
			\mathcal{V} (y,s) &\leq -\frac{2y}{1-\dot{\tau}}+\frac{5y}{2}+\frac{Ce^{-s}}{1-\dot{\tau}}
			\leq \frac{y}{4}+C\veps \quad \text{ for }y<0.
		\end{split}
	\end{equation*}
	This verifies the assertion \eqref{UW_far}.
\end{proof}
\end{lemma}

Now we shall close  the bootstrap assumptions \eqref{Boot_2-p}. 
	\begin{proposition}\label{mainprop}
		There exist constant $M>0$ and $\veps_0>0$ such that if for each $\veps\in (0,\veps_0)$,  
		  $U$   satisfies  \eqref{Boot_2-p}  for  $s\in[s_0,\sigma_1]$ and $y\in\mathbb{R}$, then  it holds for  $s\in[s_0,\sigma_1]$ and $y\in\mathbb{R}$, that  
		\begin{subequations}\label{close-boot}
			\begin{align}
				&|U_y(y,s)-\overline U'(y)| \leq \frac{ y^2}{1500(1+y^2)},\label{1}
				\\
				&|U_y(y,s)-\overline{U}'(y)| \le \frac{\frac{6}{13}-\theta}{1+y^{2/5}}, \label{Utildey_M-str}
				\\
				& |\partial_y ^2 U(y,s)| \leq  \frac{M^{1/8}|y|}{2(1+y^2)^{1/2}}, \label{3}
				\\
				&  |\partial_y ^3 U(0,s)-256| \leq C\veps^{\deltaz},   \label{4} 
				\\
				& \|\partial_y ^3 U(\cdot,s)\|_{L^{\infty}} \leq \frac{M^{3/4}}{2},\label{5}
				\\
				& \|\partial_y ^4 U(\cdot,s)\|_{L^{\infty}} \leq \frac{M}{2}, \label{6}
			\end{align}
		\end{subequations}
		where $\theta>0$ is defined in \eqref{theta}, and $\deltaz>0$ is defined in \eqref{init_w_3-p}. 
	\end{proposition}
%
%
	\begin{proof}
Since the proof is rather lengthy, we decompose it into several lemmas in Section~\ref{sec3-boot}. The desired estimates are established through Lemmas~\ref{U30lem}--\ref{mainprop_1-p}. For the reader's convenience, we provide the following list: \eqref{1} is obtained in Lemma~\ref{Wy1_lem_p}, \eqref{Utildey_M-str} in Lemma~\ref{mainprop_1-p}, \eqref{3} in Lemma~\ref{Wy2_lem}, \eqref{4} in Lemma~\ref{U30lem}, \eqref{5} in Lemma~\ref{w-3-est}, and \eqref{6} in Lemma~\ref{Wy4_lem-p}.
	\end{proof}	
	
	\begin{remark}
		The bootstrap assumption \eqref{Boot_2-p} indicates that $U_y$ is close to $\overline{U}'$ near $y=0$ in the $C^2$ norm, ensuring the constraints \eqref{constraint-p} at $y=0$, and that $U$ remains uniformly bounded in the $C^4$ norm. In fact, the asymptotic behavior of $U$ near $y=0$ plays an important role in resolving the issues due to degeneracy of the damping term around $y=0$ in the transport-type equations in the proof of Lemma~\ref{mainprop_1-p}. 
	\end{remark}

\subsection{Global continuation} \label{global-conti}
Now, we use a standard continuation argument to show that the desired estimates \eqref{close-boot} in Proposition~\ref{mainprop} hold \emph{globally} in $s \in[s_0, \infty)$.
Let  the vector $\{V_i(s)\}_{1\leq i\leq 6}$ be
\begin{equation*}
	\begin{split}
		V_1 &:= \sup_{y\in\mathbb{R}} \left( \frac{1+y^2}{y^2} |U_y(y,s)-\overline{U}'(y)| \right) , \quad V_2 := \sup_{y\in\mathbb{R}} \left( (y^{2/5} + 1 ) | U_y(y,s) - \overline U'(y) | \right),
		\\
		V_3 & : = \sup_{y\in\mathbb{R}} \left( \frac{(1+y^2)^{1/2}}{|y|} | U_{yy}(y,s) | \right), \quad V_4  := | \partial_y^3 U(0,s) -256 |,  \quad V_5  := \| \partial_y^3 U (\cdot, s) \|_{L^\infty}, 
		\\
		V_6 & := \| \partial_y^4 U (\cdot, s) \|_{L^\infty}.
	\end{split}
\end{equation*}
Also, 
let $\{K_i\}_{1\leq i\leq 6}$ be 
\begin{equation*}
	K_1:= \frac{1}{1000}, \ \ K_2:= \frac{6}{13},\ \ K_3:= M^{1/8}, \ \ K_4:=1 , \ \  K_5:= M^{3/4}, \ \ K_6 := M. 
\end{equation*}
For any $\bb=(b_1,\cdots , b_6)\in (0,1)\times \cdots \times (0,1)$, we define the solution space $\mathfrak{X}_\bb(s)$ by
\begin{multline*}
\mathfrak{X}_\bb(s) :=\{ U \in  C([s_0,s]; H^5( \mathbb{R}) ) : V_i(s') \leq b_i K_i, \;\forall s'\in[s_0,s], \; 1\leq i\leq 6  \}.
\end{multline*}
For simplicity, let $\mathfrak{X}_{\bold{1}}(s)$ denote $\mathfrak{X}_\bb(s)$ with $b_i=1$ for $1\le i \le 6$.

Note that the set of initial conditions \eqref{init_w_3-p}--\eqref{4.3a-p2} for \eqref{CH} implies that the new unknown $U$  in  the self-similar variables $(y,s)$  satisfy  \eqref{EP2_1D1-p}--\eqref{EP2_1D4-p} at $s=s_0 = -\log \veps$. 
That is,  
\begin{equation*}
	V_i(s_0)\leq a_iK_i 
\end{equation*}
 for some $a_i\in(0,1)$, $1\leq i\leq 6$.  Let $\aaa :=(a_1, \dots, a_6)$. 
Proposition~\ref{mainprop} indicates that if $U(y,s)\in \mathfrak{X}_{\bold{1}}(\sigma_1)\cap \mathfrak{X}_\aaa(s_0)$, then $U(y,s)\in \mathfrak{X}_\bb(\sigma_1)$ for some $b_i\in (a_i,1)$, $1\leq i\leq 6$. 
 
 We will use a standard continuation argument, along with the local existence theory, Lemma~\ref{criterion}, and the initial conditions \eqref{in-H-C}--\eqref{4.3a-p2}, to show that the solution $U$ exists globally  and satisfies \eqref{close-boot} in Proposition~\ref{mainprop} for all $s\in[s_0,\infty)$. That is, $U \in \mathfrak{X}_\bb(\infty)$.

 To this end, let us define
 \begin{equation*}
 	\sigma:= \sup \{ s\ge s_0 : 
 	U \in \mathfrak{X}_\bb(s) \}.
 \end{equation*} 
 We claim that $\sigma=\infty$. 
 Suppose to the contrary that $\sigma<\infty$. Note that for the corresponding time $t=T_\sigma := s^{-1}(\sigma)$, thanks to \eqref{ys}, it is easy to see that $\sigma<\infty$ implies $\tau(t)>t$  for all $t\in[-\veps,T_\sigma]$. This means that
the map $(x,t)  \to (y,s)$ from $\mathbb{R}\times[-\veps,T_\sigma]$ to $ \mathbb{R}\times [s_0,\sigma]$ is well-defined. 
 Using \eqref{EP2_1D1-p} and \eqref{num_4}, we have 
  \[ | U_y(y,s) | \le  | U_y(y,s)-\overline{U}'(y) | + | \overline{U}'(y) | \le  \frac{y^2}{1000(1+y^2)}+ 2-\frac{6y^2}{5(1+y^2)}\le 2. \]
 This together with the fact that $u_x = e^{s} U_y$ implies that $
 \|u_x(\cdot, t)\|_{L^\infty}\le 2e^{\sigma}<\infty$ for all $t\le T_\sigma$.
 %
 Then, thanks to Lemma~\ref{criterion}, 
 there is a unique extended solution $u\in C([-\ve, T_\sigma+\delta) ; H^5(\mathbb{R} ))$ for some $\delta>0$.
This immediately implies that 
 $U$ can be extended so that $U\in \mathfrak{X}_{\bold{1}}(\sigma+\delta')$ for some $\delta'>0$. Then,   Proposition~\ref{mainprop} implies that $U\in \mathfrak{X}_\bb(\sigma+\delta')$. 
  This contradicts to the definition of $\sigma$, which concludes that  $\sigma=\infty$. Therefore, $U \in \mathfrak{X}_\bb(\infty)$, i.e.,  the solution $U$ exists \emph{globally} and satisfies \eqref{close-boot} in Proposition~\ref{mainprop} for all $s\in[s_0,\infty)$.  

Now we are ready to  prove Theorem~\ref{mainthm}.
\subsection{Proof of Theorem~\ref{mainthm}}\label{C13_subsec-p}
We split the proof into several steps.

Step 1: 
Recall the definition of $T_\ast$ in \eqref{T-star}, i.e., $T_\ast:=  \inf\{ t\ge-\ve : \tau(t)=t\}$. 
For sufficiently small $\ve>0$, thanks to  \eqref{tau} implying $|\dot{\tau}(t)|<1$,  and $\tau(-\veps)=0$, there exists 
	a number $t_\ast<\infty$ such that $\tau(t_\ast)=t_\ast$. This means that  $T_\ast<\infty$ is well-defined. 
	By a standard continuation argument in subsection~\ref{global-conti}, we have shown that  $U\in\mathfrak{X}_\bb(\infty)$. 
	This implies that the solution $u(x,t)$ exists for all $t<T_\ast$, and that   
	\begin{equation*}
		u \in C([-\ve, T_\ast); H^5(\mathbb{R})) \subset  C^4( \mathbb{R} \times [-\veps,T_*) ). 
	\end{equation*}
	(In fact, in Step 2, we prove that $T_\ast$ is the maximal existence time of the smooth solution $u$ by showing that $\| u_x(\cdot, t) \|_{L^\infty} \nearrow \infty$ as $t\nearrow T_*$.)

	Furthermore, we claim that $|T_\ast| = O( \ve^{3})$.
	By \eqref{tau} together with  existence of $\tau(t)$ up to $T_\ast$, we see that  
	\begin{equation*}
		|\dot{\tau}|\leq Ce^{-2s}, \quad t\in [-\ve, T_\ast]. 
	\end{equation*}	
	Using  this together with $\textstyle \frac{ds}{dt} = \frac{1-\dot\tau}{\tau(t) - t} = (1-\dot\tau) e^s$ from \eqref{ys}, we have 
	\[
	|\tau(t)|\leq \int^{t}_{-\veps}|\dot{\tau}(t')|\,dt' \leq C \int^{s}_{-\log\veps} e^{-3s'}\,ds' \leq  C\veps^{3}. 
	\] 
	Since  $\tau(T_\ast)=T_\ast$, we see that $|T_\ast|=O(\veps^{3})$. 
	Moreover, thanks to Lemma~\ref{xilem} and the fact that 
	$\xi, \kappa \in C^1_b([-\ve, T_\ast))$ are continuously differentiable and uniformly bounded,
	we see that $\xi(t)$ and $\kappa(t)$ converge respectively to $\xi(T_\ast)\in\mathbb{R}$ and $\kappa(T_\ast)\in\mathbb{R}$ as $t\nearrow T_\ast$.

	Step 2: We will prove that $\textstyle \sup_{t<T_\ast} \left[ u(\cdot, t) \right]_{C^{3/5}(\Omega)}<\infty$ for any bounded open set $\Omega\subset \mathbb{R}$. 
	Thanks to a standard continuation argument as discussed in subsection~\ref{global-conti}, specifically \eqref{Utildey_M-str} in Proposition~\ref{mainprop} holds globally, i.e.,  
	\begin{equation*}
		\sup_{y \in\mathbb{R}}  \left( (y^{2/5}+1)|U_y(y,s)-\overline{U}'(y)| \right) \le \frac{6}{13}-\theta \quad  \text{ for all } s\ge s_0.
	\end{equation*}
	From this and \eqref{decay-infty-35}, it holds, for all $s\ge s_0$ and $y\in\mathbb{R}$, that 
	\begin{equation}\label{Wydec_fin_2}
		|U_y (y,s) | \le | U_y (y, s) - \UU'(y) | + | \UU'(y) | \le \frac{C}{1+y^{2/5}}+|\overline{U}'(y) |\leq \frac{C}{1+y^{2/5}} 
	\end{equation} 
	for some   constant $C>0$.   
	Using this, for any $y, \wt{y}\in\mathbb{R}$ and for all $s\ge s_0$,  we have 
	\begin{equation}\label{not-zero}
		{\frac{|U(y,s)-U(\wt{y},s)|}{{|y-\wt{y}|}^{3/5}}} =  \frac{1}{{|y-\wt{y}|}^{3/5}} {\left|\int_{\wt{y}}^y U_y(\hat{y},s) d\hat{y} \right|} \le   \frac{C}{{|y-\wt{y}|}^{3/5}}  \left| \int_{\wt{y}}^{y}{(1+\hat{y}^2)^{-1/5}\,d\hat{y}} \right| \lesssim 1,
	\end{equation} 
	where the estimate is uniform in $y, \wt{y} \in \mathbb{R}$ and $s\ge s_0$. 
	Consider any two points $x'\neq \wt{x}' \in \Omega \subset \mathbb{R}$. 
	By the change of variables \eqref{ys} and \eqref{WZPhi-p} together with \eqref{not-zero}, we have 
	\begin{equation*}\label{Eq_holder}
	\frac{|u(x,t) - u(\tilde x, t ) | }{ | x - \tilde x|^\alpha} = \beta^{-1/5} 	\frac{|\hu(x',t)-\hu(\wt{x}',t)|}{{|x'-\wt{x}' |}^{3/5}}= \beta^{-1/5}  \frac{|U(y,s)-U(\wt{y},s)|}{{|y-\wt{y}|}^{3/5}} \lesssim 1
	\end{equation*}
	for all $t\in[-\ve, T_\ast)$. 
	This
	  yields that 
	\begin{equation*}
		\sup_{t<T_\ast} [u(\cdot, t)]_{C^{3/5}(\Omega) } <\infty
	\end{equation*}
	for any bounded open set $\Omega\subset\mathbb{R}$.
	It immediately follows that $\textstyle \sup_{t< T_\ast} [ u(\cdot, t)]_{C^{\alpha}(\Omega) } <\infty$ for any $\alpha \le 3/5$. 

	
	Step 3: 
	Let $x_\ast : =  \beta^{-1/2} \xi(T_\ast)$. Then for any $\alpha> 3/5$ and for any bounded open set $\Omega$, we will prove that
	$\lim_{t\nearrow T_\ast} \left[ u (\cdot, t) \right]_{C^{\alpha}(\Omega)} = \infty$ if $x_\ast \in \Omega$; and $\lim_{t\nearrow T_\ast} \left[ u (\cdot, t) \right]_{C^{\alpha}(\Omega)} < \infty$ if $x_\ast \notin \overline{\Omega}$. 

	We note from  \eqref{ys} and \eqref{WZPhi-p} that   
	\begin{equation*}
		\frac{|\hu(x',t)-\hu(\wt{x}' ,t)|}{|x'-\wt{x}' |^\alpha}=  e^{(\frac{5}{2}\alpha-\frac{3}{2})s}\frac{|U(y,s)-U(\wt{y},s)|}{|y-\wt{y}|^{\alpha}}. 
	\end{equation*} 
	Let $\wt{y}=0$ and $y \in(-\eta_1,\eta_1)\setminus\{0\}$ for some $0<\eta_1 \ll1$. 
	By the mean value theorem,
	\begin{equation*}
		\frac{|U(y,s)-U(0,s)|}{|y|^{\alpha}}=|U_y(\overline{y},s)||y|^{1-\alpha} 
	\end{equation*}
	for some $\overline{y}\in (-\eta_1,\eta_1)$.
	From \eqref{constraint-p} and \eqref{EP2_1D1-p}, we see that $|U_y(\overline{y},s)| \geq 1/2$ for $\eta_1\ll1$, which implies that $|U_y(\overline{y},s)||y|^{1-\alpha}\ge c_0$ for some $c_0>0$. 
	In view of this, for any bounded open set $\Omega$ containing $x_\ast$, 
	 which corresponds to $y=0$, we have 
	\begin{equation}\label{LB-blow}
	\begin{split} 
	[u(\cdot, t)]_{C^{\alpha}(\Omega)}
	& \geq  \beta^{-\frac12(1-\alpha)}   e^{(\frac{5}{2}\alpha-\frac{3}{2})s}  \frac{|U(y,s)-U(0,s)|}{|y|^{\alpha}} 
	\\
	&= \beta^{-\frac12(1-\alpha)}   e^{(\frac{5}{2}\alpha-\frac{3}{2})s} |U_y(\overline{y},s)||y|^{1-\alpha}  \ge  c_0   \beta^{-\frac12(1-\alpha)}  e^{(\frac{5}{2}\alpha-\frac{3}{2})s}. 
	\end{split}
	\end{equation}
	Since $\alpha>3/5$, we see that $[u(\cdot, t)]_{C^{\alpha}(\Omega)}  \to \infty$ as $t\nearrow T_\ast$, i.e., $s\to \infty$. 
	 From this, we deduce that $x_\ast =\beta^{-1/2}\xi(T_*)$ is the blow up location.

	
	On the other hand, for any bounded open set $\Omega$ such that $\overline{\Omega} \not\owns x_\ast$, we prove that $\textstyle \lim_{t\nearrow T_\ast} \left[ u (\cdot, t) \right]_{C^{\alpha}(\Omega)} <\infty$.  
	Set  $d_\ast:= d(x_\ast,\Omega) =\inf_{x\in\Omega } |x-x_\ast|>0$. One can check by \eqref{w-u-x-xp} and \eqref{ys} that for any $x\in \Omega$,  we see that the corresponding $\textstyle y = \frac{\sqrt{\beta} x - \xi(t)}{(\tau(t) - t)^{5/2}}$  satisfies $\textstyle y\geq \frac12 \sqrt{\beta}d_\ast e^{5s/2} = c_1 e^{5s/2}$ for $t$ sufficiently close to $T_\ast$.  
	This together with \eqref{Wydec_fin_2} gives 
	\begin{equation*}
		|U_y(y,s)|\leq Cy^{-2/5}=Cy^{\alpha-1}y^{{3/5}-\alpha}\leq Cy^{\alpha-1}e^{(\frac{3}{2}-\frac{5}{2}\alpha)s}
	\end{equation*}
	 for any $\alpha>3/5$. 
	Hence, we obtain
	\begin{equation*}
		\begin{split}
			[u (\cdot,t)]_{C^\alpha(\Omega)} &: = \sup_{x, \tilde x \in\Omega, x\ne  \tilde x} \frac{|u(x,t) - u(\tilde x, t ) | }{ | x - \tilde x|^\alpha}
			\\
			& = C \sup_{y,\wt{y}\in Y_t(\Omega), y\neq \wt{y}}e^{(\frac{5}{2}\alpha-\frac{3}{2})s}\frac{\left|\int^y_{\wt{y}}U_y(\hat{y},s)\,d\hat{y}\right|}{|y-\wt{y}|^{\alpha}}
			\\
			&\leq C\sup_{y,\wt{y}\in Y_t(\Omega), y\neq \wt{y}}\frac{\left|\int^y_{\wt{y}}\hat{y}^{\alpha-1}\,d \hat{y}\right|}{|y-\wt{y}|^{\alpha}}
			\le C \sup_{y,\wt{y}\in Y_t(\Omega), y\neq \wt{y}}\frac{|y|^{\alpha}-|\wt{y}|^{\alpha}}{|y-\wt{y}|^{\alpha}}<\infty,
		\end{split}
	\end{equation*} 
	where $\textstyle Y_t:  x  \to  y : = \frac{\sqrt{\beta} x - \xi(t)}{(\tau(t) - t)^{5/2}}$ is a map from $\mathbb{R}$ to $\mathbb{R}$, and in the last inequality, we used the fact that  $ | y - \wt{y} | \ge 1$ for $t$  sufficiently close to $T_\ast$ if $x \ne \tilde x$ and $x, \tilde x \in\Omega$. 
	This proves that $\textstyle \lim_{t\nearrow T_\ast} \left[ u (\cdot, t) \right]_{C^{\alpha}(\Omega)} <\infty$ for any $\overline{\Omega} \not\owns x_\ast$ and $\alpha>3/5$. 
This also implies that the blow-up location $x_\ast$ is unique. 
	
	Step 4: We shall prove the temporal blow-up rate \eqref{blow-up-rate}, i.e., 
	  for $x_\ast\in \Omega$ and $\alpha>3/5$,
				\begin{equation*} \left[ u(\cdot, t) \right]_{C^\alpha(\Omega)}\sim (T_\ast-t)^{-\frac{5\alpha-3}{2}}
				\end{equation*} 
				for all $t$ sufficiently close to $T_\ast$.
	To this end, we first note that  
	\begin{equation*}
		|U(y,s)-U(\wt{y},s)|\leq \int^y_{\wt{y}}|U_y(\hat{y},s)|\,d\hat{y} \leq C\int^y_{\wt{y}}(1+\hat{y}^2)^{(\alpha-1)/2}\,d\hat{y} \leq C|y-\wt{y}|^{\alpha},
	\end{equation*}
	where in the second inequality, we have used \eqref{Wydec_fin_2}, i.e.,  
	\begin{equation*}
		|U_y(y,s)|\leq C(1+y^2)^{-1/5}\leq C(1+y^2)^{(\alpha-1)/2}
	\end{equation*} 
	for $\alpha>3/5$.
	Then again from \eqref{ys} and \eqref{WZPhi-p}, we have for any $x', \wt{x}' \in \mathbb{R}$, 
	\begin{equation}\label{u_Calp2}
		\frac{|\hu(x',t)-\hu(\wt{x}', t )|}{|x'-\wt{x}' |^\alpha}=e^{(\frac{5}{2}\alpha-\frac{3}{2})s}\frac{|U(y,s)-U(\wt{y},s)|}{|y-\wt{y}|^{\alpha}}\leq Ce^{(\frac{5}{2}\alpha-\frac{3}{2})s}
	\end{equation}
	for some $C>0$. 
	Note from \eqref{tau} that
	\begin{equation*}
		\frac{1}{2}(T_*-t)\leq \tau(t)-t=\int^{T_*}_t 1-\dot{\tau}(t')  \, dt'\leq \frac{3}{2}(T_*-t).
	\end{equation*}
	This, together with \eqref{LB-blow} and \eqref{u_Calp2}, yields the desired blow-up rate of $u$ as 
	\begin{equation}\label{pf-br}
		[u(\cdot,t)]_{C^\alpha} = \beta^{-\frac12(1-\alpha)}   [w(\cdot,t)]_{C^\alpha} \sim(T_*-t)^{-(\frac{5}{2}\alpha-\frac{3}{2})}.
	\end{equation} 
	In particular, from \eqref{pf-br}, we see that the blow up rate of $\|\partial_xu(\cdot,t)\|_{L^{\infty}}$ is $(T_*-t)^{-1}$.
	This completes the proof of Theorem~\ref{mainthm}. 
	\qed

\section{Closure of Bootstrap assumptions}\label{sec3-boot}
In Lemma \ref{U30lem}, we show \eqref{4} closing the bootstrap assumption \eqref{EP2_1D2-p}. 
\begin{lemma}\label{U30lem}
Assume that \eqref{Boot_2-p} holds. Then for sufficiently small $\ve>0$, it holds that 
	\begin{equation}\label{str_U3-p}
		|\partial_y^3U(0,s)-256|\leq C\veps^{\deltaz},  	\quad s\in [s_0,\sigma_1],
	\end{equation}
where $\deltaz>0$ is a constant as in \eqref{init_w_3-p}. 
\end{lemma}

\begin{proof}
	Evaluating \eqref{CH_diff_3} at $y=0$ with \eqref{constraint-p}, we obtain
	\begin{equation*}
		\partial_s \partial_y^3U(0, s) + D^U_3(s) \partial_y^3U(0, s) = F^U_3(s),
	\end{equation*} 
	where
	\begin{subequations}
		\begin{align*}
			D^U_3(s) &:=  - \frac{6\dot{\tau}}{1-\dot{\tau}},
			\\
			F^U_3(s) &:= -\frac{e^{3s/2}(\kappa-\dot{\xi})}{1-\dot{\tau}}\partial_y^4U(0, s) -\frac{e^{-2s}}{1-\dot{\tau}}\left(\frac{1}{\beta}P_{yy} (0, s) -\frac{8}{\beta}e^{-3s}\right).
		\end{align*}
	\end{subequations}
	Then using \eqref{tau} and \eqref{dottau}, $D^U_3$ is estimated as  
	\begin{equation} \label{DU3}
		| D^U_3 (s) |\leq 6(1+\veps)|\dot{\tau}|\leq  C e^{-2s}. 
	\end{equation} 
	 Also, from \eqref{CH_P} and \eqref{constraint-p}, 
	 we see that 
	  $P_{yy}( 0, s) = \beta^{-1}e^{-5s} (P (0, s)-\kappa^2) -  2\gamma\beta^{-1/2} \kappa e^{-5s} - 2e^{-3s}$. 
	  Using this, we have
	\begin{equation}\label{FU3}
		\begin{split}
			|F^U_3|&=\left|\frac{e^{3s/2}(\kappa-\dot{\xi})}{1-\dot{\tau}}\partial_y^4U (0, s) +\frac{e^{-2s}}{1-\dot{\tau}}\left(\frac{e^{-5s}}{\beta^2}(P (0, s) -\kappa^2)-\frac{2\gamma\kappa}{\sqrt{\beta}}e^{-5s}-\frac{10}{\beta}e^{-3s}\right)\right|  
			\\
			&\leq Ce^{-s}+Ce^{-7s}+C\beta^{-1}e^{-5s} \leq Ce^{-\deltaz s}, 
		\end{split}
	\end{equation}
	where we have used \eqref{dottau}, \eqref{P-Py}, \eqref{EP2_1D4-p},  \eqref{xibound} and \eqref{temp_1-p} in the first inequality, and \eqref{beta} in the last inequality. 
	By \eqref{EP2_1D2-p}, \eqref{DU3} and \eqref{FU3}, we have 
	\begin{equation*}
		\left|\partial_s \partial_y^3U(0,s)\right|\leq \left|D^U_3(s) \partial_y^3 U(0,s) \right|+\left|F^U_3 (s) \right|\leq  Ce^{-\deltaz s}, 
	\end{equation*} 
	which yields 
	\begin{equation*} 
		 \left|\partial_y^3U (0, s)-256\right|  = \left|  \partial_y^3U (0, s) - \partial_y^3U (0, s_0 )  \right|  \leq \int_{s_0}^{s} \left|\partial_{s'} \partial_y^3 U (0, s' ) \right|\,ds' \leq C\int_{s_0}^s {e^{-\deltaz s'}\,ds'}\leq C\veps^{\deltaz}. 
	\end{equation*}
	This  gives \eqref{str_U3-p}.
\end{proof}

	Next we show \eqref{1}, which closed  \eqref{EP2_1D1-p}.
\begin{lemma}\label{Wy1_lem_p}
Assume that \eqref{Boot_2-p} holds. Then for sufficiently small $\ve>0$, it holds that 
	\begin{equation}
		|U_y(y,s)-\overline{U}'(y)|\leq \frac{y^2}{1500(1+y^2)} \label{Burgers_B.1-p}
	\end{equation}
	for all $y\in\mathbb{R}$ and $s\in[s_0,\sigma_1]$.
\end{lemma}
\begin{proof}
	Let $\wt{U} := U- \overline{U}$.  Then, from \eqref{Ueq-int}, we have the equation for $\wt{U}(y,s)$: 
	\begin{equation}\label{Eq_diff-p}
		\begin{split}
			\partial_s \wt{U}_y + \left(1 + \frac{\wt{U}_y + 2\overline{U}'}{2(1-\dot{\tau})}\right) \wt{U}_y + \mathcal{V} \wt{U}_{yy}  
			&=-\frac{e^{-2s}}{1-\dot{\tau}}\left(\frac{P- (e^{-3s/2}U+\kappa)^2}{\beta}-\frac{2\gamma}{\sqrt{\beta}} (e^{-3s/2}U+\kappa) \right)
			\\
			&-\frac{\dot{\tau}(\overline{U}')^2}{2(1-\dot{\tau})}-\left(\frac{\wt{U}}{1-\dot{\tau}}+\frac{\dot{\tau}\overline{U}}{1-\dot{\tau}}+\frac{e^{3s/2}(\kappa-\dot{\xi})}{1-\dot{\tau}}\right)\overline{U}'', 
		\end{split}
	\end{equation}
	where $\mathcal{V}$ is defined in \eqref{V}. 
	Defining a new weighted function
	\[Z (y,s) :=\frac{y^2+1}{y^2}\wt{U}_y(y,s),\]
	we have from \eqref{Eq_diff-p} that 
	\begin{equation*}
		\partial_s Z +D^Z (y,s) Z +\mathcal{V}(y,s) Z_y = F^Z(y,s)+\int_{\mathbb{R}} Z(y',s) K^Z(y,s;y') \,dy',
	\end{equation*}
	where
	\begin{subequations}
		\begin{align*}
			&D^Z(y,s) := 1+\frac{\wt{U}_y+2\overline{U}'}{2(1-\dot{\tau})}+\frac{2}{y(1+y^2)}\mathcal{V},
			\\
			& F^Z(y,s) := -\frac{e^{-2s}}{1-\dot{\tau}}\frac{y^2+1}{y^2} \left(\frac{P- (e^{-3s/2}U+\kappa)^2}{\beta}-\frac{2\gamma}{\sqrt{\beta}} (e^{-3s/2}U+\kappa) \right)
			-\frac{y^2+1}{y^2}\frac{\dot{\tau}(\overline{U}')^2}{2(1-\dot{\tau})}
			\\
			& \hspace{6cm} -\frac{y^2+1}{y^2}\left(\frac{\dot{\tau}\overline{U}}{1-\dot{\tau}}+\frac{e^{3s/2}(\kappa-\dot{\xi})}{1-\dot{\tau}}\right)\overline{U}'',
			\\
			& K^Z(y,s;y') := -\frac{1}{1-\dot{\tau}}\frac{y^2+1}{y^2}\overline{U}''(y)\mathbb{I}_{[0,y]}(y')\frac{y'^2}{1+y'^2}.
		\end{align*}
	\end{subequations}
	
	First, let us estimate $Z$ near $y=0$.
	By the Taylor expansion, 
	\begin{equation*}
		U_y(y,s)=-2+\frac{y^2}{2}\partial_y^3U(0,s)+\frac{y^3}{6}\partial_y^4U(y',s) \quad \text{ for some } |y'|<|y|<l,
	\end{equation*} 
	where $l=1/10M>0$. 
	Using \eqref{EP2_1D4-p} and \eqref{str_U3-p}, we have
	\begin{equation*}
		\begin{split}
			| U_y(y, s) + 2 - 128 y^2 | 
			& = \left| \frac{y^2}{2}\left(\partial_y^3U( 0, s) - 256 \right)+\frac{y^3}{6}\partial_y^4U(y',s)\right|
			\leq C\veps^{ \deltaz }  y^2 + \frac{M}{6}  |y|^3  \leq y^2\left( C \veps^{ \deltaz }+ \frac{ M | l | }{ 6  } \right)
		\end{split}
	\end{equation*}
	for all $| y | \le l$. 
	By this together with \eqref{U_taylor_small},  we have for sufficiently small $l=l(M)>0$, 
	\begin{equation*}
		\begin{split}
			|U_y(y, s) -\overline{U}'(y) |&\leq y^2\left(C\veps^{\deltaz}+\frac{M|l|}{6}\right)+C\cdot y^4
			\\
			&\leq y^2\left(C\veps^{\deltaz}+\frac{M|l|}{6}+C\cdot l^2\right)\leq \frac{ y^2}{3000(1+y^2)} \quad \text{for}\quad |y|\leq l, 
		\end{split}
	\end{equation*}
	which yields 
	\begin{equation}\label{V-l-p}
		|Z(y,s)|\leq \frac{1}{3000}  \quad \text{for} \quad |y|\leq l.
	\end{equation}
	
	Now, 
	we recall \eqref{num_2}: 
	\begin{equation*}
		1+\overline{U}'+\frac{2}{y(1+y^2)}\left(\frac{5y}{2}+\overline U\right)\geq \frac{y^2}{5(1+y^2)}, \quad y\in\mathbb{R}.
	\end{equation*}
	Let us examine $D^Z$ as follows: 
	\begin{equation}\label{DV}
		\begin{split}
			D^Z &= 1+\frac{\wt{U}_y+2\overline{U}'}{2(1-\dot{\tau})}+\frac{2}{y(1+y^2)}\left(\frac{\overline{U}+\wt{U}}{1-\dot{\tau}}+\frac{5y}{2}+\frac{e^{3s/2}(\kappa-\dot{\xi})}{1-\dot{\tau}}\right)
			\\
			&\geq \left(1+\overline{U}'+\frac{2}{y(1+y^2)}\left(\frac{5y}{2}+\overline{U}\right)\right)-\left|\frac{\wt{U}_y+2\dot{\tau}\overline{U}'}{2(1-\dot{\tau})}+\frac{2}{y(1+y^2)}\left(\frac{\wt{U}+\dot{\tau}\overline{U}}{1-\dot{\tau}}+\frac{e^{3s/2}(\kappa-\dot{\xi})}{1-\dot{\tau}}\right)\right|.
		\end{split}
	\end{equation}
	Here using \eqref{tau}, \eqref{EP2_1D1-p} and $|\overline{U}'|\leq 2$ from \eqref{Wbar_neg}, we see that 
	\begin{equation}\label{DV1}
		\left|\frac{\wt{U}_y+2\dot{\tau}\overline{U}'}{2(1-\dot{\tau})}\right|\leq \frac{1+\veps}{2}\left(\frac{ y^2}{1000(1+y^2)}+O(\veps)\right). 
	\end{equation}
	 We use \eqref{tau} and  \eqref{temp_1-p} to have  
	\begin{equation}\label{DV2}
		\begin{split}
			&\left|\frac{2}{y(1+y^2)}\frac{\wt{U}}{1-\dot{\tau}}\right|+\left|\frac{2}{y(1+y^2)}\left(\frac{\dot{\tau}\overline{U}}{1-\dot{\tau}}+\frac{e^{3s/2}(\kappa-\dot{\xi})}{1-\dot{\tau}}\right)\right|
			\\
			&\qquad \leq (1+\veps)\left|\frac{2}{y(1+y^2)}\right|\left|\frac{1}{1000} \int^{y}_0\frac{y'^2}{1+y'^2}\,dy'\right|+(1+\veps)\left|\frac{2}{y(1+y^2)}\right| \left(|\dot{\tau}\overline{U}|+|e^{3s/2}(\kappa-\dot{\xi})| \right)
			\\
			&\qquad \leq \left|\frac{1}{500y(1+y^2)}\right|\left|\int^{y}_0\frac{y'^2}{1+y'^2}\,dy'\right|+O(\veps)
			\\
			&\qquad \leq \frac{y^2}{1500(1+y^2)}+O(\veps), \qquad |y|\geq l.
		\end{split}
	\end{equation}
	Here in the first inequality we have used the fact that 
	\begin{equation}\label{U-x2}
	 | \wt{U}(y,s) | \le \int_0^y |  \wt{U}_y (y', s) | dy'  \le \frac{1}{1000} \int_0^y \frac{y'^2}{1+y'^2 } dy',  
	 \end{equation}
	which follows from $\wt{U}(0,s)=0$ and \eqref{EP2_1D1-p}. 
	Hence, 
	\begin{equation}\label{damp-p}
		\begin{split}
			D^Z (y, s) &\geq \frac{ y^2}{5(1+y^2)}-\frac{y^2}{2000(1+y^2)}-\frac{y^2}{1500(1+y^2)}-O(\veps) 
			\geq \frac{l^2}{10(1+l^2)}, \qquad |y|\geq l.
		\end{split}
	\end{equation}
	%
	%
	Next,  we compare $K^Z$ with $D^Z$. 
	From \eqref{num_6}, we have 
	\begin{equation*}
		\begin{split}
			\int_\mathbb{R}{| K^Z (y,s;y')|\,dy'}&\leq\frac{|\overline{U}''(y)|}{1-\dot{\tau}}\frac{y^2+1}{y^2}\int^{|y|}_0{\frac{y'^2}{1+y'^2 }\,dy'}
			\\
			&\leq \delta \left(1+\overline{U}'+\frac{2}{y(1+y^2)}\left(\frac{5y}{2}+\overline{U}\right)-\frac{y^2}{500(1+y^2)}\right) 
			\leq   \delta D^Z(y,s), \quad |y|\geq l 
		\end{split}
	\end{equation*}
	for some constant $\delta \in (0,1)$. Here, we have used \eqref{DV}--\eqref{DV2} in the last inequality.

	On the other hand, using \eqref{tau}, we see that 
	\begin{equation*}
		\begin{split}
			&	\|F^Z (\cdot,s)\|_{L^{\infty}(|y|\geq l)}
			\\
			&\leq \left\|\frac{e^{-2s}}{1-\dot{\tau}}\frac{y^2+1}{y^2} \left(\frac{P- (e^{-3s/2}U+\kappa)^2}{\beta}-\frac{2\gamma}{\sqrt{\beta}} (e^{-3s/2}U+\kappa) \right)
			\right\|_{L^{\infty}(|y|\geq l)}
			\\
			&\quad +\left\|\frac{y^2+1}{y^2}\frac{\dot{\tau}(\overline{U}')^2}{2(1-\dot{\tau})}+\frac{y^2+1}{y^2}\left(\frac{\dot{\tau}\overline{U}}{1-\dot{\tau}}+\frac{e^{3s/2}(\kappa-\dot{\xi})}{1-\dot{\tau}}\right)\overline{U}''\right\|_{L^{\infty}(|y|\geq l)}
			\\
			&\leq Ce^{-2s}\frac{l^2+1}{l^2}\left\| \frac{P- (e^{-3s/2}U+\kappa)^2}{\beta}-\frac{2\gamma}{\sqrt{\beta}} (e^{-3s/2}U+\kappa) \right\|_{L^{\infty}(\mathbb{R})} 
			\\
			&\quad+C\frac{l^2+1}{l^2}\left(|\dot{\tau}|\|\overline{U}'\|_{L^{\infty}(\mathbb{R})}^2+\left(|\dot{\tau}|\|\overline{U}\|_{L^{\infty}(\mathbb{R})}+e^{3s/2}|\kappa-\dot{\xi}|\right)\|\overline{U}''\|_{L^{\infty}(\mathbb{R})}\right).
		\end{split}
	\end{equation*}
	We shall estimate the right-hand side term by term as follows. 
	Thanks to \eqref{hat_up}, we have
	\begin{equation*}
		e^{-2s}\left\| \frac{P- (e^{-3s/2}U+\kappa)^2}{\beta}-\frac{2\gamma}{\sqrt{\beta}} (e^{-3s/2}U+\kappa) \right\|_{L^{\infty}(\mathbb{R})}\leq Ce^{-2s}, 
	\end{equation*}
	and by \eqref{U-rough}, \eqref{tau} and the fact that $|y\overline{U}''(y)|=2(2+\overline{U}')^{1/2}(-y^{2/7}\overline{U}')^{7/2}\leq C$, which is from \eqref{U_taylor_far} and \eqref{U''_TEMP}, we get
	\begin{equation*}
		|\dot{\tau}|\|\overline{U}(y) \overline{U}''(y) \|_{L^{\infty}(\mathbb{R})}\leq Ce^{ - 2s }.
	\end{equation*}
	In addition, thanks to  
	\eqref{asymp-y-infty}, \eqref{Wbar_neg}, \eqref{tau} and \eqref{temp_1-p}, we have  
	\begin{equation*}
		|\dot{\tau}|\|\overline{U}'\|_{L^{\infty}(\mathbb{R})}^2 + e^{3s/2}|\kappa-\dot{\xi}|\|\overline{U}''\|_{L^{\infty}(\mathbb{R})} \leq Ce^{-s}.
	\end{equation*}
	Combining all, for sufficiently small $\ve>0$,  we get
	\begin{equation*}
		\|F^Z (\cdot,s)\|_{L^{\infty}(|y|\geq l)}\leq C e^{- s} \leq C\veps.
	\end{equation*} 
	On the other hand, by {\eqref{4.3a-p-w}, i.e., \eqref{Utildey_M-p} and \eqref{W-y2}}, we have 
	\begin{equation}\label{Wy_claim}
		|Z(y,s_0)|\leq \frac{1}{3000}, \qquad \limsup_{|y|\rightarrow \infty}|Z (y,s)|=0,
	\end{equation} 
	respectively. 
	Appealing to Theorem \ref{max_2} along with \eqref{V-l-p} and \eqref{damp-p}--\eqref{Wy_claim}, we conclude 
	\begin{equation*}
		\|Z (\cdot,s)\|_{L^{\infty}(\mathbb{R})}\leq \frac{1}{1500}.
	\end{equation*}
	We remark that the condition \eqref{D-cond} in Theorem~\ref{max_2} holds for sufficiently small $\ve>0$. This gives the desired estimate \eqref{Burgers_B.1-p}.	
\end{proof}

We now show \eqref{3}.
\begin{lemma}\label{Wy2_lem}
	Assume that \eqref{Boot_2-p} holds. Then for sufficiently large $M>0$ and small $\ve>0$, it holds that 
	\begin{equation}\label{Uyy-bd}
		|U_{yy} (y, s) |\leq  \frac{M^{1/8} |y|}{2(1+y^2)^{1/2}} 
	\end{equation}
	for all $y\in\mathbb{R}$ and $s\in[s_0,\sigma_1]$.
\end{lemma}

\begin{proof}
	Setting $\wt{Z}(y,s) := (y^2+1)^{1/2} y^{-1}   U_{yy}(y,s)$, 
	we obtain from \eqref{CH_diff_2} that 
	\begin{equation*}
			\partial_s\wt{Z} + D^{\wt{Z}} \wt{Z} + \mathcal{V} \wt{Z}_y=F^{\wt{Z}},
	\end{equation*}
	where $\mathcal{V}$ is given in \eqref{V}, and
	\begin{subequations}
		\begin{align*}
			D^{\wt{Z}} (y, s) &:= \frac{7}{2}+\frac{2U_y}{1-\dot{\tau}}+\frac{1}{y(1+y^2)}\left(\frac{U}{1-\dot{\tau}}+\frac{5y}{2}+\frac{e^{3s/2}(\kappa-\dot{\xi})}{1-\dot{\tau}}\right),
			\\
			F^{\wt{Z}} (y, s) &:= -\frac{e^{-2s}}{1-\dot{\tau}}\frac{(y^2+1)^{1/2}}{y}\left(\frac{P_y- 2e^{-3s/2}(e^{-3s/2}U+\kappa)U_y}{\beta}-\frac{2\gamma}{\sqrt{\beta}} e^{-3s/2}U_y \right).
		\end{align*}
	\end{subequations}

	First, in the region $\{ y : |y|\leq {l=1/M} \}$, where $l>0$ is sufficiently small, i.e., $M\gg1$,  by the Taylor expansion together with \eqref{EP2_1D2-p} and \eqref{EP2_1D4-p}, we have for some $y'\in (-| y |, | y |)$,  
\begin{equation*}
	|U_{yy}(y,s)|\leq |y||\partial_y^3U(0,s)|+\frac{y^2}{2}|\partial_y^4U(y',s)|\leq 257|y|+\frac{M}{2}y^2\leq \frac{M^{1/8} |y|}{4(1+y^2)^{1/2}}  \quad \text{for}\quad |y|\leq l =\frac1M.
\end{equation*}
This gives 
\begin{equation}\label{Uyy_local}
	\|\wt{Z}(\cdot,s)\|_{L^{\infty}(|y|\leq l)}\leq  \frac{M^{1/8}}{4}.
\end{equation}

	Next, we obtain a lower bound for $D^{\wt{Z}}$. We note by \eqref{num_3}  that 
	\begin{equation*}
		\frac{7}{2}+2\overline{U}'+\frac{1}{y(1+y^2)}\left(\overline{U}+\frac{5y}{2}\right) \geq \frac{19y^2}{10(1+y^2)}.
	\end{equation*}
	 In light of this, we split  $D^{\wt{Z}}$ into two parts as 
	\begin{equation*}
		\begin{split}
			D^{\wt{Z}}=\left(\frac{7}{2}+2\overline{U}'+\frac{1}{y(1+y^2)}\left(\overline{U}+\frac{5y}{2}\right)\right)+\left(\frac{2(\wt{U}_y+\dot{\tau}\overline{U}')}{1-\dot{\tau}}+\frac{1}{y(1+y^2)}\left(\frac{\wt{U}+\dot{\tau}\overline{U}}{1-\dot{\tau}}+\frac{e^{3s/2}(\kappa-\dot{\xi})}{1-\dot{\tau}}\right)\right). 
		\end{split}
	\end{equation*}
	Thanks to \eqref{Wbar_neg}, \eqref{tau} and \eqref{EP2_1D1-p},  
	 we have 
	\begin{equation*}
		|\wt{U}_y+\dot{\tau}\overline{U}'|\leq \frac{ y^2}{1000(1+y^2)}+C\veps^2.
	\end{equation*}
	Also, similarly as \eqref{U-x2}, using \eqref{EP2_1D1-p}, 
	\begin{equation*}
		\left|\frac{\wt{U}(y, s) }{y(1+y^2)}\right|\leq 	\frac{1}{|y|(1+y^2)}\left|\int^y_0\frac{ y'^2}{1000(1+y'^2)}\,dy'\right|\leq \frac{y^2}{3000(1+y^2)}.
	\end{equation*} 
	In addition, 
	 from  \eqref{U-rough}, \eqref{tau} and  \eqref{temp_1-p}, 
	 we get
	\begin{equation*}
		\left|\frac{1}{y(1+y^2)}\left(\frac{\dot{\tau}\overline{U}}{1-\dot{\tau}}+\frac{e^{3s/2}(\kappa-\dot{\xi})}{1-\dot{\tau}}\right)\right|
		\leq Ce^{-s}(1+|y|^{-1}).
	\end{equation*} 
	Combining all, we have 
	\begin{equation}\label{Uyy_D}
		\begin{split}
			D^{\wt{Z}}(y, s)	&\geq \frac{19y^2}{10(1+y^2)}- \frac{ 2 y^2}{1000(1+y^2)}-\frac{y^2}{3000(1+y^2)}- \frac{C  \ve }{l}
			\geq \frac{y^2}{1+y^2}\quad \text{for}\quad |y|\geq l,
		\end{split}
	\end{equation}
	where the last inequality holds true if $\ve\ll l^3$. 

	Now, we estimate $F^{\wt{Z}}$. By \eqref{hat_up}, \eqref{tau}, \eqref{P-Py}  and  \eqref{Uy1-p},
	\begin{equation}\label{Uyy_F}
		\begin{split}
			|F^{\wt{Z}} (y, s) | &\leq Ce^{-2s}(\beta^{-1}|P_y|+ 2\beta^{-1}e^{-3s/2}|e^{-3s/2}U+\kappa||U_y|+2\gamma\beta^{-1/2} e^{-3s/2}|U_y| )
			\\
			&\leq C\beta^{-1/2}e^{-7s/2}
			 \leq Ce^{-s} \quad \text{for}\quad |y|\geq l. 
		\end{split}
	\end{equation}
	Here, the last inequality holds due to \eqref{beta-exp}, namely, $\beta^{-1/2}e^{-7s/2} \lesssim  e^{-s}$. 
	In addition, as we checked in \eqref{1D3'}, we have 
%
%
%
%
%
%
%
%
	\begin{equation*}
		\|\wt{Z}(\cdot,s_0)\|_{L^{\infty}(\mathbb{R})}\leq \frac{M^{1/8}}{4}.
	\end{equation*}
%
	Finally, we claim that $\textstyle\limsup_{|y|\rightarrow \infty}|\wt{Z} (y,s ) |\leq \frac14 M^{1/8}$ for sufficiently large $M>0$. 
	Thanks to \eqref{Uyy_D} and \eqref{Uyy_F}, we have  
	\begin{equation*}
		\inf_{|y|\geq 10}D^{\wt{Z}}(y,s)\geq \frac{ 10^2}{1+ 10^2}=:\lambda_D, 
	\end{equation*}
	and
	\begin{equation*}
		\|F^{\wt{Z}}(y,s)\|_{L^{\infty}(|y|\geq 10)}\leq Ce^{-s}. 
	\end{equation*}
	Applying Lemma~\ref{rmk2} with \eqref{UW_far}, we obtain the claimed inequality: 
	\begin{equation}\label{Uyy_far}
		\limsup_{|y|\rightarrow \infty}|\wt{Z}(y,s)|\leq \limsup_{|y|\rightarrow \infty}|\wt{Z}(y,s_0)|e^{-\lambda_D(s-s_0)}+  Ce^{-\lambda_Ds}  \le  \frac{M^{1/8}}{4},  
	\end{equation}
	where we have used the fact that   $\textstyle \|\wt{Z}(y,s_0)\|_{L^{\infty}(|y|\geq 10)} \le ( 10^2+1)^{1/2} \cdot 10^{-1}  \cdot C_2 < \frac14 M^{1/8}$, which comes from 
	\eqref{1D3}, i.e., $\|U_{yy}(\cdot,s_0)\|_{L^{\infty}}\leq C_2$. 

	Then we apply Theorem \ref{max_2} with \eqref{Uyy_local}-\eqref{Uyy_far} to conclude that $\textstyle \|\wt{Z} (\cdot, s) \|_{L^{\infty}}\leq \frac12 M^{1/8}$ for all $s\in[s_0,\sigma_1]$, which yields the desired estimate \eqref{Uyy-bd}. 
	This completes the proof. 
\end{proof}

We show \eqref{6}, thereby closing the bootstrap assumption \eqref{EP2_1D4-p}.
\begin{lemma}\label{Wy4_lem-p}
	Assume that \eqref{Boot_2-p} holds. Then for sufficiently large $M>0$ and small $\ve>0$, it holds that 
	\begin{equation*}
		\|\partial_y^4U(\cdot,s)\|_{L^{\infty}}\leq \frac{M}{2} 
	\end{equation*}
for all $s\in [s_0,\sigma_1]$.
\end{lemma}

\begin{proof}
	From \eqref{CH_diff_4}, we have
	\begin{equation}\label{Uyyyy-eq}
		\partial_s \partial_y^4U + D^U_4 \partial_y^4U + \mathcal{V} \partial_y^5 U  = F^U_4, 
	\end{equation}
	where $\mathcal{V}$ is defined in \eqref{V}, and
	\begin{subequations}
		\begin{align*}
			D^U_4 &:= \frac{17}{2} + \frac{4U_y}{1-\dot{\tau}},
			\\
			F^U_4 &:= -\frac{ e^{-2s} }{1-\dot{\tau}}\left(\frac{\partial_y^3P-6e^{-3s}U_yU_{yy}-2e^{-3s/2}(e^{-3s/2}U+\kappa)\partial_y^3U}{\beta}-\frac{2\gamma}{\sqrt{\beta}} e^{-3s/2}\partial_y^3U\right)-\frac{7U_{yy}\partial_y^3U}{1-\dot{\tau}}.
		\end{align*}
	\end{subequations}
	Using \eqref{tau} and \eqref{Uy1-p}, it is straightforward to check that 
	\begin{equation}\label{DU4}
		D^U_4 = \frac{17}{2}+\frac{4U_y}{1-\dot{\tau}}\geq \frac{17}{2}-8(1+\veps)\geq \frac{1}{4}. 
	\end{equation}
%
	On the other hand, applying $\partial_y$ to \eqref{CH_P}, and using   \eqref{hat_up}, \eqref{P-Py}, \eqref{EP2_1D3-p}  and \eqref{Uy1-p},  we have 
	\begin{equation*}
		\begin{split}
			|\partial_y^3P|&=\left|\frac{e^{-5s}P_y-2(e^{-3s/2}U+\kappa)e^{-13s/2}U_y}{\beta}-\frac{2\gamma}{\sqrt{\beta}} e^{-13s/2}U_y-e^{-3s}U_yU_{yy}\right| 
			\\
			&\leq C\beta^{-1/2}e^{-13s/2}+Ce^{-3s}\leq Ce^{-3s}, 
		\end{split}
	\end{equation*} 
	where in the last inequality, we used \eqref{beta-exp}, i.e., 
	$\beta^{-1/2}e^{-13s/2} \lesssim 
	 e^{-4s}$.  
Using this, together with \eqref{hat_up},  \eqref{tau}, \eqref{EP2_1D3-p}, \eqref{EP2_1D5-p},  \eqref{Uy1-p} and \eqref{beta-exp},  we have 
	\begin{equation}\label{FU4}
		\begin{split}
			\|F^U_4(\cdot,s)\|_{L^{\infty}}&\leq C\beta^{-1}e^{-5s}+C\beta^{-1/2}e^{-7s/2}+ 7 (1+\veps)  M^{7/8}
			\\
			&\leq C\veps^{\deltaz}+ 7 (1+\veps) M^{7/8} \leq 8M^{7/8}. 
		\end{split}
	\end{equation}
	Here, similarly to before, we invoked $\beta^{-1}e^{-5s}\leq C   \veps^{\deltaz}$, thanks to \eqref{beta-exp}.

	Let  $\psi$ be the characteristic curve  satisfying $\partial_s \psi=\mathcal{V} (\psi, s)$ and $\psi(y,s_0)=y$. 
By integrating \eqref{Uyyyy-eq} along $\psi$, and using \eqref{DU4}, \eqref{FU4} and the fact that $\|\partial_y^4U(\cdot,s_0)\|_{L^{\infty}}\lesssim1$ from \eqref{1D4}, 
we have 
\begin{equation*}
\begin{split} 
	\|\partial_y^4U(\cdot,s)\|_{L^{\infty}} & \le 	\|\partial_y^4U(\cdot,s_0)\|_{L^{\infty}}e^{-(s-s_0)/4}+ 32M^{7/8} \leq \frac{M}{2},
\end{split}
\end{equation*}
where in the last inequality, we choose $M>0$ to be sufficiently large. This completes the proof.
\end{proof}

Next we show \eqref{5}.
\begin{lemma}\label{w-3-est}
	Assume that \eqref{Boot_2-p} holds. Then for sufficiently large $M>0$ and small $\ve>0$, it holds that 
	\begin{equation*}
		\|\partial_y^3U(\cdot,s)\|_{L^{\infty}}\leq \frac{M^{3/4}}{2}
	\end{equation*}
	for all $s\in[s_0,\sigma_1]$.
\end{lemma}

\begin{proof}
	From \eqref{CH_diff_3}, we have
	\begin{equation*}\label{Uyyyeq}
		\partial_s \partial_y^3U + D^U_3 \partial_y^3U + 	\mathcal{V} \partial_y^4 U  =  F^U_3,
	\end{equation*}
	where $\mathcal{V}$ is defined in \eqref{V}, and 
	\begin{subequations}
		\begin{align*}
			D^U_3 &:= 6 + \frac{3U_y}{1-\dot{\tau}} ,
			\\
			F^U_3 &:=  -\frac{ e^{-2s} }{1-\dot{\tau}}\left(\frac{P_{yy}-2e^{-3s}U_y^2-2e^{-3s/2}(e^{-3s/2}U+\kappa)U_{yy}}{\beta}-\frac{2\gamma}{\sqrt{\beta}} e^{-3s/2}U_{yy}\right) - \frac{2 U_{yy}^2}{1-\dot{\tau}}.
		\end{align*}
	\end{subequations}
	Due to \eqref{EP2_1D2-p} and \eqref{EP2_1D4-p}, it holds  for all  $|y|\leq (8M^{1/4})^{-1}$ that 
	\begin{equation} \label{Uy3loc}
		|\partial_y^3U(y,s)|\leq 		|\partial_y^3U(0,s)|+|y|\|\partial_y^4U(y,s)\|_{L^{\infty}} \leq 257+|y|M\leq \frac{M^{3/4}}{4} 
	\end{equation}
	for sufficiently large $M>0$.

	Now we shall check for the case $|y|\geq (8M^{1/4})^{-1}$. 
	Thanks to \eqref{num_4} and \eqref{EP2_1D1-p}, we have 
	\begin{equation*}
		6+3\overline{U}'+3\wt{U}_y \geq 6+3\overline{U}'-\frac{3 y^2}{1000(1+y^2)} \geq
		\frac{3y^2}{1+y^2}.
	\end{equation*}
	Using this together with \eqref{tau} and \eqref{Uy1-p}, we have for $|y|> ( 8M^{1/4})^{-1}$, 
	\begin{equation}\label{Uy3D'}
		\begin{split}
			D^U_3&= (6+3\overline{U}'+3\wt{U}_y)+\frac{3\dot{\tau}U_y}{1-\dot{\tau}}
			\geq
			\frac{3y^2}{1+y^2}-C\veps^2 \ge \frac{3}{64M^{1/2}+1}  -C\veps^2 \geq \frac{1}{25M^{1/2}}. 
		\end{split}
	\end{equation} 
	On the other hand, using \eqref{hat_up}, \eqref{P-Py} and \eqref{Uy1-p} for \eqref{CH_P}, we have
	\begin{equation*}
		| P_{yy} | = \left| \frac{e^{-5s}P-e^{-5s}(e^{-3s/2}U+\kappa)^2}{\beta}-\frac{2\gamma}{\sqrt{\beta}}e^{-5s}(e^{-3s/2}U+\kappa)-\frac{e^{-3s}U_y^2}{2}\right|\leq Ce^{-3s}.
	\end{equation*} 
	Using this along with  \eqref{hat_up}, \eqref{EP2_1D3-p}, \eqref{Uy1-p} and  \eqref{beta-exp},
	\begin{equation}\label{Uy3F}
		\begin{split}
			 | F^U_3(\cdot,s) | & \leq \frac{e^{-2s}}{1-\dot{\tau}}\left(\frac{|P_{yy}|+ 2e^{-3s}|U_y|^2+2e^{-3s/2}|e^{-3s/2}U+\kappa||U_{yy}|}{\beta}+\frac{2\gamma}{\sqrt{\beta}} e^{-3s/2}|U_{yy}| \right)+\frac{2|U_{yy}|^2}{1-\dot{\tau}}
			\\
			&\leq C(\beta^{-1}e^{-5s}+\beta^{-1/2}e^{-7s/2})+2(1+\veps) (M^{1/8})^2 < 3M^{1/4}. 
		\end{split}
	\end{equation}
	
	Next, we claim that $\textstyle\limsup_{|y|\rightarrow \infty}|\partial_y^3U(y,s)|\leq \frac14 M^{3/4} $. 
	Due to \eqref{Uy3D'} and \eqref{Uy3F}, 
	it holds that for $|y|\geq 10$, 
	\begin{subequations}
		\begin{align*}
			&\inf_{|y|\geq 10}D^U_3(y,s)\geq \frac{3\cdot 10^2}{1+10^2}-C\veps^2 =:\lambda_{D_3^U} >0, 
			\\
			&\|F^U_3(y,s)\|_{L^{\infty}(|y|\geq 10)}\leq 3M^{1/4}.
		\end{align*}
	\end{subequations}
	 Having these and \eqref{UW_far}, we apply Lemma~\ref{rmk2} to get
	\begin{equation}\label{Uy3dec} 
		\limsup_{| y |\rightarrow \infty}|\partial_y^3U( y,s)|  \leq
		 \limsup_{|y|\rightarrow \infty}|\partial_y^3U(y,s_0)|e^{-  \lambda_{ D_3^U} (s-s_0)}+ 3M^{1/4} ( \lambda_{ D_3^U})^{-1} \leq \frac{M^{3/4}}{4}
	\end{equation}
	for some constant $M>0$ sufficiently large. 
	Here we have used the initial condition \eqref{1D5} yielding $\|\partial_y^3U(\cdot,s_0)\|_{L^{\infty}}\leq C_3$. 
	Then, applying Theorem \ref{max_2} with  \eqref{init_24-p-w} and \eqref{Uy3loc}-\eqref{Uy3dec},  we arrive at
	\begin{equation*}
		\|\partial_y^3U(\cdot,s)\|_{L^{\infty}}\leq \frac{M^{3/4}}{2}.
	\end{equation*}
	This completes the proof. 
\end{proof}

\begin{lemma}
	Assume that \eqref{Boot_2-p} holds. Then, for sufficiently small $\ve>0$, it holds that 
	\begin{equation}\label{Phiyy_decay-p}
		\left\|(y^{2/5}+1) \left(\frac{P (y,s)-(e^{-3s/2}U(y,s) +\kappa)^2 }{\beta} -\frac{2\gamma}{\sqrt{\beta}}(e^{-3s/2}U(y,s) +\kappa)\right)\right\|_{L^{\infty}_y }\leq Ce^{(2-\deltaz /5)s}, 
	\end{equation} 	
where $\deltaz>0$ is as in \eqref{init_w_3-p}. 
\end{lemma}

\begin{proof}
	Note that $\textstyle K_a (x):=\frac{1}{2a}e^{-a|x|}$ is the resolvent kernel for the operator $(a^2-\partial_x^2)$ on $\mathbb{R}$ for $a>0$. Using this with $a=\beta^{-1/2}e^{-5s/2}$, we have from \eqref{CH_P} that
	\begin{equation*}
		\begin{split}
			\frac{P(y,s)}{\beta}=\frac{e^{-5s/2}}{2\sqrt{\beta}}\int_{\mathbb{R}}e^{-\beta^{-1/2}e^{-5s/2}|y-z|}\left( \frac{(e^{-3s/2}U+\kappa)^2}{\beta}+\frac{2\gamma}{\sqrt{\beta}} (e^{-3s/2}U+\kappa) +\frac{e^{2 s}U_y^2}{2}\right)(z,s)\,dz.
		\end{split}
	\end{equation*}
	Using this together with the fact that $\textstyle \frac{a}{2}\int_{\mathbb{R}} e^{-a|x|} dx =1$,  
	we have 
	\begin{equation*}
		\begin{split}
			y^{2/5}&\left|\frac{P-(e^{-3s/2}U+\kappa)^2}{\beta}-\frac{2\gamma}{\sqrt{\beta}}(e^{-3s/2}U+\kappa)\right|
			\\
			&\leq
			\frac{e^{-5s/2}}{2\sqrt{\beta}}y^{2/5}\int_{\mathbb{R}}e^{-\beta^{-1/2}e^{-5s/2}|y-z|}\left|\frac{(e^{-3s/2}U+\kappa)^2(z,s)}{\beta}-\frac{(e^{-3s/2}U+\kappa)^2(y,s)}{\beta}\right|\,dz
			\\
			&\quad+ 
			\frac{\gamma}{\beta} e^{-4s}y^{2/5}\int_{\mathbb{R}} e^{-\beta^{-1/2}e^{-5s/2}|y-z|}\left|U(z,s)-U(y,s)\right|\,dz 
			\\
			&\quad+\frac{e^{-s/2}}{4\sqrt{\beta}}y^{2/5}\int_{\mathbb{R}}e^{-\beta^{-1/2}e^{-5s/2}|y-z|}|U_y (z,s) |^2 \,dz
			\\
			&\leq
			C \beta^{-1/2}y^{2/5}\int_{A} e^{-\beta^{-1/2}e^{-5s/2}|y-z|}\left(\beta^{-1/2}e^{-4s}\left|U(z,s)-U(y,s)\right| +e^{-s/2}|U_y (z,s) |^2 \right)\,dz
			\\
			&\quad+C \beta^{-1/2}y^{2/5}\int_{A^c} e^{-\beta^{-1/2}e^{-5s/2}|y-z|}\left( e^{-5s/2}+e^{-s/2}|U_y (z,s) |^2 \right) \,dz
			\\
			&=:I_{A} + I_{A^c},
		\end{split}
	\end{equation*}
	where $A:=\{z\in\mathbb{R} \;|\; |y-z|<|y|/2\}$. Here, in the second inequality, we used the following inequalities, which hold true due to \eqref{u_bound}: 
	\begin{equation*}
		\left|(e^{-3s/2}U+\kappa)^2(z,s)-(e^{-3s/2}U+\kappa)^2(y,s)\right|\leq C\beta^{1/2}e^{-3s/2}\left|U(z,s)-U(y,s)\right|
	\end{equation*} for $z\in A$, and 
	\begin{equation*}
		\begin{split}
			&|U(z,s)-U(y,s)|=e^{3s/2}|(e^{-3s/2}U(z,s)+\kappa)-(e^{-3s/2}U(y,s)+\kappa)|\leq C\beta^{1/2}e^{3s/2},
			\\
			&\left|(e^{-3s/2}U+\kappa)^2(z,s)-(e^{-3s/2}U+\kappa)^2(y,s)\right|\leq C\beta
		\end{split}
	\end{equation*} for $z\in A^c$.
	First, we estimate $I_A$.
	Using \eqref{ax1_0}, one can check that
	\begin{equation*}
		\begin{split}
			I_{A}
			&\leq C\int_{A}e^{-\beta^{-1/2}e^{-5s/2}|y-z|}\left( \beta^{-1}e^{-4s}y^{2/5}|y-z||U_y(\wt{z},s)| +\beta^{-1/2}e^{-s/2}\frac{y^{2/5}}{1+z^{2/5}}|U_y (z,s)| \right) \,dz
			\\
			&\leq  C\int_{A}e^{-\beta^{-1/2}e^{-5s/2}|y-z|}\left( \beta^{-1}e^{-4s}\frac{y^{2/5}}{1+\wt{z}^{2/5}}|y-z| + \beta^{-1/2} e^{-s/2}\frac{y^{2/5}}{1+z^{2/5}}|U_y (z,s) | \right)\,dz
		\end{split}
	\end{equation*}
	for some $\wt{z}$ between $y$ and $z$. 
	Note that $z\in A$ implies $|\wt{z}-y|< |y|/2$, in turn, $|y|/2<|\wt{z}|<3|y|/2$. Then 
	\begin{equation}\label{I01}
		\begin{split}
			I_{A}&\leq  C\int_{A}e^{-\beta^{-1/2}e^{-5s/2}|y-z|}\left(\beta^{-1}e^{-4s}\frac{y^{2/5}}{1+(y/2)^{2/5}}|y-z| +\beta^{-1/2}e^{-s/2}\frac{y^{2/5}}{1+(y/2)^{2/5}}|U_y (z,s) | \right)\,dz 
			\\
			&\leq C\beta^{-1}e^{-4s}\int_{\mathbb{R}}e^{-\beta^{-1/2}e^{-5s/2}|y-z|}|y-z|\,dz +C\beta^{-1/2}e^{-s/2}\int_{\mathbb{R}}e^{-\beta^{-1/2}e^{-5s/2}|y-z|}|U_y (z,s) | \,dz 
			\\
			&\leq C\beta^{-1}e^{-4s}\int_{\mathbb{R}}e^{-\beta^{-1/2}e^{-5s/2}|y-z|}|y-z|\,dz 
			\\
			&\quad + C\beta^{-1/2}e^{-s/2}  \left( \int_{\mathbb{R}}e^{-2\beta^{-1/2}e^{-5s/2}|y-z|}\,dz \right)^{1/2} \left( \int_{\mathbb{R}}|U_y (z,s) |^2 \,dz \right)^{1/2}
			\\
			&\leq Ce^{s} \leq Ce^{(2-\deltaz/5)s}.
		\end{split}
	\end{equation} 
	Here, in the fourth inequality, we used 
	\begin{equation}\label{uyint}
		e^{-s/2}\int_{\mathbb{R}}|U_y (y,s) |^2 \,dy=\int_{\mathbb{R}}|\hu_{x'} (x', t) |^2\,dx=\sqrt{\beta}\int_{\mathbb{R}}|u_x(x,t)|^2\,dx\leq C\sqrt{\beta}, 
	\end{equation}
	which holds thanks to \eqref{H}, and also used the fact that 
	$\textstyle \int_{\mathbb{R}}e^{-a|y-z|}\,dz =2a^{-1}$, $\textstyle \int_{\mathbb{R}}e^{-a|y-z|}|y-z|\,dz=2a^{-2}$. 
	
	Now, we estimate $I_{A^c}$. 
	It is straightforward to check that 
	\begin{equation}\label{I02}
		\begin{split}
			I_{A^c}&\leq C\beta^{-1/2}\int_{A^c}e^{-\beta^{-1/2}e^{-5s/2}|y-z|}y^{2/5}\left(e^{-5s/2}+e^{-s/2}|U_y (z,s) |^2 \right)\,dz
			\\
			&\leq C\beta^{-1/2}e^{-5s/2}\int_{A^c}e^{-\beta^{-1/2}e^{-5s/2}|y-z|/2}\left(e^{-\beta^{-1/2}e^{-5s/2}|y|/4}y^{2/5}\right)\,dz
			\\
			&\quad + C\beta^{-3/10}e^{s/2} \int_{A^c}e^{-\beta^{-1/2}e^{-5s/2}|y-z|}(\beta^{-1/2}e^{-5s/2})^{2/5}|y-z|^{2/5}|U_y (z,s) |^2 \,dz 
			\\
			&\leq C\beta^{-3/10}e^{-3s/2}\int_{A^c}e^{-\beta^{-1/2}e^{-5s/2}|y-z|/2}\,dz + C\beta^{-3/10}e^{s/2}\int_{A^c}|U_y (z,s) |^2 \,dz 
			\\
			&\leq C\beta^{1/5}e^s\leq Ce^{(2-\deltaz/5)s}.
		\end{split}
	\end{equation}
	Here, in the third inequality, we used 
	$\textstyle e^{-a |y | /4}y^{2/5}\leq \beta^{1/5}e^{s}\left(e^{-a | y | /4}(a | y | )^{2/5}\right)  \leq C\beta^{1/5}e^s$, 
	where $a=\beta^{-1/2}e^{-5s/2}$, and the fact that $e^{-|x|}|x|^{2/5}\leq C$ for all $x\in\mathbb{R}$. Also, we used \eqref{beta} and \eqref{uyint}  in the last two inequalities. 
	 Then, \eqref{I01} and \eqref{I02} yield 
	 \begin{equation*}
		\left\|y^{2/5}\left(\frac{P (y,s)-(e^{-3s/2}U(y,s) +\kappa)^2 }{\beta} -\frac{2\gamma}{\sqrt{\beta}}(e^{-3s/2}U(y,s) +\kappa)\right)\right\|_{L^{\infty}_y }\leq Ce^{(2-\deltaz/5)s}, 
	\end{equation*}
	which, together with \eqref{hat_up}, yields the desired estimate \eqref{Phiyy_decay-p}.
	  This completes the proof.  
\end{proof}

Next, we present an auxiliary lemma which will be used in the proof of Lemma~\ref{mainprop_1-p}.
\begin{lemma}\label{dec-lem}
	Assume that \eqref{Boot_2-p} holds. Then for sufficiently small $\ve>0$, it holds that  
	\begin{equation}\label{Ut_dec_p}
		\limsup_{|y|\rightarrow \infty} |(y^{2/5}+1)(U_y (y, s) -\overline{U}' (y) ) | <  \frac{6}{13} -  2  \theta,
	\end{equation}
	where $\theta$ is defined in \eqref{theta}.
\end{lemma}
\begin{proof}
	We first note that 
	\begin{equation*}
		\begin{split}
			\limsup_{|y|\rightarrow \infty}|(y^{{2/5}}+1)(U_y-\overline{U}')|&\leq \limsup_{|y|\rightarrow \infty}|(y^{2/5}+1)U_y|+\lim_{|y|\rightarrow \infty}|(y^{2/5}+1)\overline{U}'|
			\\
			&= \limsup_{|y|\rightarrow \infty}|(y^{2/5}+1)U_y|+\frac{1}{50^{1/5}},
		\end{split}
	\end{equation*}
	where we have \eqref{U_taylor_far} with $\beta=1$. 
	Since ${50^{-1/5}}< {6/13}-3\theta$ due to \eqref{theta}, it suffices to show that
	\begin{equation*}\label{3.00}
		\limsup_{|y|\rightarrow \infty} |(y^{2/5}+1)U_y(y,s)|< \theta. 
	\end{equation*}
	Letting $\mu(y,s):=(y^{2/5}+1)U_y(y,s)$, we have from \eqref{CH_diff_1} that
	\begin{equation*}
		\partial_s \mu+D^{\mu}\mu+\mathcal{V} \partial_y \mu =F^{\mu},
	\end{equation*}
	where $\mathcal{V}$ is defined in \eqref{V} and
	\begin{subequations}
		\begin{align*}
			D^{\mu} &:= 1+\frac{U_y}{2}-\frac{2y^{2/5}}{5(1+y^{2/5})}\left(\frac{5}{2}+\frac{U}{y}\right),
			\\
			F^{\mu} &:= -\frac{e^{-2s}}{1-\dot{\tau}}(y^{2/5}+1)\left(\frac{P-(e^{-3s/2}U+\kappa)^2}{\beta}-\frac{2\gamma}{\sqrt{\beta}}(e^{-3s/2}U+\kappa)\right)
			\\
			&\qquad -\left(\frac{\dot{\tau}U_y}{2}-\frac{2y^{2/5}}{5(1+y^{2/5})}\left(\frac{e^{3s/2}(\kappa-\dot{\xi})}{y}+\frac{\dot{\tau}U}{y}\right)\right)\frac{\mu}{1-\dot{\tau}}.
		\end{align*}
	\end{subequations}
%
	Using \eqref{Utildey_M-p} and $\wt{U}(0,s)(0,s)=0$, where $\wt{U}(y,s):=U(y,s)-\overline{U}(y)$, we have
\begin{equation*}
	\begin{split}
		D^{\mu}(y,s)&=1+\frac{\wt{U}_y+\overline{U}'}{2}-\frac{2y^{2/5}}{5(1+y^{2/5})}\left(\frac{5}{2}+\frac{\wt{U}+\overline{U}}{y}\right)
		\\
		&\geq 1-\frac{3}{13(1+y^{2/5})}+\frac{\overline{U}'}{2}-\frac{2y^{2/5}}{5(1+y^{2/5})}\left(\frac{5}{2}+\frac{\overline U}{y}+\frac{6}{13y}\int_{0}^{y}{\frac{dy'}{1+{y'}^{2/5}}}\right)
		\\
		&= \frac{10}{13(1+y^{2/5})}+\frac{\overline{U}'}{2}-\frac{2y^{2/5}}{5(1+y^{2/5})}\left(\frac{\overline U}{y}+\frac{6}{13y}\int_{0}^{y}{\frac{dy'}{1+{y'}^{2/5}}}\right). 
	\end{split}
\end{equation*}
	Thanks to \eqref{num_1} and $\overline{U}'\leq 0$ from \eqref{Wbar_neg}, we see that $D^\mu$ is non-negative for $|y|\geq 2\cdot 10^6$, 
i.e.,  \begin{equation}\label{Dmu}
	\inf_{|y|\geq 2\cdot 10^6}D^{\mu} (y,s) \geq 0.
\end{equation} 
Moreover, from  \eqref{Uy1-p} and \eqref{ax1_0},  
 we deduce that $\|\mu (\cdot, s) \|_{L^{\infty}}\leq C$ for some $C>0$. 
Using  
\eqref{tau}, \eqref{Uy1-p}, \eqref{temp_1-p} and \eqref{Phiyy_decay-p}, we get 
\begin{equation}\label{Fmu}
	\begin{split}
		\|F^{\mu}(y,s)\|_{L^{\infty}(|y|\geq 2\cdot 10^6)}&\leq Ce^{-2s}\left\|(y^{2/5}+1)\left(\frac{P-(e^{-3s/2}U+\kappa)^2}{\beta}-\frac{2\gamma}{\sqrt{\beta}} (e^{-3s/2}U+\kappa)\right)\right\|_{L^{\infty}(\mathbb{R})}
		\\
		&\qquad+C\|\mu\|_{L^{\infty}}\left\||\dot{\tau}||U_y|+\frac{e^{3s/2}|\kappa-\dot{\xi}|}{|y|}+\frac{|\dot{\tau}U|}{|y|}\right\|_{L^{\infty}(|y|\geq 2\cdot 10^6)}
		\\
		&\leq 	Ce^{-\deltaz s/5}+C\left(e^{-2s}+e^{-s}+e^{-2s}\right)
		\\
		&\leq Ce^{-\deltaz s/5}.
	\end{split}
\end{equation}
			Now  we apply Lemma~\ref{rmk2} along with \eqref{UW_far}, \eqref{Dmu} and \eqref{Fmu} to get
			\begin{equation*}
				\limsup_{|y|\rightarrow \infty}|\mu(y,s)|\leq \limsup_{|y|\rightarrow \infty}|\mu(y,s_0)|+C\veps^{\deltaz / 5} \le \frac{\theta}{2}+C\veps^{\deltaz /5} < \theta
			\end{equation*}
			for sufficiently small $\ve>0$. 
			Here, we used \eqref{4.3a-p2-w}, that is equivalent to 
			\begin{equation*}
				\limsup_{|y|\rightarrow \infty}|\mu(y,s_0)|\leq \frac{\theta}{2}.
			\end{equation*}
			This completes the proof. 
		\end{proof}

Finally, we show \eqref{Utildey_init}, thereby closing \eqref{Utildey_M-p}.
\begin{lemma}\label{mainprop_1-p}
Assume that \eqref{Boot_2-p} holds. Then for sufficiently small $\ve>0$, it holds that  
		\begin{equation*}
		|(y^{2/5}+1) ( {U}_y (y, s) - \UU'(y) ) | \le \frac{6}{13}-\theta
	\end{equation*}
	 for all $y\in\mathbb{R}$ and $s\in[s_0,\sigma_1]$, where $\theta$ is defined in \eqref{theta}. 
\end{lemma}

\begin{proof}
Let $\nu (y,s ) : = (y^{2/5}+1)\wt{U}_y(y,s)$ where $\wt{U} (y, s) := U(y,s) - \overline{U}(y)$. 
From \eqref{CH_diff_1}, $\nu$ satisfies 
\begin{equation}\label{nu_eq1-p}
	\partial_s \nu + D^\nu(y,s)\nu+\mathcal{V}(y,s) \partial_y \nu = F^\nu_1(y,s) +F^\nu_2(y,s) +\int_{\mathbb{R}}{\nu(y',s) K^\nu(y, s; y') \,dy'},
\end{equation}
where $\mathcal{V}$ is defined in \eqref{V}, and
\begin{subequations}
	\begin{align*}
		D^\nu(y,s) &:= 1+\frac{\wt{U}_y+2\overline{U}'}{2}-\frac{2y^{2/5}}{5(1+y^{2/5})}\left(\frac{5}{2}+\frac{\wt{U}+\overline{U}}{y}\right),
		\\
		F_1^\nu(y,s) &:= -\frac{e^{-2s}}{1-\dot{\tau}}(y^{2/5}+1)\left(\frac{P- (e^{-3s/2}U+\kappa)^2}{\beta}-\frac{2\gamma}{\sqrt{\beta}} (e^{-3s/2}U+\kappa) \right)
		\\
		&\qquad -\frac{\dot{\tau}(y^{2/5}+1)(\overline{U}')^2}{2(1-\dot{\tau})} - (y^{2/5}+1)\overline{U}''\left(\frac{\dot{\tau}\overline{U}}{1-\dot{\tau}}+\frac{e^{3s/2}(\kappa-\dot{\xi})}{1-\dot{\tau}}\right),
		\\
		F^\nu_2(y,s) &:= -\left(\frac{\dot{\tau}(\wt{U}_y+2\overline{U}')}{2(1-\dot{\tau})}-\frac{2y^{2/5}}{5(1+y^{2/5})}\left(\frac{\dot{\tau}U}{(1-\dot{\tau})y}+\frac{e^{3s/2}(\kappa-\dot{\xi})}{(1-\dot{\tau})y}\right)\right)\nu (y,s),
		\\
		K^\nu(y,s;y') &:= -\frac{1}{1-\dot{\tau}}(y^{2/5}+1)\overline{U}''(y)\mathbb{I}_{[0,y]}(y')\frac{1}{1+ y'^{2/5}}.
	\end{align*}
\end{subequations}
We first claim that  
\begin{equation}\label{D-K-p}
	D^\nu(y,s)\geq \int_{\mathbb{R}}|K^\nu(y,s;y')|\,dy', \qquad |y|\geq 2\cdot 10^6.
\end{equation}
We note  by the bootstrap assumption \eqref{Utildey_M-p} that  
\begin{equation}\label{nu_D-p}
	\begin{split}
		D^\nu(y,s)&=1+\frac{1}{2}\wt{U}_y+\overline{U}'-\frac{2y^{2/5}}{5(1+y^{2/5})}\left(\frac{5}{2}+\frac{\wt{U}+\overline{U}}{y}\right)
		\\
		&\geq 1-\frac{3}{13(1+y^{2/5})}+\overline{U}'-\frac{2y^{2/5}}{5(1+y^{2/5})}\left(\frac{5}{2}+\frac{\overline U}{y}+\frac{6}{13y}\int_{0}^{y}{\frac{dy'}{1+{y'}^{2/5}}}\right) 
		\\
		&=\frac{10}{13(1+y^{2/5})}+\overline{U}'-\frac{2y^{2/5}}{5(1+y^{2/5})}\left(\frac{\overline U}{y}+\frac{6}{13y}\int_{0}^{y}{\frac{dy'}{1+{y'}^{2/5}}}\right) 
		=: D^\nu_{-}(y),
	\end{split}
\end{equation}
and by \eqref{dottau}, we have
\begin{equation}\label{nu_K-p}
	\begin{split}
		\int_{\mathbb{R}}{|K^\nu(y, s; y')|\,dy'}
		&\leq (1+\veps)(y^{2/5}+1)|\overline{U}''(y)|\int^{|y|}_{0}{\frac{dy'}{1+{y'}^{2/5}}} =: K^\nu_{+}(y).
	\end{split}
\end{equation}
Thanks to \eqref{num_1}, i.e.,   
\begin{multline*}\label{4.43}
	\lambda(y^{2/5}+1)|\overline{U}''(y)|\int^{|y|}_0\frac{dy'}{1+(y')^{2/5}}
	\\
	\leq \frac{10}{13(1+y^{2/5})}+\overline{U}' (y) - \frac{2y^{2/5}}{5(1+y^{2/5})}\left(\frac{\overline{U}( y ) }{y}+\frac{6}{13y}\int^y_0\frac{dy'}{1+ y'^{2/5}}\right), \quad |y|\geq 2\cdot 10^6
\end{multline*}
for some constant $\lambda>1$, we have 
\begin{equation}\label{D-K+}
D^\nu_-(y) \ge  K^\nu_+(y), \quad |y|\geq 2\cdot 10^6, 
\end{equation} 
which implies \eqref{D-K-p}. 

%
%
Next  we show that 
\begin{equation}\label{F-nu-p}
	\|F^\nu_1(\cdot,s)\|_{L^{\infty}(|y|\geq 2\cdot 10^6)} + \|F^\nu_2(\cdot,s)\|_{L^{\infty}(|y|\geq 2\cdot 10^6)} \leq Ce^{-\deltaz s/5}.
\end{equation}
For $F^\nu_1$, by 
\eqref{asymp-y-infty}, \eqref{tau}, \eqref{temp_1-p} and \eqref{Phiyy_decay-p},
we obtain
\begin{equation}\label{F1-p}
	\begin{split}
	|F^\nu_1 (y, s) |&\leq Ce^{-2s}(y^{2/5}+1)\left|\frac{P- (e^{-3s/2}U+\kappa)^2}{\beta}-\frac{2\gamma}{\sqrt{\beta}} (e^{-3s/2}U+\kappa) \right|
	\\
	&\quad +|\dot{\tau}|(y^{2/5}+1) | \overline{U}'|^2+ 2(y^{2/5}+1)|\overline{U}''||\dot{\tau}\overline{U}+e^{3s/2}(\kappa-\dot{\xi})| 
	\\
	&\leq Ce^{-\deltaz s/5}.
	\end{split}
\end{equation}
 Using \eqref{tau} and \eqref{Utildey_M-p}, we have 
\begin{equation*}
	\begin{split}
	|F^\nu_2 (y, s) |	&\leq C\left( |\dot{\tau}|\left(|\wt{U}_y|+2|\overline{U}'|\right)+\left|\frac{\dot{\tau}U}{y}\right|+\frac{e^{3s/2}|\kappa-\dot{\xi}|}{|y|}\right).
	\end{split}
\end{equation*}
Let us estimate the right-hand side term by term. By \eqref{Wbar_neg} and \eqref{EP2_1D1-p}, we have 
\begin{equation}\label{2.64}
	|\dot{\tau}|\left(|\wt{U}_y|+2|\overline{U}'|\right)\leq Ce^{-2s}.
\end{equation}
Also, using $U(0,s)=0$ and $|U_y|\leq C$ from \eqref{Uy1-p},  
\begin{equation*}\label{F2_third-p}
	\begin{split}
		 \frac{|\dot{\tau} U|}{|y|} &\leq \frac{|\dot{\tau}|}{|y|}\left|\int^{y}_0U_y(y',s)\,dy'\right|
		\leq Ce^{-2s },
	\end{split}
\end{equation*}
and by 
\eqref{temp_1-p}, 
\begin{equation}\label{F2_second-p}
	\frac{e^{3s/2}|\kappa-\dot{\xi}|}{|y|}
	\leq Ce^{-s}
\end{equation}
for all $|y|\geq 2\cdot 10^6$.
Combining \eqref{2.64}--\eqref{F2_second-p} , we obtain
\begin{equation}\label{F2-p}
	\|F^\nu_2(\cdot,s)\|_{L^{\infty}(|y|\geq 2\cdot 10^6)} \leq Ce^{-s}.
\end{equation}
Then \eqref{F1-p} and \eqref{F2-p} give \eqref{F-nu-p}.

In order to conclude the proof, we claim that
\begin{equation}\label{claim_reg-p}
	\|\nu(\cdot,s)\|_{L^{\infty}}\leq \frac{6}{13}-\theta, \qquad s\in[s_0,\sigma_1].
\end{equation}
Suppose to the contrary that \eqref{claim_reg-p} fails. 
Thanks to the fact that $\nu \in C([s_0, \sigma_1]; C^3(\mathbb{R}))$ and the initial condition  \eqref{4.3a-p-w}, i.e.,  $|\nu(\cdot,s_0)|\leq {6/13}-3\theta$,  
we have that 
\begin{equation}\label{s2-def}
s_2:= \min \{ s \in [s_0, \sigma_1]: \|\nu(\cdot,s)\|_{L^{\infty}} = {6/13}-\theta \}
\end{equation} 
 is well-defined, and that there exists $s_1\in (s_0, s_2)$  such that
\begin{equation}\label{34-p}
	\frac{6}{13}-2\theta = \|\nu(\cdot,s_1)\|_{L^{\infty}}  \le \|\nu(\cdot,s)\|_{L^{\infty}} < \frac{6}{13}-\theta \quad \text{ for all } s\in(s_1,s_2). 
\end{equation}
Then, for each $s\in[s_1,s_2]$, since $\nu \in C[s_0, \sigma_1]; C^3(\mathbb{R}))$ and $\nu$ enjoys the decay property \eqref{Ut_dec_p} in Lemma~\ref{dec-lem},  we see that $C^{\nu}:=\{ y \in \mathbb{R} :  \partial_y \nu(y,s) =0 \}$  is a non-empty closed set of the critical points of $\nu$,  and that $|\nu (\cdot, s) |$ attains its maximum at a point in $C^\nu$, i.e., $\| \nu(\cdot, s) \|_\infty = \max\{ | \nu(y, s) | : y \in C^{\nu}\}.$ Thus, for all $s \in [s_1,s_2]$, we can define $$y_*(s):=\min\{ y\in C^{\nu} : |\nu(y,s) | = \|\nu(\cdot, s) \|_{L^{\infty}}    \} \in(-\infty, \infty)$$
such that
\begin{equation}\label{claim2_nudec0-p}
	\|\nu(\cdot, s) \|_{L^{\infty}} = |\nu(y_*(s), s)|.
\end{equation}
In view of \eqref{34-p} and  \eqref{claim2_nudec0-p}, we see that $\nu(y_*(s), s) \ne 0$ and $\partial_y \nu(y_*(s), s) =0$  for each $s\in[s_1, s_2]$.
On the other hand, 
 thanks to \eqref{EP2_1D1-p},    
\begin{equation*}
	\|\nu(\cdot,s)\|_{L^{\infty}(|y|\leq 2\cdot 10^6)}\leq ((2\cdot 10^6)^{2/5}+1)\|\wt{U}_y(\cdot,s)\|_{L^{\infty}(|y|\leq 2\cdot 10^6)}\leq  \frac{ (2\cdot 10^6)^{2/5}+1 }{1000} < 50^{-1/5} < \frac{6}{13}-3\theta, 
\end{equation*}  
where the last inequality holds true since $\theta>0$ is the number defined in \eqref{theta}. This together with \eqref{34-p} implies that   
\[ \|\nu(\cdot,s)\|_{L^{\infty}(|y|\leq 2\cdot 10^6)} <  \frac{6}{13}-3\theta <  \frac{6}{13}-2\theta 
 \le \|\nu(\cdot,s)\|_{L^{\infty} (\mathbb{R} ) },  \]
 from which we deduce that
\begin{equation*}
	|y_*(s)|\geq 2\cdot 10^6.
\end{equation*} 
Hence   by  \eqref{D-K-p},
 \eqref{nu_D-p}, \eqref{nu_K-p}, \eqref{D-K+} and \eqref{claim2_nudec0-p},  
if $\nu(y_*(s),s)> 0$, then
\begin{equation}
	\begin{split}\label{nu_DK-p}
		D^\nu(y_*(s),s)\nu(y_*(s),s)&\geq D^\nu_{-}(y_*(s))\|\nu(\cdot,s)\|_{L^{\infty}}
		\\
		&\geq K^\nu_{+}(y_*(s))\|\nu(\cdot,s)\|_{L^{\infty}}\geq \left|\int_{\mathbb{R}}K^\nu(y_*(s),s;y')\nu(y',s)\,dy'\right|.
	\end{split}
\end{equation}
Similarly, if  $\nu(y_*(s),s) < 0$, we have 
\begin{equation}
	\begin{split}\label{nu_DK2-p}
		D^\nu(y_*(s),s)\nu(y_*(s),s)&\le - D^\nu_{-}(y_*(s))\|\nu(\cdot,s)\|_{L^{\infty}}
		\\
		&\le - K^\nu_{+}(y_*(s))\|\nu(\cdot,s)\|_{L^{\infty}} \le - \left|\int_{\mathbb{R}}K^\nu(y_*(s),s;y')\nu(y',s)\,dy'\right|.
	\end{split}
\end{equation}
Using  \eqref{F-nu-p}, \eqref{nu_DK-p}, \eqref{nu_DK2-p} and the property $\partial_y \nu(y_*(s), s) =0$, which is due to the fact that $y_*(s)\in C^{\nu}$, we have from \eqref{nu_eq1-p} that if   $\nu(y_*(s),s)> 0$,
\begin{equation}\label{nu_temp-p}
	\begin{split}
		\partial_s \nu(y_*(s) ,s) & \leq  Ce^{-\deltaz s/5}+\int_{\mathbb{R}} \nu(y',s)K^\nu(y_*,s;y')\,dy'-D^\nu(y_*,s)\nu
		\leq Ce^{-\deltaz s/5},
	\end{split}
\end{equation}
and that if $\nu(y_*(s),s)< 0$, 
\begin{equation}\label{nu_temp--p}
	\begin{split}
		\partial_s \nu(y_*(s) ,s) & \ge - Ce^{-\deltaz s/5}.
	\end{split}
\end{equation} 
%
%
%
%
%
Then for fixed $s > s_0$, by the definition of $y_*$, it holds that
 \[ 
 \|\nu(\cdot,s-h)\|_{L^\infty} \geq |\nu(y_*(s) - h\mathcal{V} (y_*(s),s), s - h)| 
\] 
for any small $h>0$. 
Then, it is straightforward to check that 
\begin{equation}\label{AP_R1} 
\begin{split}
\lim_{h \to 0^+} \frac{\|\nu(\cdot,s-h)\|_{L^\infty} - \|\nu(\cdot,s)\|_{L^\infty}}{-h} 
& \leq (\partial_s + \mathcal{V} (y_*(s),s)\partial_y)|\nu|(y,s)|_{y=y_*(s)}, 
\end{split}
\end{equation} 
provided that the limit on the LHS exists. 
Note that, by Rademacher's theorem, 
$ \| \nu(\cdot, s)\|_{L^{\infty}}$, being Lipschitz continuous in $s$, is differentiable at almost all $s\in[s_1, s_2]$. Since  $\partial_y \nu(y_*(s), s) =0$, we deduce from \eqref{AP_R1} that 
\begin{equation} \label{L-thm-p}
\begin{split}
\frac{d}{ds} \| \nu(\cdot, s)\|_{L^{\infty}} 
& \leq  \left\{ \begin{array}{l l}
\partial_s \nu(y, s)|_{y=y_*(s)} & \text{if } \nu(y_*(s),s)>0, \\
-\partial_s \nu(y, s)|_{y=y_*(s)} & \text{if } \nu(y_*(s),s)<0
\end{array}
\right.  
\end{split}
\end{equation} 
for almost all $s\in[s_1, s_2]$. 

%
%
%
%
%
%

Hence, integrating \eqref{L-thm-p} with respect to $s$ over $[s_1, s_2]$, and using  \eqref{nu_temp-p} and \eqref{nu_temp--p}, we have 
\begin{equation*}
\begin{split} 
 \| \nu(\cdot, s_2) \|_{L^{\infty}} 
& =  \| \nu(\cdot, s_1) \|_{L^{\infty}} +  \int_{s_1}^{s_2} \frac{d}{ds} \| \nu(\cdot, s)\|_{L^{\infty}} \; ds  
\\
& \le \| \nu(\cdot, s_1) \|_{L^{\infty}} + \int_{s_1}^{s_2} C e^{-\deltaz s/5} \; ds \\
&  \leq  \| \nu(\cdot, s_1) \|_{L^{\infty}} + C \veps^{\deltaz/5} 
\\
&= \frac{6}{13}- 2\theta + C \ve^{\deltaz /5} < \frac{6}{13}-\theta, 
\end{split} 
\end{equation*}
provided that $\ve>0$ is sufficiently small. 
This leads to a contradiction to the choice of $s_2$ in \eqref{s2-def}. This proves \eqref{claim_reg-p}, and completes the proof of Lemma~\ref{mainprop_1-p}. 
\end{proof}

\section{Self-similar blow-up for the Hunter-Saxton equation}\label{HS-proof}

In this section, we study the blow-up solutions for the Hunter-Saxton (HS) equation:
\begin{equation}\label{HS}
	v_{xt} + v v_{xx} + \frac12 v_x^2 =0, \quad
	 x\in\mathbb{R}, \quad t\ge-1.
\end{equation}
Similar to \eqref{CH}, the HS equation \eqref{HS} can be written as
\begin{subequations}\label{HS_sys}
	\begin{align}
		&v_t+vv_x=-q_x,\label{HS_1}
		\\
		&q_{xx}=- \frac{1}{2} v_x^2.\label{HS_2}
	\end{align}
\end{subequations} 
We consider the initial value problem \eqref{HS} on $\mathbb{R}\times [-1, \infty)$, i.e., the initial time $t=-1$, with the initial data $v_0\in H^5(\mathbb{R})$ satisfying \eqref{init_v}--\eqref{4.3a-v2} with $\alphav=1$.
As stated in Theorem~\ref{thm-HS}, we check that the solution $v$ blows up at $t=0$ at which $v(\cdot, 0)\in 
C^{3/5}$. 
Since the proof of Theorem~\ref{thm-HS} is similar to, yet simpler than, that of Theorem~\ref{mainthm} for the CH equation, we outline it here with emphasis on noteworthy differences. 

Similarly as in the case of the CH equation, we introduce the self-similar variables and modulations as 
\begin{equation}\label{H-ys}
	y(x,t) := \frac{x-\xi(t)}{(\tau(t)-t)^{5/2}}, \quad s(t) :=-\log\left( \tau(t) - t \right)
\end{equation}
and
\begin{equation}\label{H-WZPhi-p}
	v(x,t)-\kappa(t) =: e^{-3s/2} V(y,s),  \quad  q(x,t) = : Q(y,s),
\end{equation}
where  $\tau, \xi, \kappa: [-1, 0) \to \mathbb{R}$ are continuous satisfying the initial conditions
\begin{equation*}
	\tau(-1)=0,\quad \xi(-1)=0,\quad \kappa(-1)=v(0,-1).
\end{equation*}
In fact, $\tau, \xi$, and $\kappa$ are subsequently chosen to satisfy the constraints: 
\begin{equation}\label{HS-constraint}
	V(0,s)=0, \quad V_y(0,s)= -2, \quad V_{yy}(0,s)=0. 
\end{equation} 
Inserting the ansatz \eqref{H-WZPhi-p} into \eqref{HS_sys}, we obtain the equations for $V$ and $Q$:   
\begin{subequations}
	\begin{align}
		&\partial_s V - \frac{3}{2} V+\left(\frac{V}{1-\dot{\tau}} + \frac{5}{2} y  + \frac{e^{3s/2}(\kappa-\dot{\xi})}{1-\dot{\tau}}\right)V_y=-\frac{e^{s/2}\dot{\kappa}}{1-\dot{\tau}}-\frac{e^{3s} Q_y}{1-\dot{\tau}}, \label{HS_diff_0'}
		\\
		&-  Q_{yy} =  \frac{e^{-3s}V_y^2}{2} . \label{HS_P'}
	\end{align}
\end{subequations}
Applying $\partial_y^n$ for $n=1,2$ to \eqref{HS_diff_0'}, we have
\begin{subequations}\label{HS_diff'}
	\begin{align}
		&\left( \partial_s + 1 + \frac{V_y}{2(1-\dot{\tau})} \right)V_y + \left(\frac{V}{1-\dot{\tau}} + \frac{5}{2} y  + \frac{e^{3s/2}(\kappa-\dot{\xi})}{1-\dot{\tau}}\right) V_{yy} = 0,
		\label{HS_diff_1'} 
		\\ 
		& \left( \partial_s + \frac{7}{2} + \frac{2V_y}{1-\dot{\tau}} \right)V_{yy} + \left(\frac{V}{1-\dot{\tau}} + \frac{5}{2} y  + \frac{e^{3s/2}(\kappa-\dot{\xi})}{1-\dot{\tau}}\right) \partial_y^3 V  = 0.  \label{HS_diff_2'} 
	\end{align}
\end{subequations}
Using the constraints \eqref{HS-constraint} for \eqref{HS_diff_0'}, \eqref{HS_diff_1'} and \eqref{HS_diff_2'}, we obtain a set of ODEs for $\tau, \kappa$ and $\xi$:  
\begin{equation}\label{H-b-mod}
		  \dot{\tau} = 0,  \quad \dot{\kappa} = -e^{5s/2}Q_y(0,s), \quad \dot{\xi} = \kappa. 
\end{equation}
In particular, thanks to $\dot\tau(t) =0$ with $\tau(-1) =0$, we see that  $\tau(t) \equiv0$, and  that the maximal existence time, defined as $T_\ast:= \inf \{ t\ge -1 : \tau(t) = t \}$ is $T_\ast =0$.

Updating \eqref{H-ys} and \eqref{HS_diff_0'} with \eqref{H-b-mod}, i.e., the self-similar variables in \eqref{H-ys} are simplified as 
\begin{equation*}\label{H-ys'}
	y(x,t) := \frac{x-\xi(t)}{(-t)^{5/2}}, \quad s(t) :=-\log\left(- t \right),
\end{equation*}
we obtain
\begin{equation}\label{HS_diff_0}
		\partial_s V - \frac{3}{2} V+\mathcal{U}V_y=-e^{s/2}\dot{\kappa}-e^{3s} Q_y, 
\end{equation}
where $\textstyle \mathcal{U}:= V + \frac{5}{2} y$. 
By recursively differentiating \eqref{HS_diff_0} in $y$, we have 
\begin{subequations}\label{HS_diff}
	\begin{align}
		&\left( \partial_s + 1 + \frac{V_y}{2} \right)V_y + \mathcal{U} V_{yy} = 0,
		\label{HS_diff_1} 
		\\ 
		& \left( \partial_s + \frac{7}{2} + 2V_y \right)V_{yy} + \mathcal{U} \partial_y^3 V  = 0,  \label{HS_diff_2} 
		\\
		& \left( \partial_s + 6 + 3V_y \right)\partial_y^3V + \mathcal{U} \partial_y^4 V  
		=  - 2 V_{yy}^2, \label{HS_diff_3}
		\\
		& \left( \partial_s + \frac{17}{2} + 4V_y \right) \partial_y^4V + \mathcal{U} \partial_y^5 U 
		= -7V_{yy}\partial_y^3V. \label{HS_diff_4}
	\end{align}
\end{subequations}
 
Here, we note that, in contrast to \eqref{CH_P} and \eqref{CH_diff}, any modulation functions $\tau, \xi$ and $\kappa$ are involved in  the equations \eqref{HS_P'} and \eqref{HS_diff}. This makes the analysis much simpler than that for the CH equation.
Since the smallness condition on $\veps > 0$ in Theorem~\ref{mainthm} and Proposition~\ref{mainprop} is only required to bound the modulation functions, this smallness condition can be removed  in Theorem~\ref{thm-HS}. Consequently, one can set the initial time as $t_0 = -1$, then the blow-up time becomes $T_\ast =0$, as stated in Theorem~\ref{thm-HS}. 

	It is straightforward to check that \eqref{HS} possesses an invariant Hamiltonian integral. Multiplying \eqref{HS} by $v_x$, we obtain
\begin{equation*}
	\frac12 \frac{d}{dt} \int_{\mathbb{R}}  v_x^2\,dx+\frac{1}{2}\int_{\mathbb{R}}(vv_x^2)_x\,dx 
=0. 
\end{equation*}
This implies that   
\begin{equation}\label{HS-H}
	\Hv(t):=\int_{\mathbb{R}} v_x^2(x,t) \,dx=\Hv(-1).
\end{equation} 
Now, we will obtain the uniform bound of $Q_y$, $\kappa$ and $\xi$. Although the bounds for $Q_y$, $\kappa$ and $\xi$ appear to be the same as those in Lemma~\ref{P-Py-lem} and Lemma~\ref{xilem}, we obtain them slightly differently, which we present in the following.
\begin{lemma}
	It holds that 
	\begin{equation}\label{H-P-Py}
		e^{5s/2} | Q_y (y, s) | \le C, \quad y\in \mathbb{R}, \;s\ge s_0,
	\end{equation}
and
	\begin{equation*}
		|\kappa(t) |\leq C, \qquad |\xi(t) |\leq C,
	\end{equation*}
	where $C=C(\Hv(-1))>0$ is a uniform constant depending only on the initial data $\Hv(-1) = \| \partial_x v_0 \|_{L^2}^2$. 
\end{lemma}
\begin{proof}
We note from $v \in C([-1,T_\ast); H^5(\mathbb{R}))$ and \eqref{HS_1}, that 
$\textstyle \lim_{y\rightarrow -\infty}Q_y(y,s)=0$. 
	Then, we get
	\begin{equation*}
		\begin{split}
			| Q_y(y,s) | & \le 
			\int^y_{-\infty} | Q_{yy}(z,s) | \,dz= \frac{e^{-3s}}{2}\int^y_{-\infty}V_y^2(z,s)\,dz 
			 \leq 
			   \frac{e^{-5s/2}}{2}\int_{\mathbb{R}}v_x^2(x,t)\,dx
			\leq Ce^{-5s/2},
		\end{split}
	\end{equation*}
	where we used \eqref{HS_P'} and \eqref{HS-H}. 
	This proves \eqref{H-P-Py}.

	Also, by \eqref{H-b-mod} and \eqref{H-P-Py}, we have 
	\begin{equation*}
		|\kappa(t)|\leq |\kappa_0|+\int^t_{-1}|\dot{\kappa}(t')|\,dt' \leq |\kappa_0|+\int^\infty_{0}e^{3s'/2}| Q_y (0, s') |\,ds' \leq C.
	\end{equation*} 
On the other hand, in view of the relation $\dot{\xi} = \kappa$ 
	from  \eqref{H-b-mod}, we have
	\begin{equation*}
		|\xi (t) |\leq \int^t_{-1} |\dot{\xi} (t') |\,dt'
		=\int^s_{0}|\dot{\xi}|e^{-s'}\,ds' 
		\leq C\int^s_{0}e^{-s'}ds'
		\leq C.
	\end{equation*}
	This completes the proof. 
\end{proof}
With the aforementioned differences in hand, one can readily complete the proof of Theorem~\ref{thm-HS} by loosely following the proof of Theorem~\ref{mainthm}. 

\section{Appendix}

\subsection{Inequalities related to $\overline{U}$ and its derivatives}

In this subsection, we shall sequentially verify the inequalities, \eqref{num_4} -- \eqref{num_1} in Lemma~\ref{lem-ubar},  where $\overline{U}$ is the solution $\overline{U}_\beta(y)$ with $\beta=1$, of the ODE problem \eqref{Ueq}--\eqref{decay-infty}.
First, we prove \eqref{num_4}.
\begin{lemma}\label{5.1}
	It holds that
	\begin{equation}\label{num-4-0}
		2+\overline{U}'(y) - \frac{6 y^2}{5(1+y^2)}\geq 0 
	\end{equation}
	 for all $y \in \RR$. 
\end{lemma}

\begin{proof}
	Noting that $\textstyle 2+\overline{U}'(y)=\int^y_0 \overline{U}''(y')\,dy'$ and $\textstyle  \frac{y^2}{1+y^2}=\int^y_0 \frac{2y'}{(1+y'^2)^2}\,dy'$, we see that \eqref{num-4-0} is equivalent to 
	\begin{equation*}
		\int^y_0 \overline{U}''(y')-\frac{12 y'}{5(1+y'^2)^2} \,dy' \geq 0 \quad \text{for all }y \in \RR.
	\end{equation*}
	Since $\overline{U}(y)$ is an odd function, we can, without loss of generality, restrict our consideration to the case $y > 0$. 
	Hence, it is enough to show that 
	\begin{equation*}
		\overline{U}''(y)\geq \frac{12 y}{5(1+y^2)^2}, \qquad y> 0.
	\end{equation*}
	First, we consider the region $\{ y > 0 : -2\leq \overline{U}'(y)\leq -4/5\}$.  
	From \eqref{U0FAR}, we have 
	\begin{equation*}
		\left(\frac{2+\overline{U}'}{y^2}\right)^{1/2}=\frac{30(-\overline{U}')^{5/2}}{2(\overline{U}')^2-2\overline{U}'+3}.
	\end{equation*}
	By this and \eqref{U''_TEMP}, for $y > 0$ such that   $-2\leq \overline{U}' (y) \leq -4/5$, we have
	\begin{equation*}
		\begin{split}
			\frac{(1+y^2)^2}{y}\overline{U}''(y)&=\frac{2(1+y^2)^2}{y}(2+\overline{U}')^{1/2}(-\overline{U}')^{7/2}
			\\
			&\geq 2\left(\frac{2+\overline{U}'}{y^2}\right)^{1/2}(-\overline{U}')^{7/2}
			\\
			&=\frac{60(-\overline{U}')^{6}}{2(\overline{U}')^2-2\overline{U}'+3}\geq \frac{60\cdot (4/5)^6}{2\cdot (4/5)^2+2\cdot (4/5)+3}> \frac{12}{5}. 
		\end{split}
	\end{equation*}
On the other hand, it is easy to see in the region $\{ y>0 : -4/5\leq \overline{U}'(y) \leq 0\}$ that 
	\begin{equation*}
		2+\overline{U}'-\frac{6y^2}{5(1+y^2)}\geq \frac{6}{5}-\frac{6y^2}{5(1+y^2)}\geq 0.
	\end{equation*}
Therefore, \eqref{num-4-0} holds for all $y>0$ and it is trivial for $y=0$. By odd symmetry, \eqref{num-4-0} holds for all $y\in\mathbb{R}$. 
\end{proof}

Now, we prove \eqref{num_2}--\eqref{num_3}.
\begin{lemma}\label{5.2}
	It holds that 
	\begin{equation*}\label{num_2-2}
		1+\overline{U}'+\frac{2}{1+y^2}\left(\frac{5}{2}+\frac{\overline 	U}{y}\right)\geq \frac{y^2}{5(1+y^2)} 
	\end{equation*}
	and 
	\begin{equation*}\label{num_3-3}
		\frac{7}{2}+2\overline{U}'+\frac{1}{1+y^2}\left(\frac{5}{2}+\frac{\overline{U}}{y}\right) \geq \frac{19y^2}{10(1+y^2)}
	\end{equation*}
	 for all $y \in \RR$. 
\end{lemma}

\begin{proof}
	Thanks to $\overline{U}$ being an odd function, we can restrict our consideration to the case $y\geq 0$. 
	By \eqref{num_4} and $\overline{U}(0)=0$, we obtain 
	\begin{equation*}
		\overline{U}(y)=\int^y_0 \overline{U}'(y')\,dy'\geq \int^y_0\frac{6y'^2}{5(1+y'^2)}\,dy'-2y= - \frac{4}{5} y-\frac{6}{5} \tan^{-1}(y).
	\end{equation*} 
	Hence by \eqref{num-4-0}, 
	\begin{equation*}
		\begin{split}
			1+\overline{U}'+\frac{2}{1+y^2}\left(\frac{5}{2}+\frac{\overline{U}}{y}\right)&\geq \frac{1}{1+y^2}\left(\frac{1}{5}y^2+\frac{12}{5}\left(1-\frac{\tan^{-1}(y)}{y}\right)\right) 
			\geq \frac{y^2}{5(1+y^2)}.
		\end{split}
	\end{equation*}
%
	Similarly, 	
	\begin{equation*}
		\begin{split}
			\frac{7}{2}+2\overline{U}'+\frac{1}{1+y^2}\left(\frac{5}{2}+\frac{\overline{U}}{y}\right)&\geq \frac{1}{1+y^2}\left(\frac{19}{10}y^2+\frac{6}{5}\left(1-\frac{\tan^{-1}(y)}{y}\right)\right) 
			\geq \frac{19y^2}{10(1+y^2)}.
		\end{split}
	\end{equation*}
	This completes the proof. 
\end{proof}

Next we show \eqref{num_6}.
\begin{lemma}\label{5.3}
	It holds that
	\begin{equation}\label{num_6-lem}
		\begin{split}
			\lambda|\overline{U}''(y)|\frac{y^2+1}{y^2}\int^{|y|}_0{\frac{y'^2}{y'^2+1 }\,dy'}
			\leq \delta \left(1+\overline{U}'+\frac{2}{y^2+1}\left(\frac{5}{2}+\frac{\overline{U}}{y}\right)-\frac{y^2}{500(1+y^2)}\right)
		\end{split}
	\end{equation} 
	for some $\lambda>1$ and $\delta\in(0,1)$. 
\end{lemma}
\begin{proof}
	First, we claim that there exists $\lambda'>1$ such that 
	\begin{equation}\label{5.3_claim}
		\lambda'|\overline{U}''(y)|\frac{y^2+1}{y^2}\int^{|y|}_0{\frac{y'^2}{1+y'^2 }\,dy'}
		\leq  1+\overline{U}'+\frac{2}{1+y^2}\left(\frac{5}{2}+\frac{\overline{U}}{y}\right)-\frac{y^2}{500(1+y^2)} 
	\end{equation}
	for all $y\geq 0$. If the claim holds, then thanks to odd symmetry of $\overline{U}$, we get \eqref{5.3_claim} for all $y\in\mathbb{R}$. Choosing $\lambda>1$, $\delta\in(0,1)$ such that $\lambda'=\lambda/\delta$, we can conclude \eqref{num_6-lem}.

	Now, we prove the claim, \eqref{5.3_claim}. 
	In view of  \eqref{Wbar_neg}, we see that $\RR_{y\ge0} = \{ y\geq 0 : -2\le \UU'(y) \le 0 \}$. 
	Let us split the domain $\RR_{\ge 0}$ into $\{ y\geq 0: -0.5\leq \overline{U}'(y)\leq 0\}$ and $\{ y\geq 0 : -2\leq \overline{U}'(y)\leq -0.5\}$. 
	
	First, we consider the case $\{ y\geq 0: -0.5\leq \overline{U}'(y)\leq 0\}$. Since $\textstyle \int^{y}_0\frac{y'^2}{y'^2+1}\,dy'\leq \int^{y}_0 1 \,dy'=y$, it suffices to show that  
	\begin{equation}\label{far_num_6}
		\lambda' |\overline{U}''(y)|\frac{y^2+1}{|y|}-\overline{U}'(y)\leq \frac{\frac{499}{500}y^2+1}{1+y^2}+\frac{2}{1+y^2}\left(\frac{\overline{U}}{y}+\frac{5}{2}\right). 
	\end{equation} 
	Using \eqref{U''_TEMP}, one can rewrite \eqref{far_num_6} as 
	\begin{equation*}
		\begin{split}
			&2\lambda' (2+\overline{U}')^{1/2}(-y^{2/5}\overline{U}')^{5/2}(-\overline{U}')+2\lambda' \left(\frac{2+\overline{U}'}{y^2}\right)^{1/2}(-\overline{U}')^{7/2}-\overline{U}' \leq \frac{\frac{499}{500}y^2+1}{1+y^2}+\frac{2}{1+y^2}\left(\frac{\overline{U}}{y}+\frac{5}{2}\right),
		\end{split}
	\end{equation*}
	which is also equivalent to 
	\begin{equation*}
		\begin{split}
			\frac{\lambda'}{15}(2+\overline{U}')(-\overline{U}')(2(\overline{U}')^2-2\overline{U}'+3)+\frac{60\lambda' (-\overline{U}')^6}{2(\overline{U}')^2-2\overline{U}'+3}-\overline{U}'\leq \frac{\frac{499}{500}y^2+1}{1+y^2}+\frac{2}{1+y^2}\left(\frac{2+\overline{U}'}{16(-y^{2/5}\overline{U}')^5}\right)^{1/2},
		\end{split}
	\end{equation*}
	thanks to \eqref{3.4}, \eqref{U0FAR} and 
	\begin{equation}\label{U0FAR3'}
		\left(\frac{2+\overline{U}'}{y^2}\right)^{1/2}=\frac{30(-\overline{U}')^{5/2}}{2(\overline{U}')^2-2\overline{U}'+3},
	\end{equation}
	from \eqref{U0FAR}.
	With $\lambda'=1.01$, it holds that
	\begin{equation*}
		\begin{split}
			f_1(X):=\frac{\lambda'}{15}(2+X)(-X)(2X^2-2X+3)+\frac{60\lambda' (-X)^6}{2X^2-2X+3}-X&\leq f_1(-0.5)
			\leq \frac{499}{500} 
		\end{split}
	\end{equation*}
	for $-0.5\leq X\leq 0$,
	where we used the fact that for $-0.5\leq X\leq 0$, 
	\begin{equation*}
		\begin{split}
			f_1'(X)&=-\frac{2\lambda'}{15}(4X^3+3X^2-X+3)-\frac{60\lambda'(-X)^5}{2X^2-2X+3}\left(6-\frac{(-X)(2-4X)}{2X^2-2X+3}\right)-1
			\leq 0
		\end{split}
	\end{equation*}
	coming from $4X^3+3X^2-X+3\geq 0$ and $\textstyle 0\leq \frac{(-X)(2-4X)}{2X^2-2X+3}\leq \frac{(0.5)\cdot 4}{3}=\frac{2}{3}$ for $-0.5\leq X\leq 0$. This proves \eqref{5.3_claim} for $\{ y\geq 0: -0.5\leq \overline{U}'(y)\leq 0\}$.
	
	Next, we consider the case $\{y\geq 0: -2\leq \overline{U}'(y)\leq -0.5\}$. Multiplying $y^{-2}$ to both sides of \eqref{num_6-lem}, we obtain from
	$\textstyle \int^{|y|}_0\frac{y'^2}{1+y'^2}\,dy'\leq \int^{|y|}_0 y'^2\,dy'=\frac{|y|^3}{3}$, \eqref{U''_TEMP}, \eqref{U0FAR} and \eqref{U0FAR3'}  that 
	\begin{equation}\label{lhs_num6}
		\begin{split}
			y^{-2}LHS&=\lambda'\frac{|\overline{U}''(y)|}{y^4}\int^{|y|}_0{\frac{y'^2}{1+y'^2 }\,dy'}+\lambda'\frac{|\overline{U}''(y)|}{y^2}\int^{|y|}_0{\frac{y'^2}{1+y'^2 }\,dy'}
			\\
			&\leq \frac{\lambda'}{3}\frac{|\overline{U}''(y)|}{|y|}+\lambda'\frac{|\overline{U}''(y)|}{|y|}\frac{y^2}{3}
			\\
			&= \frac{2\lambda'}{3}\left(\frac{2+\overline{U}'}{y^2}\right)^{1/2}(-\overline{U}')^{7/2} + \frac{2\lambda'}{3}\left(\frac{2+\overline{U}'}{y^2}\right)^{1/2}(-y^{2/5}\overline{U}')^{5}(-\overline{U}')^{-3/2}
			\\
			&= \frac{20\lambda' (-\overline{U}')^{6}}{2(\overline{U}')^2-2\overline{U}'+3}+ \frac{\lambda'}{45}(-\overline{U}')(2+\overline{U}')(2(\overline{U}')^2-2\overline{U}'+3).
		\end{split}
	\end{equation}
	Also, using \eqref{U0FAR3'} and $2+\overline{U}/y\geq 0$ from \eqref{Wbar_neg}, we have  
	\begin{equation*}\label{rhs_num6}
		\begin{split}
			y^{-2}RHS&=\frac{2+\overline{U}'}{y^2}+\frac{2}{y^2(1+y^2)}\left(2+\frac{\overline{U}}{y}\right)-\frac{501}{500}\frac{1}{1+y^2}
			\geq \frac{900(-\overline{U}')^{5}}{(2(\overline{U}')^2-2\overline{U}'+3)^2}-\frac{501}{500}.
		\end{split}
	\end{equation*}
	Set 
	\begin{equation*}
		\begin{split}
			f_2(X)&:=\frac{20\lambda' (-X)^{6}}{2X^2-2X+3}+ \frac{\lambda'}{45}(-X)(X+2)(2X^2-2X+3)- \frac{900(-X)^{5}}{(2X^2-2X+3)^2}+\frac{501}{500}
			\\
			&=\left(\frac{20(-X)^5}{(2X^2-2X+3)^2}+\frac{X+2}{45}\right)\left(\lambda'(-X)(2X^2-2X+3)-45\right)+ 2+X +\frac{501}{500}. 
		\end{split}
	\end{equation*}
	Then it is enough to show that $f_2(X)\leq 0$ for $-2\leq X\leq -0.5$, which will imply  \eqref{num_6-lem} holds  for $y$ such that $-2\leq \overline{U}' (y) \leq -0.5$.

	First, for $X \in [-0.8, -0.5]$, it holds that
	\begin{equation*}\label{-1-half}
		\begin{split}
			f_2(X)&\leq \left(\frac{20\cdot 0.5^5}{(2\cdot 0.5^2+2\cdot 0.5+3)^2}+\frac{2-0.5}{45}\right)\left(0.8 (2\cdot 0.8^2+2\cdot 0.8+3)\lambda'-45\right)+(2-0.5)+\frac{501}{500}
			\\
			&\leq 0.
		\end{split}
	\end{equation*}
	Here, we used the fact that $\textstyle \frac{20(-X)^5}{(2X^2-2X+3)^2}+\frac{X+2}{45}$ is a positive, decreasing function
	and $\lambda'(-X)(2X^2-2X+3)-45$ is a negative, decreasing function, 
	 where $\lambda'=1.01$.

	Similarly, for $X\in[-2, -0.8]$ and $\lambda'=1.01$, it is straightforward to check that 
	\begin{equation}\label{-2-1}
		f_2(X)\leq \left(\frac{20\cdot (0.8)^5}{(2\cdot (0.8)^2+2\cdot 0.8+3)^2}+\frac{2-0.8}{45}\right)\left(30\lambda'-45\right)+(2-0.8)+\frac{501}{500}\leq 0.
	\end{equation}
	In view of \eqref{lhs_num6}--\eqref{-2-1},	we see that  \eqref{5.3_claim} holds for $ y\geq 0$ such that $-2\leq \overline{U}'(y)\leq -0.5$.  
	Combining all, we conclude that \eqref{num_6-lem} holds for all $y\in\mathbb{R}$.
	This completes the proof. 
\end{proof}

Next, we show \eqref{num_1}. 
\begin{lemma}\label{5.4}
	There exist $m_0>0$ and $\lambda>1$ such that 
	\begin{equation}\label{num_1-lem}
		\begin{split}
			&\lambda(y^{2/5}+1)|\overline{U}''(y)|\int^{|y|}_0\frac{dy'}{1+(y')^{2/5}}
			\\
			&\qquad \leq \frac{10}{13(1+y^{2/5})}+\overline{U}'-\frac{2y^{2/5}}{5(1+y^{2/5})}\left(\frac{\overline{U}}{y}+\frac{6}{13y}\int^y_0\frac{dy'}{1+(y')^{2/5}}\right) \quad \text{ for } |y|\geq m_0. 
		\end{split}
	\end{equation}
	Moreover, one can verify that \eqref{num_1-lem} holds with $m_0= 2\cdot 10^6$ and $\lambda=1.0001$.
\end{lemma}

\begin{proof}
	Since $\overline{U}$ has odd symmetry, without loss of generality, we restrict ourselves to the case $y\geq 0$.
	
	Multiplying the left and right sides of \eqref{num_1-lem} by $y^{2/5}$, and using $\textstyle \int^y_0 \frac{dy'}{1+(y')^{2/5}}=\frac{5}{3}y^{3/5}-5(y^{1/5}-\tan^{-1}(y^{1/5}))$, we have 
		\begin{equation*}
		\begin{split}
			y^{2/5}\cdot LHS &= \lambda y^{2/5}(y^{2/5}+1)|\overline{U}''(y)|\int^{|y|}_0 \frac{dy'}{1+(y')^{2/5}} 
			\\
			&= \frac{5\lambda}{3}y^{2/5}(y^{2/5}+1)|\overline{U}''(y)|(y^{3/5}-3(y^{1/5}-\arctan(y^{1/5})))
		\end{split}
	\end{equation*}
	and 
		\begin{equation*}
		\begin{split}
			y^{2/5}\cdot RHS &= \frac{10y^{2/5}}{13(1+y^{2/5})}+y^{2/5}\overline{U}'-\frac{2y^{2/5}}{5(1+y^{2/5})}\left(\frac{\overline{U}}{y^{3/5}}+\frac{6}{13y^{3/5}} \int^y_0 \frac{dy'}{1+(y')^{2/5}}\right)
			\\
			&=\frac{10y^{2/5}}{13(1+y^{2/5})} +y^{2/5}\overline{U}' - \frac{2y^{2/5}}{5(1+y^{2/5})}\left(\frac{\overline{U}}{y^{3/5}}+\frac{10}{13y^{3/5}}\left(y^{3/5}-3(y^{1/5}-\arctan(y^{1/5}))\right)\right). 
		\end{split}
	\end{equation*}
	Then using \eqref{U_taylor_far} with $\beta=1$ in Proposition~\ref{Profile-construct}, we can check that
	\[ \lim_{y\to \infty} y^{2/5}\cdot(RHS- LHS ) =  \frac{6}{13} - \frac{1}{3 \cdot 50^{1/5}} - \frac{2 \lambda }{3 \cdot 50^{1/5}} > \frac{6}{13} - \frac{ \lambda }{50^{1/5}}  >0\]
	for some $\lambda>1$. In fact, the last inequality holds thanks to the fact that $\textstyle {6/13} - {50^{-1/5}} \approx 0.04233 >0$. 
	This proves that \eqref{num_1-lem} holds for some $m_0>0$ and $\lambda>1$. 
	
	On the other hand, one can verify \eqref{num_1-lem} with the specific values $m_0= 2\cdot 10^6$ and $\lambda=1.0001$. 
	Using \eqref{U''_TEMP}, 
	\begin{equation}\label{LHS-num_1}
		\begin{split}
			y^{2/5}\cdot LHS &= \lambda y^{2/5}(y^{2/5}+1)|\overline{U}''(y)|\int^{|y|}_0 \frac{dy'}{1+(y')^{2/5}} 
			\\
			&= \frac{5\lambda}{3}y^{2/5}(y^{2/5}+1)|\overline{U}''(y)|(y^{3/5}-3(y^{1/5}-\arctan(y^{1/5})))
			\\
			&=\frac{10\lambda}{3}(2+\overline{U}')^{1/2}(-\overline{U}')^{7/2}(y^{7/5}+y)-L
			\\
			&=\frac{10\lambda}{3}(2+\overline{U}')^{1/2}\left((-y^{2/5}\overline{U}')^{7/2}+(-y^{2/5}\overline{U}')^{5/2}(-\overline{U}')\right)-L
			\\
			&=\frac{10\lambda}{3(30)^{7/5}}(2+\overline{U}')^{6/5}(2(\overline{U}')^2-2\overline{U}'+3)^{7/5}+\frac{\lambda}{9}(2+\overline{U}')(2(\overline{U}')^2-2\overline{U}'+3)(-\overline{U}')-L,
		\end{split}
	\end{equation}
	where 
	\begin{equation*}
		L(y,\overline{U}'):=10\lambda y^{2/5}(y^{2/5}+1)(y^{1/5}-\arctan (y^{1/5}))(2+\overline{U}')^{1/2}(-\overline{U}')^{7/2}\geq 0.
	\end{equation*}
	
	Similarly, multiplying the right hand side of \eqref{num_1-lem} by $y^{2/5}$, and using \eqref{U0FAR} with $\beta=1$, we have 
	\begin{equation}\label{RHS-num_1}
		\begin{split}
			y^{2/5}\cdot RHS &= \frac{10y^{2/5}}{13(1+y^{2/5})}+y^{2/5}\overline{U}'-\frac{2y^{2/5}}{5(1+y^{2/5})}\left(\frac{\overline{U}}{y^{3/5}}+\frac{6}{13y^{3/5}} \int^y_0 \frac{dy'}{1+(y')^{2/5}}\right)
			\\
			&=\frac{10y^{2/5}}{13(1+y^{2/5})}-(30)^{-2/5}(2+\overline{U}')^{1/5}(2(\overline{U}')^2-2\overline{U}'+3)^{2/5}
			\\
			&\quad-\frac{2y^{2/5}}{5(1+y^{2/5})}\left(\frac{\overline{U}}{y^{3/5}}+\frac{10}{13y^{3/5}}\left(y^{3/5}-3(y^{1/5}-\arctan(y^{1/5}))\right)\right)
			\\
			&= \frac{6y^{2/5}}{13(1+y^{2/5})}-(30)^{-2/5}(2+\overline{U}')^{1/5}(2(\overline{U}')^2-2\overline{U}'+3)^{2/5}-\frac{2y^{2/5}}{5(1+y^{2/5})}\frac{\overline{U}}{y^{3/5}}  +R,
		\end{split}
	\end{equation}
	where
	\begin{equation*}
		R(y,\overline{U}'):=\frac{12y^{2/5}}{13(1+y^{2/5})}\frac{y^{1/5}-\arctan(y^{1/5})}{y^{3/5}}\geq 0.
	\end{equation*}  
	Then, by subtracting them and substituting \eqref{Ubar35} with $\beta=1$, one can check that for $y\geq 2\cdot 10^6$,
	\begin{equation*}\label{far_num1_1}
		\begin{split}
			y^{2/5}(LHS-RHS)&\leq g(\overline{U}'),
		\end{split}
	\end{equation*} 
	where 
	\begin{equation*}
		\begin{split}
			g(X)&:=\frac{10\lambda}{3(30)^{7/5}}(2+X)^{6/5}(2X^2-2X+3)^{7/5}+\frac{\lambda}{9}(2+X)(2X^2-2X+3)(-X)
			\\
			&\quad -\frac{6\cdot (2\cdot 10^6)^{2/5}}{13(1+(2\cdot 10^6)^{2/5})}+(30)^{-2/5}(2+X)^{1/5}(2X^2-2X+3)^{2/5}
			\\
			&\quad -\frac{(30)^{3/5}\cdot (2\cdot 10^6)^{2/5}}{15(1+(2\cdot 10^6)^{2/5})}\frac{(2+X)^{1/5}(1-X)}{(2X^2-2X+3)^{3/5}}.
		\end{split}
	\end{equation*}
	Here, we used the facts that $\textstyle \frac{y^{2/5}}{1+y^{2/5}}$ is increasing for $y\geq 0$ and $\textstyle \frac{(2+X)^{1/5}(1-X)}{(2X^2-2X+3)^{3/5}}\geq 0$ for $-2\leq X\leq 0$. Note by \eqref{Wbar_neg} that $-2\leq \overline{U}'(y) \leq 0$ for all $y\in\RR$. 
	
	First, for $\lambda=1.0001$ and $k=-1/400$, we show that 
	\begin{equation}\label{far_num1_3}
		\begin{split}
			g(X)\leq 0 \quad\text{for } X\in \Omega_1:=[k,0].
		\end{split}
	\end{equation}
	To check this, it suffices to show that $g'(X)\leq 0$ for all $X\in\Omega_1$ and $g\left(k\right)<0$. 
	By a straightforward calculation, we have  
	\begin{equation*}
		\begin{split}
			g'(X) &= \frac{4\lambda}{(30)^{7/5}}(2+X)^{1/5}(2X^2-2X+3)^{7/5} + \frac{14 \lambda}{3 (30)^{7/5}}(2+X)^{6/5} ( 2X^2 - 2X + 3)^{2/5} (4X - 2) \\
			& + \frac{\lambda}{9} ( 2X^2 - 2X + 3 )(-X) + \frac{\lambda}{9} (2+X)(4X - 2)(-X)- \frac{\lambda}{9} (2+X) \left( 2X^2 - 2X + 3 \right)  \\
			& + \frac{(30)^{-2/5}}{5}(2+X)^{-4/5}(2X^2-2X+3)^{2/5}+\frac{2\cdot (30)^{-2/5}}{5}(2+X)^{1/5}(2X^2-2X+3)^{-3/5}(4X-2) \\
			& -\frac{(30)^{3/5}\cdot(2\cdot 10^6)^{2/5}}{75(1+(2\cdot10^6)^{2/5})}(2+X)^{-4/5}(1-X)(2X^2-2X+3)^{-3/5} \\
			& + \frac{(30)^{3/5}\cdot(2\cdot 10^6)^{2/5}}{15(1+(2\cdot10^6)^{2/5})} (2+X)^{1/5}(2X^2-2X+3)^{-3/5} \\
			& +\frac{3\cdot (30)^{3/5}\cdot(2\cdot 10^6)^{2/5}}{75(1+(2\cdot10^6)^{2/5})}(2+X)^{1/5}(1-X)(2X^2-2X+3)^{-8/5}(4X-2). \\
		\end{split}
	\end{equation*}
	Note that for $\lambda=1.0001$, it holds for $X\in\Omega_{1}$ that
	\begin{equation*}
		\begin{split}
			g'(X)&\leq \frac{4\lambda}{(30)^{7/5}}\cdot 2^{1/5}\left(2k^2-2k+3\right)^{7/5} -\frac{28\lambda}{3(30)^{7/5}}\cdot \left(2+k\right)^{6/5}\cdot 3^{2/5} \\
			&-\frac{\lambda}{9}\left(2k^2-2k+3\right)k+0-\frac{\lambda}{3}\left(2+k\right)  \\
			&+\frac{(30)^{-2/5}}{5}\left(2+k\right)^{-4/5}\left(2k^2-2k+3\right)^{2/5}-\frac{4\cdot (30)^{-2/5}}{5} \left(2+k\right)^{1/5}  \left(2k^2-2k+3\right)^{-3/5}  \\
			&-\frac{(30)^{3/5}\cdot (2\cdot 10^6)^{2/5}}{75(1+(2\cdot 10^6)^{2/5})}\cdot 2^{-4/5} \left(2k^2-2k+3\right)^{-3/5} 
			+\frac{(30)^{3/5}\cdot (2\cdot 10^6)^{2/5}}{15(1+(2\cdot 10^6)^{2/5})} \cdot 2^{1/5} \cdot 3^{-3/5} \\
			&-
			\frac{6\cdot(30)^{3/5}(2\cdot 10^6)^{2/5}}{75(1+(2\cdot 10^6)^{2/5})}\cdot \left(2+k\right)^{1/5} \cdot\left(2k^2-2k+3\right)^{-8/5} \\
			&\leq 0.
		\end{split}
	\end{equation*}
	Therefore,  \eqref{far_num1_3} holds. 
	
	Next we deal with the case $\textstyle \{y\in\mathbb{R}\; : \; \overline{U}'(y)\in \Omega_2:=  [-2,-\frac{1}{400}]\}$.
	Using \eqref{3.4}, \eqref{LHS-num_1} and \eqref{RHS-num_1}, we find 
	\begin{equation*}
		\begin{split}
			y^{2/5}(LHS-RHS)&\leq h(\overline{U}'),
		\end{split}
	\end{equation*}
	where
	\begin{equation*}
		\begin{split}
			h(X)&:=\frac{10\lambda}{3(30)^{7/5}}(2+X)^{6/5}(2X^2-2X+3)^{7/5}+\frac{\lambda}{9}(2+X)(2X^2-2X+3)(-X)
			\\
			&\quad-\frac{6\cdot (2\cdot 10^6)^{2/5}}{13(1+(2\cdot 10^6)^{2/5})} +(30)^{-2/5}(2+X)^{1/5}(2X^2-2X+3)^{2/5}
			\\
			&\quad -\frac{(2\cdot 10^6)^{4/5}}{1+(2\cdot 10^6)^{2/5}}+\frac{1}{10\cdot (2\cdot 10^6)^{1/5}(1+(2\cdot 10^6)^{2/5})}\frac{(2+X)^{1/2}}{(-X)^{5/2}}.
		\end{split}
	\end{equation*}
	Note that for $\lambda=1.0001$, it holds for $\textstyle X\in \Omega_2:=  [-2,-\frac{1}{400}]$ that
	\begin{equation*}
		\begin{split}
			h(X)&\leq \frac{10\lambda}{3(30)^{7/5}}\left(2-\frac{1}{400}\right)^{6/5}\cdot 15^{7/5}+\frac{10\lambda}{3}\left(2-\frac{1}{400}\right)
			\\
			&\quad-\frac{6\cdot (2\cdot 10^6)^{2/5}}{13(1+(2\cdot 10^6)^{2/5})} +(30)^{-2/5}\left(2-\frac{1}{400}\right)^{1/5}\cdot15^{2/5}
			\\
			&\quad -\frac{(2\cdot 10^6)^{4/5}}{1+(2\cdot 10^6)^{2/5}}+\frac{1}{10\cdot (2\cdot 10^6)^{1/5}(1+(2\cdot 10^6)^{2/5})}\frac{(2-\frac{1}{400})^{1/2}}{\left(\frac{1}{400}\right)^{5/2}}
			\\
			&\leq 0,
		\end{split}
	\end{equation*}
	since $2X^2-2X+3\leq 15$ and $2+X\leq 2-400^{-1}$ for $400^{-1}\leq -X\leq 2$. 
	Combining all, we conclude that \eqref{num_1-lem} holds for all $y\in\mathbb{R}$ with $\lambda=1.0001$ and $m_0=2\cdot 10^6$. 
\end{proof}

\begin{remark}
	In Lemma~\ref{5.4}, we prove that there exist  $m_0>0$ and $\lambda>1$ such that \eqref{num_1-lem} holds for $|y| \ge m_0$, and also verify that it holds with $m_0=2\cdot 10^6$ and $\lambda = 1.0001$. In fact, we can numerically verify that 
	\eqref{num_1-lem} holds with $m_0 = 93$ and $\lambda =1.0001$, see Figure~\ref{LHS_RHS_w_and_U_bar_scale_vs_V_bar}. 
	However, for mathematical completeness, we carry out our analysis using Lemma~\ref{5.4} with $m_0=2\cdot 10^6$ and $\lambda = 1.0001$. 
%
%
%
%
%
%
%
%
%
%
%
%
%
%
\end{remark}

\subsection{Numerical verification}
We first numerically construct the profile $\UU(y)$, the solution to the ODE \eqref{Ueq-int} with \eqref{U0}. 
	As we discussed in the proof of Proposition~\ref{Profile-construct}, we consider the initial value problem $W'=V(W), W(0)=0 \; (y \geq0)$, where $V=V(W)$ is implicitly defined by \eqref{CW2}.
To numerically solve this, 
we utilize the $8$th order explicit Runge-Kutta method. 
 Then, using the relation $\UU(y) = W(y) - 5y/2$ and the odd symmetry of $\UU$, we obtain a numerical plot of $\UU$ that displays its asymptotic behavior for $|y| \ll 1$ and $|y| \gg 1$, respectively (see Figure~\ref{profile_asymp}).
	
	\begin{figure}[H]
	\centering
	\begin{subfigure}[b]{0.3\columnwidth}
		\centering
		\includegraphics[width=\columnwidth]{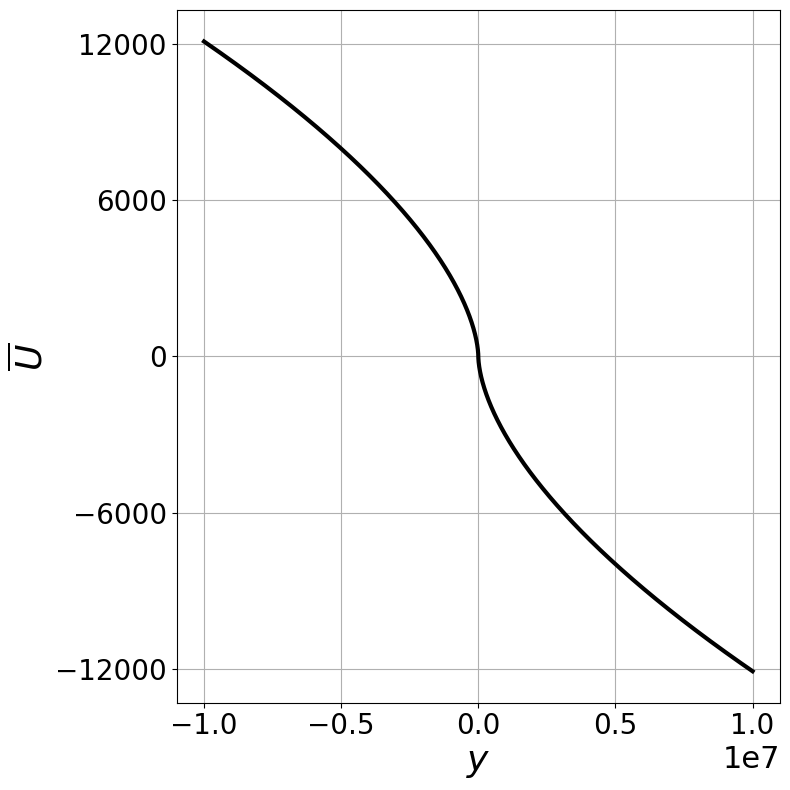} 
		\label{overline_U_profile}
	\end{subfigure}
	\hspace{-0.15cm}
	\begin{subfigure}[b]{0.3\columnwidth}
		\centering
		\includegraphics[width=\columnwidth]{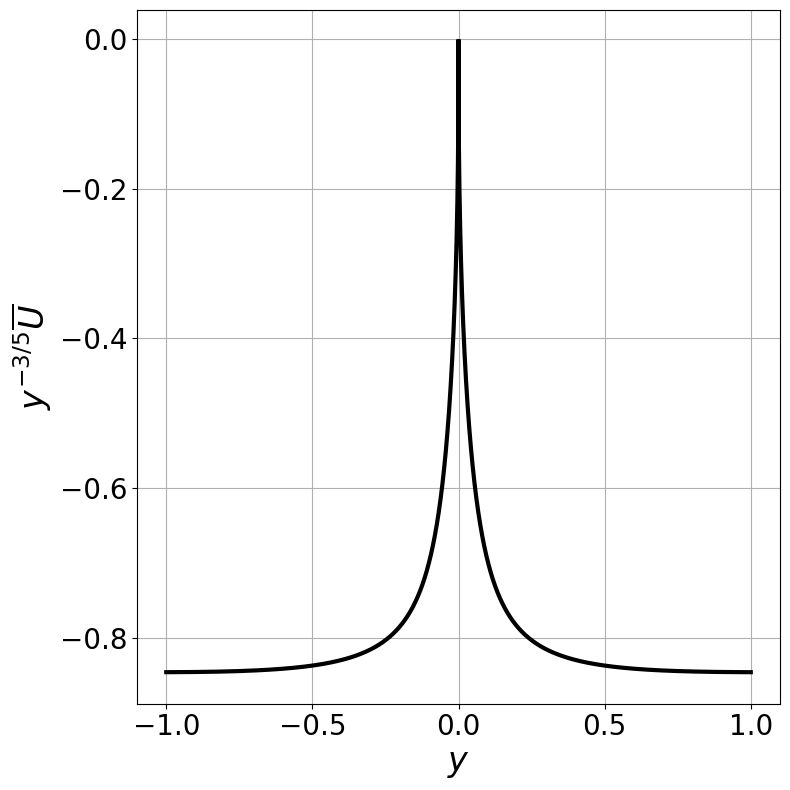} 
		\label{asymptotic_near_0}
	\end{subfigure}
	\hspace{-0.15cm}
	\begin{subfigure}[b]{0.3\columnwidth}
		\centering
		\includegraphics[width=\columnwidth]{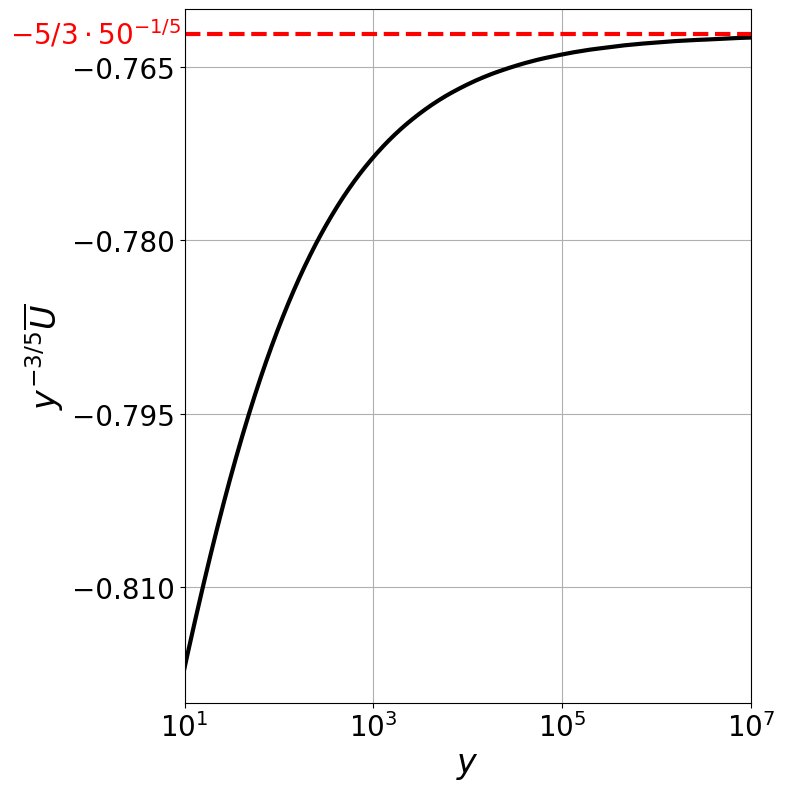} 
		\label{asymptotic_infty}
	\end{subfigure}
	\caption{Numerical profile of $\UU$ (left) and the asymptotic behavior of $\UU$ for $|y|\ll1$ and $|y| \gg 1$, respectively (center and right).}
	\label{profile_asymp}
\end{figure}

	With $\UU$ numerically constructed profile, we can numerically verify that \eqref{num_1-lem} in Lemma~\ref{5.4} holds for all $|y|\geq 93$ with $\lambda = 1.0001$.   
Specifically, we multiply both side of \eqref{num_1-lem} by $10^3 y^{2/5}$ and set $\lambda=1.0001$ to have 
	 \begin{equation*}
	 \begin{split}
	 L_{w}(y) & =1.0001 \cdot 10^{3}y^{2/5} (y^{2/5}+1)|\overline{U}''(y)|\int^{|y|}_0\frac{dy'}{1+(y')^{2/5}},
	 \\
	 R_{w}(y) & =10^{3}y^{2/5}\left(\frac{10}{13(1+y^{2/5})}+\overline{U}'-\frac{2y^{2/5}}{5(1+y^{2/5})}\left(\frac{\overline{U}}{y}+\frac{6}{13y}\int^y_0\frac{dy'}{1+(y')^{2/5}}\right)\right).
	 \end{split}
	 \end{equation*} 
	 Then, by a direct numerical computation, we obtain the graphs of $L_{w}$ and $R_{w}$, illustrated in Figure~\ref{LHS_RHS_w_and_U_bar_scale_vs_V_bar}. The graphs display that $L_w(y) < R_w(y)$ for $y\in [93, 10^4]$. This verifies numerically that \eqref{num_1-lem} holds for $y\in [93, 10^4]$ and $\lambda =1.0001$. 

Next we shall compare the profiles of $\UU$ solving \eqref{Ueq}, and the self-similar blow-up profile of the Burgers solution. It is known, see \cite{CSW}, \cite{EF} for instance, that the Burgers equation admits the stable self-similar solution $\overline{V}$ satisfying 
\begin{equation*}	
\left(1+\frac{1}{2}\overline{V}'\right)\overline{V}'+\left(\frac{1}{2}\overline{V}+\frac{3}{2}y\right)\overline{V}''=0.
\end{equation*} 
In fact, we rescaled it from the original one by simply multiplying by $2$ so that $\overline{V}'(0) =-2$, aligning it with the blow-up profile $\UU$ which satisfies $\UU'(0) =-2$.  
Note that this rescaling does not affect the asymptotic behavior for $|y|\gg 1$. 
%
The plots of $\UU$ and $\overline{V}$ are shown in Figure \ref{LHS_RHS_w_and_U_bar_scale_vs_V_bar}, which demonstrate that $\UU$ grows faster than $\overline{V}$ for $|y|\gg1$.
This supports our result in Proposition~\ref{Profile-construct} that $\UU$ grows faster than $\overline{V}$ of the Burgers solution does, which eventually yields that the singularities in the CH equation we have constructed are milder than those of the Burgers equation.
%

\begin{figure}[H]
	\centering
	\begin{subfigure}[b]{0.3\columnwidth}
		\centering
		\includegraphics[scale=0.275]{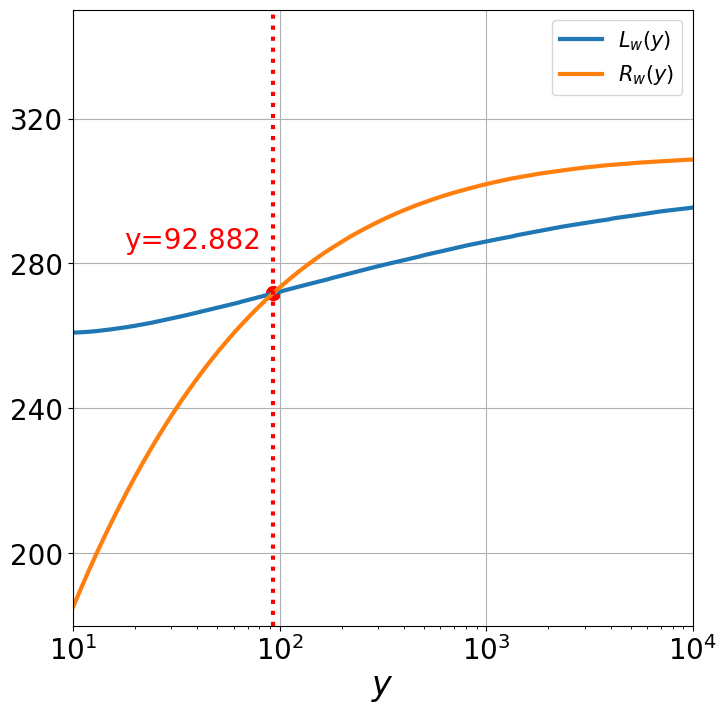}
		\label{LHS_RHS_w}
	\end{subfigure}
	\hspace{1.5cm}
	\begin{subfigure}[b]{0.3\columnwidth}
		\centering
		\includegraphics[scale=0.245]{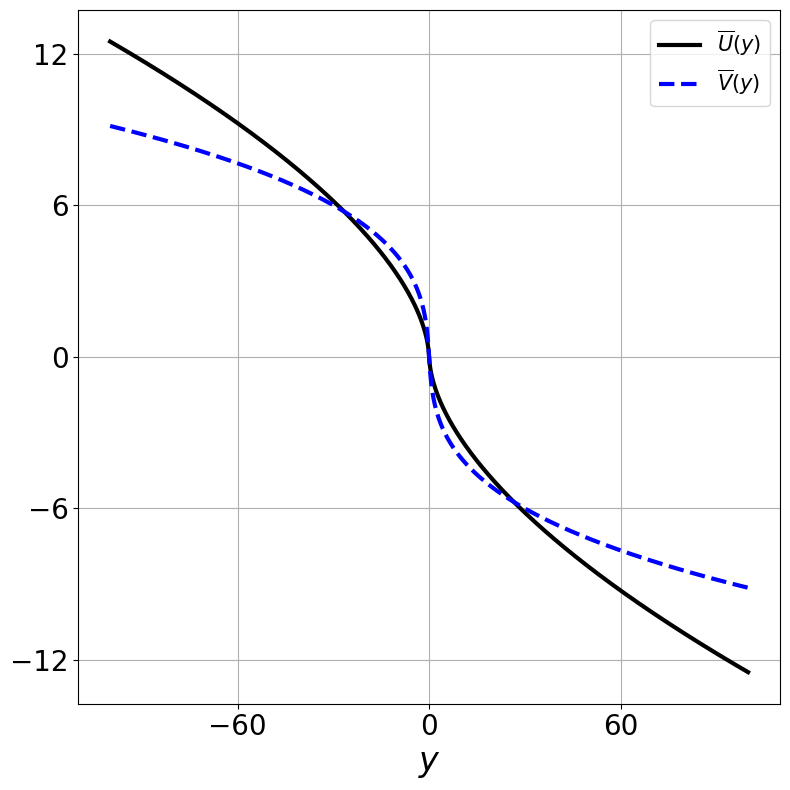}
		\label{U_bar_scale_vs_V_bar}
	\end{subfigure}
	\caption{Numerical plots of $L_{w}$ and $R_{w}$, with $L_{w}-R_{w}$ changing sign at $y\approx92.882$ (left). Profiles of $\UU$ and $\overline{V}$, where $\overline{V}$ is the self-similar profile of the Burgers (right).}
	\label{LHS_RHS_w_and_U_bar_scale_vs_V_bar}
\end{figure}

\subsection{Maximum principles}

We present a maximum principle, a modified form of that developed in \cite{BSV}, for use in our analysis. 
We consider the initial value problem: 
\begin{equation}\label{IVP-f}
	\begin{split} 
		& \partial_s f(y,s)+D(y,s) f(y,s)+U(y,s)\partial_yf(y,s)=F(y,s)+\int_{\mathbb{R}}f(y',s)K(y,s;y')\,dy',  \quad
		s\ge s_0,  y\in \mathbb{R},
		\\
		& f(y,s_0) = f_0(y). 
	\end{split} 
\end{equation}
\begin{lemma}(\cite{BKK})
	\label{max_2}
	Let $f$ be a classical solution to IVP \eqref{IVP-f}. Let $\Omega \subseteq \mathbb{R}$ be any compact set. 
	Suppose that 
	\begin{subequations}
		\begin{align}
			& \|f(\cdot,s)\|_{L^{\infty}(\Omega)}\leq \mo, \label{max_2_1}
			\\
			&  \|f(\cdot,s_0)\|_{L^{\infty}(\mathbb{R})}\leq \mo, \label{max_2_1'}
			\\
			& \int_{\mathbb{R}}|K(y,s;y')|\,dy'\leq \delta D(y,s)\quad \text{for}\quad (y,s)\in \Omega^c\times [s_0,\infty), \label{max_2_2}
			\\
			& \inf_{(y,s)\in \Omega^c\times [s_0,\infty)}D(y,s)\geq \lambda_D >0, \label{max_2_3}
			\\
			& \|F(\cdot,s)\|_{L^{\infty}(\Omega^c)}\leq F_0, \label{max_2_4}
			\\ 
			& \limsup_{|y|\rightarrow \infty}|f(y,s)| <  2\mo \label{max_2_6}
		\end{align}
	\end{subequations} 
	for some $\mo, F_0, \lambda_D >0$ and $\delta\in(0,1)$ satisfying 
	\begin{equation}\label{D-cond} 
	\mo \lambda_D> \frac{F_0}{ 2 (1-\delta ) }. 
	\end{equation}  
	Then, it holds that  $\|f(\cdot,s)\|_{L^{\infty}(\mathbb{R})}\leq 2 \mo$ for all $s\ge s_0$. 
\end{lemma}
For the proof, we refer to that of Lemma 6.6 in the Appendix of \cite{BKK}. 

Next, we consider the initial value problem: 
	\begin{equation}\label{f-ivp-2}
	\begin{split}
		& \partial_s f(y,s) + D(y,s)f(y,s) + U(y,s)\partial_y f(y,s)  =F(y,s),  \quad
		s\ge s_0,  y\in \mathbb{R},
				\\
		& f(y,s_0) = f_0(y).  
        \end{split}
	\end{equation}  
	We present Lemma~\ref{rmk2}, which addresses the spatial decay properties of solutions to the transport-type equation \eqref{f-ivp-2} under suitable assumptions. 
\begin{lemma}(\cite{BKK}) \label{rmk2}
	Let $f$ be a classical solution to IVP \eqref{f-ivp-2}. Assume that 
	$U$, $D$ and $F$ are smooth functions satisfying
	\begin{subequations}
		\begin{align} 
			& \inf_{\{|y| \geq N ,\, s\in[s_0,\infty)\}}  U(y,s) \frac{y}{|y|}  > 0,\label{rmk2_ass0}
			\\
			& \inf_{\{|y| \geq N,\, s\in[s_0,\infty)\}}D(y,s)\geq \lambda_D,\label{rmk2_ass1}
			\\
			& \|F (\cdot, s) \|_{L^{\infty}(|y| \geq N)}\leq F_0e^{-s\lambda_F}\label{rmk2_ass2}
		\end{align}
	\end{subequations}
	for some {$ \lambda_D, \lambda_F, N, F_0  \geq  0$}.   Then it holds that  
	\begin{equation*}
		\begin{array}{l l}
			\limsup_{| y |\rightarrow \infty}|f(y,s)|\leq \limsup_{|y|\rightarrow \infty}{|f(y,s_0)|}e^{-\lambda_D(s-s_0)}+\frac{F_0}{\lambda_D-\lambda_F}e^{-s\lambda_F} & \quad \text{if } \lambda_D>\lambda_F, \\
			\limsup_{| y |\rightarrow \infty}|f( y ,s)|\leq \limsup_{|y|\rightarrow \infty}{|f(y,s_0)|}e^{-\lambda_D(s-s_0)}+\frac{F_0e^{-s_0\lambda_F}}{\lambda_F-\lambda_D}e^{-\lambda_D(s-s_0)} & \quad \text{if } \lambda_F>\lambda_D.
		\end{array}
	\end{equation*}
\end{lemma}
For the proof, we refer to Lemma~$6.7$ in the Appendix of \cite{BKK}.

\section*{Acknowledgements.}
B.K. was supported by Basic Science Research Program through the National Research Foundation of Korea (NRF) funded by the Ministry of science, ICT and future planning (NRF-2020R1A2C1A01009184).

 \end{document}